\documentclass[final,sort,dvipsnames]{siamart190516}

\title{Stationary distributions of continuous-time Markov chains: a review of theory and truncation-based approximations 
\thanks{The first author was supported by a BBSRC PhD Studentship (BB/F017510/1). GBS acknowledges support by the EPSRC Fellowship for Growth EP/M002187/1, a Royal Academy of Engineering Chair in Emerging Technologies, and the EU H2020 FET-OPEN RIA grant 766840 (COSY-BIO). MB acknowledges support from EPSRC grant EP/N014529/1 funding the EPSRC Centre for Mathematics of Precision Healthcare.}}

\headers{Stationary distributions of CTMCs}{Juan Kuntz, Philipp Thomas, Guy-Bart Stan, and Mauricio Barahona}

\author{
  Juan Kuntz\thanks{Department of Mathematics and Department of Bioengineering, Imperial College London, London SW7 2AZ, UK.  \, \textit{Current address:} Department of Statistics,  University of Warwick, Coventry, CV4 7AL,  UK
    (\email{juan.kuntz-nussio@warwick.ac.uk}).}
  \and
  Philipp Thomas\thanks{Department of Mathematics, Imperial College London, London SW7 2AZ, UK (\email{p.thomas@imperial.ac.uk}).}
  \and
  Guy-Bart Stan\thanks{Co-corresponding author. Department of Bioengineering, Imperial College London, London SW7 2AZ, UK
    (\email{g.stan@imperial.ac.uk}).}
  \and
  Mauricio Barahona\thanks{Co-corresponding author. Department of Mathematics, Imperial College London, London SW7 2AZ, UK
    (\email{m.barahona@imperial.ac.uk}).}
}

\usepackage{graphicx}        

\def\juansec#1{\subsubsection{\textbf{#1}}}

\usepackage{xcolor}
\usepackage{color}

\usepackage{colortbl}
\usepackage{amsmath}
\usepackage{amssymb}
\usepackage{mathrsfs}
\usepackage{mathtools}  
\usepackage[font=small,labelfont=bf]{caption}
\usepackage{enumitem}
\usepackage{chngcntr}
\usepackage{bm} 
\usepackage{algorithmic,algorithm}

\usepackage{pifont}
\newcommand{\cmark}{\ding{51}}%
\newcommand{\xmark}{\ding{55}}%

\usepackage{caption}
\usepackage{tabularx}
\newcolumntype{L}{>{\raggedright\arraybackslash}X}

\newtheorem{assumption}[theorem]{Assumption}
\newtheorem{example}[theorem]{Example}
\newtheorem{remark}[theorem]{Remark}

\numberwithin{equation}{section}
\counterwithout{figure}{section}
\counterwithout{table}{section}

\renewcommand{\cal}[1]{\mathcal{#1}}

\renewcommand{\r}{\mathbb{R}}

\newcommand{\n}{\mathbb{N}}

\newcommand{\nn}{\mathbb{N}^n}

\newcommand{\zp}{\mathbb{Z}_+}

\newcommand{\mmag}[1]{\left|#1\right|}

\newcommand{\s}{\mathcal{S}}
\renewcommand{\L}{\mathcal{L}}

\newcommand{\iprod}[2]{\left\langle{#1},{#2}\right\rangle}

\newcommand{\norm}[1]{\left|\left|{#1}\right|\right|}

\newcommand{\wnorm}[1]{\left|\left|{#1}\right|\right|_w}

\newcommand{\Pb}{\mathbb{P}}
\newcommand{\Eb}{\mathbb{E}}

\newcommand{\Pbx}[1]{\mathbb{P}_x\left(#1\right)}
\newcommand{\Ebx}[1]{\mathbb{E}_x\left[#1\right]}
\newcommand{\Pbp}[1]{\mathbb{P}_\pi\left(#1\right)}

\newcommand{\Pbl}[1]{\mathbb{P}_\lambda\left(#1\right)}

\begin{document}

\maketitle
\begin{abstract} 
Computing the stationary distributions of a continuous-time Markov chain (CTMC) involves solving a set of linear equations. In most cases of interest, the number of equations is infinite or too large, and the equations cannot be solved analytically or numerically. Several approximation schemes overcome this issue by truncating the state space to a manageable size. In this review, we first give a comprehensive theoretical account of the stationary distributions and their relation to the long-term behaviour of CTMCs that is readily accessible to non-experts and free of irreducibility assumptions made in standard texts. We then review truncation-based approximation schemes for CTMCs with infinite state spaces paying particular attention to the schemes' convergence and the errors they introduce, and we illustrate their performance with an example of a stochastic reaction network of relevance in biology and chemistry. We conclude by discussing computational trade-offs associated with error control and several open questions.
\end{abstract}




\section{Introduction}\label{sec:intro}
Continuous-time Markov chains (or \emph{continuous-time chains} for short) are  pervasively used throughout science and engineering to model stochastic phenomena evolving in time over a discrete space.
When applied to describe the time evolution of populations of interacting species with indistinguishable members, continuous-time chains are often referred to as \emph{stochastic reaction networks (SRNs)}. SRNs date back to the early days of Markov chain theory (e.g.\ \cite{Feller1939,Delbruck1940,Kendall1951,mcquarrie1967}), but their applications have rapidly proliferated over the last two decades. In biology, SRNs have been popularised through the Kendall-Gillespie algorithm~\cite{Kendall1950,Gillespie1976} (a.k.a.\ the direct method~\cite{Gillespie1977} or the stochastic simulation algorithm~\cite{gillespie2007}) and its use in modelling chemical reactions \cite{Delbruck1940,mcquarrie1967,higham2008}, gene expression variability across cell populations~\cite{mcadams1997,shahrezaei2008,Thomas2014}, neural network dynamics~\cite{bressloff2009}, predator-prey interactions~\cite{McKane2005}, and the spread of epidemics~\cite{Youssef2011}, among many others. 
Elsewhere, SRNs have also been used to model social dynamics \cite{helbing2010,haag2017} and in a range of engineering, business, and financial applications~\cite{Castro1996,Caliskan2007,Aksin2009,Cont2010}.

While the probability distribution of the chain's state generally changes with time, it often becomes independent of the chain's (typically unknown) initial conditions in the long-term. Such limiting distributions are known as \emph{stationary distributions}, and are commonly used as summary statistics for these models.
%
%
%
%

The theory regarding stationary distributions and the long-term behaviour of continuous-time chains 
is classical. Yet standard texts (e.g.\ \cite{Norris1997,Asmussen2003,Bremaud1999,Anderson1991}) on this subject assume irreducibility of the chain, an assumption which guarantees a unique stationary distribution. This condition is often  difficult to verify in practice or not met in applications  \cite{Pauleve2014,Kuntz2017,Gupta2018}.
To the best of our knowledge, comprehensive accounts of the theory of stationary distributions of continuous-time chains that omit irreducibility
can only be found in technically advanced papers written for general Markov processes, e.g.\ the series of articles by S.~P.~Meyn and R.~L.~Tweedie~\cite{Meyn1992,Meyn1993a,Meyn1993b}. The first aim of this review is to make this material accessible to non-experts. 
This is covered in Section~\ref{sec:theory}, where we first introduce SRNs, the chemical master equation, 
and the stationary distributions.  
Next, we delineate the conditions under which the stationary distributions determine the long-term behaviour of the chain, and  explain why there may be more than one such distribution. We proceed by giving a simple characterisation of the set of stationary distributions in terms of the closed communicating classes and their associated ergodic distributions. We end the section by
reviewing the Foster-Lyapunov criteria used in practice to establish the existence of these distributions, and by 
discussing simulation methods used to approximate them.

%
The stationary distributions of the chain are the solutions of a set of linear equations with as many unknowns and equations as there are states in the state space.
In most cases of interest,  the state space is infinite or too large and these equations cannot be solved analytically or numerically (see \cite{Serfozo1999,Baskett1975,Jackson1957,Chao1999,haken1974,van1976equilibrium,Staff1970,Kelly1979,Jahnke2007,shahrezaei2008,Anderson2010a,Grima2012,Mather2010,Kumar2014,Dattani2017,Anderson2019,Melykuti2014} for notable exceptions). 
In practice, we overcome this issue through numerical procedures~\cite{Asmussen2007,Goutsias2013,andreychenko2017,schnoerr2017} that yield approximations of the stationary distribution. Among these are methods that approximate the distribution only within a given finite subset (a \emph{truncation}) of the state space, neglecting the rest. Such \emph{truncation-based schemes} date back to the late 60s~\cite{Seneta1967,Seneta1968} and several ways to solve the truncated problem have been proposed over the last decade~\cite{Baumann2010,Phung-Duc2010,Dayar2011a,Hart2012,Liu2015,Masuyama2016,Masuyama2017,Masuyama2017a,Kazeev2015,Cao2016,Cao2016b,Gupta2017,Liu2018,Liu2018a,Dayar2011,Spieler2014,Kuntz2017,Kuntz2018a,Kuntzthe,Baumann2013a,Dayar2012}. How to compute and control the approximation errors introduced by these schemes is a matter of ongoing research.

The main aim of this work is to provide an overview of truncation-based schemes applicable to SRNs, explaining their relationships, and comparing different aspects of their performance. We pay particular attention to the \emph{convergence} properties of the schemes (i.e.\  their ability to produce arbitrarily accurate approximations given sufficient computation power); to the errors they introduce; and to \emph{computable} errors or error bounds that may be used in practice to verify the accuracy of the approximations. 
In Section~\ref{sec:truncs}, we introduce the general properties of truncation-based schemes: the notions of convergence, 
the different errors, and  
the optimal approximating distribution (i.e.\  the stationary distribution conditioned on the chain being inside the truncation), which we refer to as the \emph{conditional distribution}. 

The truncation-based schemes are reviewed in Section~\ref{fiveschemes}. 
We start by describing truncation-based approximations for the simple case of birth-death processes, 
 whose conditional distribution can be computed exactly. We proceed by reviewing several truncation-based schemes that recover the conditional distribution in the birth-death case and yield approximations in the general case. 
In Section~\ref{togglesec}, we compare the numerical performance of the schemes on a genetic toggle switch, a well-known model for which no analytical solution exists. We finish the review in Section~\ref{sec:conclusion} by discussing the advantages and limitations of the different schemes, as well as pointing to several theoretical and practical open directions in the field.

For completeness, we have included in this review a series of proofs relevant to truncation-based schemes that we have been unable to locate elsewhere in the literature. For ease of reading, these have been relegated to the appendices.

\section{Stochastic reaction networks, continuous-time chains, and their stationary distributions}\label{sec:theory}

A stochastic reaction network (SRN) involving $n$ species $S_1,S_2,...,S_n$ and $m$ reactions
\begin{align}
\label{eq:network}
\nu_{1j}^-S_1+\dots+\nu_{nj}^-S_n \xrightarrow{a_j} \nu_{1j}^+S_1+\dots+\nu_{nj}^+S_n, 
 \qquad j=1,\ldots,m,  
\end{align}
is often modelled with a minimal time-homogeneous continuous-time Markov chain $X:=(X_t)_{t\geq0}$,  where $X_t:=(X^1_t,\dots,X^n_t)$ compiles the number of individuals (or \emph{molecules}) of each of the species $S_1,\dots,S_n$ at time $t$. The chain takes values in a (possibly infinite) subset $\s$ of $\n^n$, where $\n$ denotes the set of non-negative integers, known as the \emph{state space} and has \emph{rate matrix} $Q:=(q(x,y))_{x,y\in\s}$ defined by
\begin{equation}
\label{eq:qmatrixsrn}
q(x,y):=\sum_{j=1}^ma_j(x)(1_{x+\nu_j}(y)-1_x(y))\quad\forall x,y\in\s.
\end{equation}
In the above, $1_x$ denotes the indicator function of state $x$ ($1_x(y)$ is one if $x=y$ and zero otherwise), $\nu_{ij}^{\pm}\in\n$ the stoichiometric coefficients, $a_j:\nn\to[0,\infty)$ the propensity of the $j^{th}$ reaction, and $\nu_j:=(\nu_{1j}^+-\nu_{1j}^-,\dots,\nu_{nj}^+-\nu_{nj}^-)$ the stoichiometric vector gathering the net changes in species numbers produced by the reaction. It is straightforward to check that the rate matrix $Q$ is \emph{totally stable} and \emph{conservative}:
\begin{equation}\label{eq:qmatrix}q(x,y)\geq0\quad\forall x\neq y,\qquad
q(x):=-q(x,x)=\sum_{y\neq x}q(x,y)<\infty\quad\forall x\in\cal{S},\end{equation}
where the sum is taken over all states $y$ in $\s$ that are not $x$. While our work here is motivated by chains with rate matrices of the form in~\eqref{eq:qmatrixsrn},  we find it convenient throughout this review to allow $Q$ to be any matrix satisfying~\eqref{eq:qmatrix}.

In practice, the chain $X$ is often generated by running the Kendall-Gillespie algorithm~\cite{Feller1940,Kendall1950,Gillespie1976}: sample  a state $x$ from an \emph{initial distribution} $\lambda:=(\lambda(x))_{x\in\s}$ and start the chain at $x$ (i.e.\  set $X_0:=x$). If $q(x)$ is zero, leave the chain at $x$ for all time. Otherwise, wait an exponentially distributed amount of time with mean $1/q(x)$, sample $y$ from the probability distribution $p(x,\cdot)$ given by
\begin{equation}\label{eq:jumpmat}
p(x,y):=
\left\{\begin{array}{cl} \left(1-1_x(y)\right)\dfrac{q(x,y)}{q(x)} &\text{if }q(x)>0 \\
1_x(y) &\text{otherwise}\end{array}\right.
\qquad \forall x,y\in\s,\end{equation}
and update the chain's state to $y$ (we say that the chain \emph{jumps} from $x$ to $y$ and we call the time at which it jumps the \emph{jump time}). Repeat these steps starting from $y$ instead of $x$. All random variables sampled must be independent of each other. Observing the chain at the jump times, we obtain a discrete-time chain with one-step matrix \eqref{eq:jumpmat} known as the \emph{embedded discrete-time chain} or \emph{jump chain}~\cite{Norris1997}.

Technically, we define our chain on a measurable space and construct a family of probability measures
\[
\left \{\Pb_\lambda:\lambda(x)\geq0\enskip\forall x\in\s,\enskip\sum_{x\in\s}\lambda(x)=1 \right \}
\]
on this space such that $\Pb_\lambda$ carries all statistical information regarding the chain if its starting location is sampled from $\lambda$, see \cite[Section~26]{Kuntz2020} for details. 
We will write $\Pb_x$ instead of $\Pb_\lambda$ if the chain starts with probability one at a given state $x$ (i.e.\  if $\lambda=1_x$).  Similarly, we use $\Eb_\lambda$ (resp. $\Eb_x$) to denote expectation with respect to $\Pb_\lambda$ (resp. $\Pb_x$).

If $T_n$ denotes the $n^{th}$ jump time, then the limit
$$T_\infty:=\lim_{n\to\infty}T_n$$
is known as the \emph{explosion time}. It is \cite[Theorem~26.10]{Kuntz2020} the instant by which the chain has left every finite subset of the state space (or infinity should this event never occur). In the case of an SRN~\eqref{eq:network}, an explosion occurs if and only if the count of at least one species diverges to infinity in a finite amount of time.  We say that the chain is \emph{non-explosive} if, with probability one, no such explosion occurs:
\begin{equation}\label{eq:nonexp}\Pbl{\{T_\infty=\infty\}}=1.\end{equation}
If \eqref{eq:nonexp} holds for every initial distribution $\lambda$, then the rate matrix $Q$ is said to be \emph{regular}.  
To simplify the exposition, \textit{we assume throughout this review that rate matrices $Q$ are regular}, an assumption typically verified using Theorem~\ref{lyareg} below. For information on how the results of this section generalise beyond the regular case, see Appendix~\ref{app:nonreg}.

\subsection{The time-varying law and the chemical master equation (CME)}\label{sec:cme}

We denote the probability of observing the chain in the state $x$ at time $t$ by
$$p_t(x)=\Pbl{\{X_t=x\}},$$
and we refer to the distribution $p_t:=(p_t(x))_{x\in\s}$ as the \emph{time-varying law}. 
We denote the $p_t$-average of a real-valued function $f$ on $\s$ by
\[p_t(f):=\sum_{x\in\s}f(x) \, p_t(x) \,, \]
provided that the sum is well-defined.   
In general, given a distribution $\rho$, we denote the $\rho$-average of $f$  by:
\[\rho(f):=\sum_{x\in\s}f(x)\rho(x) \, . \]

Assuming that $Q$ is regular,  the time-varying law is the only solution of the \emph{chemical master equation} (\emph{CME})
\begin{align}\label{eq:CME}
\frac{\mathrm{d}p_t(x)}{\mathrm{d}t} = p_tQ(x):=\sum_{z\in\s}p_t(z)q(z,x),\qquad p_0(x)=\lambda(x),\qquad\forall x\in\s,
\end{align}
that is a probability distribution on $\s$ for all $t\geq0$, see \cite[Theorems~2.8.4~and~2.8.6]{Norris1997} for the case $\lambda=1_x$ and \cite[Section~33]{Kuntz2020} for a general initial distribution $\lambda$.

\subsection{The stationary distributions and the stationary solutions of the CME}\label{sec:statintro}
 
A probability distribution $\pi:=(\pi(x))_{x\in\s}$ on $\s$ is said to be a \emph{stationary distribution} of the chain if setting the initial distribution equal to $\pi$ ensures that the chain remains distributed according to $\pi$ for all time:
\begin{equation}\label{eq:statdef}\Pbp{\{X_t=x\}}=\pi(x),\quad\forall x\in\s,\enskip t\geq0.\end{equation}
Taking time-derivatives of both sides of \eqref{eq:statdef}, we find that stationary distributions are fixed points of the CME~\eqref{eq:CME}:
\begin{equation}\label{eq:stat}\pi Q(x)=0\qquad\forall x\in\s.\end{equation}
We call any probability distribution $\pi$ satisfying \eqref{eq:stat} a \emph{stationary solution} of the CME. If the rate matrix is regular, then \emph{any} stationary solution is also a stationary distribution of the chain:
\begin{theorem}[{\cite[Theorem~1]{Miller1963}}]\label{Qstateq} If $Q$ is regular, then a probability distribution $\pi$ satisfies~\eqref{eq:statdef} if and only if it satisfies \eqref{eq:stat}.
\end{theorem}
\subsection{The long-term behaviour I: the stability of the chain}\label{sec:stability}

Stationary distributions often determine the chain's long-term behaviour. In particular, for many chains, the time-varying law converges to a stationary distribution $\pi$ in total variation (c.f.\ Section~\ref{sec:convergence}),
\begin{equation}\label{eq:spaceaverages}\lim_{t\to\infty}\norm{p_t-\pi}_{TV}=0,\end{equation}
and so does the \emph{empirical distribution} $\epsilon_T$,
\begin{equation}\label{eq:timeaverages}\lim_{T\to\infty}\norm{\epsilon_T-\pi}_{TV}=0\quad\Pb_\lambda\text{-almost surely},\end{equation}
where $\epsilon_T(x)$ denotes the fraction of the time-interval $[0,T]$ that the chain spends in state $x$:
\begin{equation}\label{eq:empirical}\epsilon_T(x):=\frac{1}{T}\int_0^{{T}}1_x(X_t)dt\quad\forall x\in\s.\end{equation}
As we will see in Theorem \ref{doeblinc} below, the stationary distributions featuring in \eqref{eq:spaceaverages} and \eqref{eq:timeaverages} generally differ. {Moreover, in contrast with the time-varying law $p_t$, the empirical distribution $\epsilon_T$ is a random object depending on the chain's path and, consequently, the limiting distribution $\pi$ in~\eqref{eq:timeaverages} may be a random combination of stationary distributions (with a slight abuse of terminology, we also refer to $\pi$ as a stationary distribution).}

In a series of articles~\cite{Meyn1992,Meyn1993a,Meyn1993b}, S.~P.~Meyn and R.~L.~Tweedie showed that for all starting conditions there exists a stationary distribution satisfying \eqref{eq:spaceaverages}  {and another satisfying}~\eqref{eq:timeaverages}
if and only if the chain is \emph{bounded in probability}. (In fact, the results in \cite{Meyn1993a,Meyn1993b} are phrased in terms of a slightly more involved property, \emph{boundedness in probability on average}, but these two properties coincide for continuous-time chains, see Appendix~\ref{app:bip} for details). A chain is bounded in probability if and only if for each $0<\varepsilon< 1$ and deterministic initial condition $x$, there exists a finite set $F\subseteq\s$ such that the probability that $X_t$ lies in $F$ is at least $(1-\varepsilon)$, for all sufficiently large times $t$:
\begin{equation}\label{eq:bip}\liminf_{t\to\infty}\Pbx{\{X_t\in F\}}\geq 1-\varepsilon.\end{equation}
A chain with a regular rate matrix is bounded in probability~\cite{Meyn1993a} if and only if  it is not
\begin{itemize}
\item \emph{transient:} the paths diverge to infinity in an infinite amount of time;
\item or \emph{null recurrent:} the paths do not tend to infinity but they do explore ever larger regions of the state space in a manner that the empirical distribution tends pointwise to zero as  time progresses;
\item or a combination of the above.
\end{itemize}
Because the above are viewed as unstable behaviours, chains that are bounded in probability are typically considered stable. {As we see in the following, these are the chains that admit stationary distributions.}

\subsection{The long-term behaviour II: the set of stationary distributions}
\label{sec:doeblin}

Stationary distributions need not be unique. 
Non-uniqueness arises only if the state space breaks down into several disjoint sets that the chain is unable to leave. In particular, we decompose the state space as
\begin{equation}\label{eq:ddc}
\s=\left(\bigcup_{i\in\cal{I}}\cal{C}_{i}\right)\cup \cal{T},\end{equation}
where $\{\cal{C}_i\}_{i\in\cal{I}}$ are the \emph{closed communicating classes}, $\cal{I}$ is their indexing set, and $\cal{T}:=\s\backslash \cup_{i\in\cal{I}}\cal{C}_i$  contains all states that do not belong to one of these classes. 
A set $\cal{C}\subseteq\s$ is said to be a closed communicating class if the chain can travel between any pair of states in $\cal{C}$ (via one or more jumps) but cannot travel from any state inside $\cal{C}$ to one outside. In terms of the rate matrix, $\cal{C}$ is a closed communicating class if and only if  given any $x,y\in\cal{C}$ {there exists a sequence of states $x_1,x_2,\dots,x_l\in\s$ through which the chain can travel from $x$ to $y$, i.e.\ } 
\begin{equation}\label{eq:irred}
q(x,x_1)\,q(x_1,x_2)\dots q(x_{l-1},x_l) \, q(x_l,y)>0,\end{equation}
but no such sequence exists 
if $y$ lies outside of $\cal{C}$, see \cite[Theorem~3.2.1]{Norris1997}. It follows that the closed communicating classes must be disjoint sets.

For instance, consider the simple SRN~\cite{Kuntz2017}
$$ \emptyset\xrightarrow{a_1} 2S_1\xrightarrow{a_2}\emptyset,\qquad S_2\xrightarrow{a_3}\emptyset,$$
with mass action-kinetics (e.g.\ $a_1(x)=1$, $a_2(x)=x_1(x_1-1)/2$, and $a_3(x)=x_2$). As all reactions preserve the parity of the number of $S_1$ molecules and no reaction produces molecules of $S_2$, the state space $\s=\n^2$ of the network  decomposes into
\begin{align*}\underbrace{\{(x_1,0):x_1\in\n\text{ is odd}\}}_{\textstyle\begin{array}{c}\cal{C}_1\end{array}}\cup\underbrace{\{(x_1,0):x_1\in\n\text{ is even}\}}_{\textstyle\begin{array}{c}\cal{C}_2\end{array}}\cup\underbrace{\{(x_1,x_2):x_1\in\n, x_2\in\zp\}}_{\textstyle\begin{array}{c}\cal{T}\end{array}},\end{align*}
{where $\zp$ denotes the set of positive integers.} Because many SRNs posses conservation laws, it is also not uncommon for there to exist infinitely many closed communicating classes. For example, the reactions
$$2S_1\rightleftarrows{}S_2,$$ 
conserve the quantity $i:=x_1+2x_2$. Hence, with the choice of state space $\s:=\n^2$, there exists a different closed communicating class for every $ i$~in~$\n$.

Suppose that the chain is stable (in the sense of Section~\ref{sec:stability}). If it starts in a closed communicating class, then it will never escape and the time and ensemble averages converge to {the} stationary distributions  featuring in \eqref{eq:spaceaverages}--\eqref{eq:timeaverages}. For this reason, there {must} exist at least one such $\pi^i$ per closed communicating class $\cal{C}_i$ with support contained in $\cal{C}_i$ ($\pi^i(\cal{C}_i):=\sum_{x\in\cal{C}_i}\pi^i(x)=1$) {and the class is said to be \emph{positive recurrent}}. Because the states within a class communicate, it can be shown that this distribution is unique. It is referred to as the \emph{ergodic distribution} associated with class $\cal{C}_i$.

{On the other hand, the chain visits any given state $x$ in $\cal{T}$ at most finitely many times.  It follows that the probability of the chain being at state $x$ decays to zero as time progresses }and thus, by \eqref{eq:statdef}, it must be the case that $\pi(x)=0$ for any state $x$ in $\cal{T}$. Because we are assuming that the chain does not diverge to infinity, it follows that it eventually enters a closed communicating class. The strong Markov property then implies that the stationary distribution $\pi$ in \eqref{eq:timeaverages} is the ergodic distribution $\pi^i$ associated with the class $\cal{C}_i$ that the sample path enters and the stationary distribution in \eqref{eq:spaceaverages} is a combination of the ergodic distributions weighed by the probabilities of entering the corresponding communicating classes. These facts are summarised in the following theorem {whose proof can be found in Appendix~\ref{app:doeblin}}. 
\begin{theorem}\label{doeblinc} {Suppose that $Q$ is regular and} let $\{\cal{C}_i\}_{i\in\cal{I}}$ be as in \eqref{eq:ddc}.
\begin{enumerate}[label=(\roman*)]
\item For each $i$ in $\cal{I}$, there exists at most one stationary distribution $\pi^i$ with support contained in class $\cal{C}_i$ (i.e.\  $\pi^i(\cal{C}_i)=1$).  
\item A probability distribution $\pi$ is a stationary distribution of the chain if and only if it is a convex combination of the ergodic distributions:
$$\pi=\sum_{i\in\cal{I}_e}\theta_i\pi^i$$
for some collection $(\theta_i)_{i\in\cal{I}_e}$ of non-negative weights satisfying $\sum_{i\in\cal{I}_e}\theta_i=1$, where $\cal{I}_e$ gathers the indices $i\in\cal{I}$ of the positive recurrent classes.
\item The chain is bounded in probability if and only if all closed communicating classes are positive recurrent and, regardless of the initial distribution $\lambda$, the chain enters one of these classes with probability one:
$$\Pbl{\cup_{i\in\cal{I}}H_i}=\sum_{i\in\cal{I}}\Pbl{H_i}=1$$
for all probability distributions $\lambda$, where $H_i$  denotes the event that the chain enters the closed communicating class $\cal{C}_i$. 
\item The chain is bounded in probability if and only if, for all initial distributions $\lambda$, \eqref{eq:spaceaverages} holds with
\begin{equation}\label{eq:pilim}\pi:=\sum_{i\in\cal{I}}\Pbl{H_i}\pi^i\end{equation}
and \eqref{eq:timeaverages} holds with
$$\pi:=\sum_{i\in\cal{I}}1_{H_i}\pi^i,$$
where $1_{H_i}$ denotes the indicator function of the event $H_i$.
\end{enumerate} 
\end{theorem}

The chain (or rate matrix) is said to be \emph{$\varphi$-irreducible} if there exists only one closed communicating class $\cal{C}$ and the chain has positive probability of  travelling from any state outside of the class to the states inside (i.e.\  for all $x\in\cal{T}$ and $y\in\cal{C}$, there exists a sequence $x_1,x_2,\dots,x_l\in\s$ satisfying  \eqref{eq:irred}). If this class is the entire state space ($\cal{C}=\s$), then the chain (or rate matrix) is said to be \emph{irreducible}. Theorem \ref{doeblinc} has the following well-known corollary for irreducible chains.

\begin{corollary}\label{iredcor} Suppose that the chain {is} $\varphi$-irreducible and $Q$ is regular. Then, it has at most one stationary distribution $\pi$. If the chain is bounded in probability, then  $\pi$ exists and  \eqref{eq:spaceaverages}--\eqref{eq:timeaverages} hold for all initial distributions $\lambda$. If the chain is irreducible, then the existence of $\pi$ implies that the chain is bounded in probability.
\end{corollary}

If the chain is both $\varphi$-irreducible and bounded in probability, then Corollary \ref{iredcor} shows that, asymptotically, the \emph{space} (or \emph{population}, or \emph{ensemble}) averages and the \emph{time} (or \emph{empirical}) averages coincide:
$$p_T\approx \epsilon_T,\quad\Pb_\lambda\text{-almost surely}$$
for large enough $T$, and the chain is said to be \emph{ergodic}. 
{If, additionally, the rate of convergence in~\eqref{eq:spaceaverages} is exponential for all deterministic initial conditions, i.e.\ for all $x$ in $\mathcal{S}$ there exists a $\alpha>0$ such that
\begin{align}\label{eq:experg}
\norm{p_t-\pi}_{TV}\leq \cal{O}(e^{-\alpha t})  \end{align}
with $\lambda=1_x$, then the chain is further classified as \emph{exponentially ergodic}. 
} %

Establishing whether a chain is $\varphi$-irreducible or irreducible is a challenging problem {for which} several computational methods have been proposed~\cite{Pauleve2014,Kuntz2017,Gupta2018}. {The stability properties of chains are typically investigated using \emph{Foster-Lyapunov criteria}.}

\subsection{Foster-Lyapunov stability criteria}\label{sec:foster} Except for special cases~\cite{Gupta2014,Anderson2018,Anderson2018a,Rathinam2015,Engblom2014}, ruling out unstable behaviours and establishing boundedness in probability are difficult tasks. Often, we must resort to Foster-Lyapunov criteria (also known as \emph{drift conditions}),  named jointly  after A.~Lyapunov~\cite{Lyapunov1892} who first introduced these types of conditions in his study of  ordinary differential equations and F.~G.~Foster who first ported them to a stochastic setting \cite{Foster1953}. These criteria also yield bounds on stationary averages: morsels of information important for a host of numerical approaches used to study the long-term behaviour of chains including some of those discussed in Sections~\ref{sec:truncs}--\ref{fiveschemes}. We review now the most common criteria, {whose proofs can be found elsewhere---see Appendix~\ref{app:lyaproofs} for appropriate references.} 

We begin with the criterion for regularity which involves a \emph{norm-like} function $w$: a real-valued function on $\s$ with finite sublevel sets, i.e.\  such that
\begin{equation}\label{eq:sublevel}\s_r:=\{x\in\s:w(x)<r\}\quad\text{is finite for all}\quad r=1,2,\dots.\end{equation}
The criterion for regularity then goes as follows:
\begin{theorem}\label{lyareg} The rate matrix $Q$ is regular if and only if there exists a norm-like $v$ such that
$$Qv(x):=\sum_{y\in\s}q(x,y)v(y)\leq d_1v(x)+d_2\qquad\forall x\in\s,$$
for some constants $d_1,d_2$ in $\r$.
\end{theorem}
The Foster-Lyapunov criterion for boundedness in probability (or ergodicity in the $\varphi$-irreducible case) is as follows:
\begin{theorem}\label{lyabia}Suppose that $Q$ is regular and that for some finite set $F$, constant $d>0$ and functions $f\geq1$, $v\geq0$,
\begin{equation}\label{eq:lyabia}Qv(x)\leq d \, 1_F(x)-f(x)\qquad\forall x\in\s,\end{equation}
where $1_F$ denotes the indicator function of the set $F$ (i.e.\  $1_F(x)=1$ if $x\in F$ and $0$ otherwise). The chain is bounded in probability. Moreover, for any stationary distribution $\pi$, the average $\pi(f)$ is bounded by $d$:
\begin{equation}\label{eq:lyabound}1\leq \pi(f)\leq d\sum_{x\in F}\pi(x)\leq d;\end{equation} 
and the probability that $\pi$ awards to the complement of $F$ is bounded by $1-1/d$:
\begin{equation}\label{eq:lyabound2}\sum_{x\not\in F}\pi(x)=1-\sum_{x\in F}\pi(x)\leq 1-\frac{\pi(f)}{d}\leq 1-\frac{1}{d}.\end{equation} 
{Conversely, if $Q$ is regular and the state space is comprised of finitely many closed communicating classes (i.e.\  $\cal{T}$ in~\eqref{eq:ddc} is empty and $\cal{I}$ therein is finite)}, then the criterion is sharp: if the chain is bounded in probability, then there exists a finite set $F$, constant $d>0$, and functions $f\geq1$, $v\geq0$ satisfying \eqref{eq:lyabia}.
\end{theorem}
{We note that the converse need not hold if either $\cal{T}$ is non-empty or $\cal{I}$ is infinite, see \cite[Section 49]{Kuntz2020} for counter-examples.} Of course, by combining Theorems \ref{lyareg} and \ref{lyabia}, we obtain a criterion for both regularity and boundedness in probability. A considerably stronger result holds:
\begin{theorem}\label{lyacom}Suppose that there exists constants $d_1>0$ and $d_2\in\r$ and a norm-like function $v\geq1$ satisfying
\begin{equation}\label{eq:lyacom}Qv(x)\leq -d_1v(x)+d_2,\qquad\forall x\in\s.\end{equation}
{Then, }the chain is bounded in probability and $Q$ is regular. {Moreover,  for any initial distribution $\lambda$ satisfying $\lambda(v)<\infty$, the time-varying law converges exponentially fast, i.e.\ \eqref{eq:experg} holds for some $\alpha>0$ where $\pi$ is as in~\eqref{eq:pilim}.} 
\end{theorem}
Finding by hand a \emph{Lyapunov function} $v$ satisfying the inequalities in the above criteria often proves challenging. For this reason, methods that search for these functions computationally have been the subject of attention over the last few years \cite{Parrilo2000,Dayar2011,Spieler2014,Gupta2014,Argeitis2014}.

\subsection{Monte-Carlo estimators for the stationary distributions}
\label{sec:simulation}

In the case of a unique stationary distribution $\pi$, \eqref{eq:timeaverages} justifies the naive Monte-Carlo approach to approximating $\pi$: choose a \emph{final time} $T>0$, generate a sample path over $[0,T]$ using an exact simulation algorithm, such as the {Kendall-Gillespie algorithm}, compute the empirical distribution $\epsilon_T$ in \eqref{eq:empirical}, and use it as an estimate of $\pi$. 

When it comes to quantifying the error of $\epsilon_T$, only asymptotic results are known. For simplicity, consider using the \emph{empirical average} $\epsilon_T(f)$ as an approximation of the \emph{stationary average} $\pi(f)$, where
$$\epsilon_T(f):=\sum_{x\in\s}f(x)\epsilon_T(x)=\frac{1}{T}\int_0^Tf(X_t)dt,$$
for some given square $\pi$-integrable real-valued function $f$ on $\s$. A central limit theorem \cite[Prop.~IV.1.3]{Asmussen2007} shows that, as $T$ approaches infinity, the \emph{error} $e_f:=\pi(f)-\epsilon_T(f)$ converges in distribution to a zero-mean Gaussian 
with variance $\sigma^2/T$, known as the \emph{asymptotic variance}.
However, computing the constant $\sigma^2$ requires the unknown {auto-covariance function of $f(X_t)$~\cite[Prop.~IV.1.3]{Asmussen2007}:
\[
\sigma^2=\int_0^\infty\Eb_\pi \left[\left(f(X_t)-\pi(f)\right)\left(f(X_0)-\pi(f)\right) \right] \, dt \, .
\]
} 
Hence, it is difficult to quantify the estimation error.

A host of improved statistical estimators for stationary distributions $\pi$ of continuous-time chains have been proposed in the literature: $\pi$ and averages thereof can be estimated using the embedded discrete-time chain~\cite{Hordijk1976}, the regenerative structure of the chain~\cite{Glynn2006}, importance sampling~\cite{Goyal1987,Glynn1989,Goyal1992,Heidelberger1995}, look-ahead estimators~\cite{Henderson2001,Stachurski2008,Stachurski2012}, and splitting methods~\cite{Salis2005,Valeriani2007,Warmflash2007,Bhatt2010}, see \cite[Chap.~IV]{Asmussen2007} for an introduction to these techniques. However, due to the asymptotic nature of the problem, these approaches all yield biased estimates of $\pi$ with difficult-to-quantify errors. One notable exception are perfect sampling estimators~\cite{Asmussen1992,Propp1996,Kendall2000,hemberg2007,Hemberg2008} that are not only unbiased but yield independent samples drawn from $\pi$. Unfortunately, these are applicable only in special cases.
\section{Convergence and errors of truncation-based approximation schemes}\label{sec:truncs}
{In this section, we consider the problem of computing a given stationary distribution $\pi$. As mentioned in Section~\ref{sec:intro}, analytical formulas for $\pi$ are only available in a few special cases. Furthermore, if the state space is infinite (or just large), the stationary equations~\eqref{eq:stat} cannot be solved directly. 
For these reasons, we have to approximate $\pi$ numerically in most cases of interest.} Truncation-based schemes use a finite subset, or \emph{truncation}, $\s_r$ of the state space and the \emph{truncated rate matrix} $(q(x,y))_{x,y\in\s_r}$ to compute an approximation $(\pi_r(x))_{x\in\s_r}$ of $\pi$'s restriction~$(\pi(x))_{x\in\s_r}$ to $\s_r$. This approximation is then padded with zeros,
\begin{equation}\label{eq:padding}\pi_r(x):=\left\{\begin{array}{ll}\pi_r(x)&\text{if }x\in\s_r\\ 0&\text{if }x\not\in\s_r\end{array}\right.\quad\forall x\in\s,\end{equation}
and used as an approximation $\pi_r=(\pi_r(x))_{x\in\s}$ of the entire stationary distribution.
\subsection{Convergence}\label{sec:convergence}

It is typically the case with these schemes that adding states to the truncation employed improves the quality of the approximation produced. In exchange, larger truncations lead to greater computational costs. To formalise these ideas, we view the schemes as procedures that return an entire sequence $\pi_1,\pi_2,\dots$ of approximations corresponding to a sequence of increasing truncations
$\s_1\subseteq\s_2\subseteq\dots,$
instead of a single approximation corresponding to a single truncation. 

A sanity check for the correctness of these methods is establishing their \emph{convergence}: their ability to produce arbitrarily accurate approximations given enough computational power. That is, showing that, if the limit of the truncations is the entire state space,
\begin{equation}\label{eq:trunc1}\lim_{r\to\infty}\s_r=\bigcup_{r=1}^\infty\s_r=\s,\end{equation}
then the sequence $(\pi_r)_{r\in\zp}$ converges, in one sense or another, to the stationary distribution of interest $\pi$. To formalise the various notions of convergence, we regard both the approximations and the stationary distributions as points in the space 
\begin{equation}\label{eq:l1}\ell^1:=\left\{(\rho(x))_{x\in\s}:\sum_{x\in\s}\mmag{\rho(x)}<\infty\right\}\end{equation}
of absolutely summable real sequences indexed by states in $\s$ and tacitly use
$$\rho(A):=\sum_{x\in A}\rho(x)\quad\forall \rho\in\ell^1$$
to identify $\ell^1$ with the space of finite signed measures on $(\s,\{A:A\subseteq\s\})$.

The weakest type of convergence we consider is pointwise convergence:
\begin{equation}\label{eq:trunc2}
\lim_{r\to\infty}\pi_r(x)=\pi(x)\quad \forall x\in\s, 
\end{equation}
which ensures that, $\pi_r(A)$ is close to the probability $\pi(A)$ that $\pi$ awards to any given finite set $A\subseteq\s$ for sufficiently large $r$. 

To study the  convergence of $\pi_r(A)$ uniformly over all $A\subseteq\s$, we use the total variation and $\ell^1$-distances:
\begin{equation}\label{eq:tvl1def}\norm{\pi_r-\pi}_{TV}:=\sup_{A\subseteq\s}\mmag{\pi_r(A)-\pi(A)},\qquad \norm{\pi_r-\pi}_{1}:=\sum_{x\in\s}\mmag{\pi_r(x)-\pi(x)}.\end{equation}
The total variation distance measures the maximum error in the probability that $\pi_r$ assigns to any event $A$, while the $\ell^1$-distance measures the total absolute error.
We say that $\pi_r$ converges in total variation (or in $\ell^1$) if the above distances tend to zero as $r$ approaches infinity. 

Even though convergence in total variation shows that $\pi_r(A)$ converges to $\pi(A)$ for any event $A$, it gives us no information whether $\pi_r(f)$ is an accurate approximation of the average $\pi(f)$ if $f$ is an unbounded real-valued function on $\s$. Here, instead we use the $w$-norm:
\begin{equation}\label{eq:wnormdef}\wnorm{\rho}:=\sum_{x\in\s}w(x)\mmag{\rho(x)},\end{equation}
where $w$ is a given positive function on $\s$. 
The approximations $\pi_r$ converge to $\pi$ in $w$-norm if and only if\footnote{If $w:=1$, then this is the well-known Schur property of the space $\ell^1$. For general positive $w$, note that $f$ has finite $w$-norm if and only if $wf$ has finite supremum norm and that $\pi_r$ converges to $\pi$ in $w$-norm if and only if $\tilde{\pi}_r$ converges to $\tilde{\pi}$ in $\ell^1$, where $\tilde{\pi}_r(x):=w(x)\pi_r(x)$ and $\tilde{\pi}(x):=w(x)\pi(x)$ for all $x$ in $\s$.} the approximate averages $\pi_r(f)$ converge to $\pi(f)$ for every function $f$ that grows no faster than $w$ times a constant:
$$\sup_{x\in\s}\frac{\mmag{f(x)}}{w(x)}<\infty.$$

For some schemes, convergence in $w$-norm for unbounded $w$ proves too stringent of a requirement and we instead use a slightly weaker notion. We say that $\pi_r$ converges {\emph{$w$-weakly*}} to $\pi$  if the averages $\pi_r(f)$ converge to $\pi(f)$ for each $f$ that asymptotically grows slower than $w$ times any positive constant:
\begin{equation}\label{eq:weakstar2}\lim_{r\to\infty}\sup_{x\not\in\s_r}\frac{\mmag{f(x)}}{w(x)}=0,\end{equation}
where $\s_1\subseteq\s_2\subseteq\dots$ are any increasing truncations approaching the state space~\eqref{eq:trunc1}.

These notions of convergence form a hierarchy  \cite[Chap.~5]{Kuntzthe}: if $w$ is positive and norm-like, then convergence 
\begin{equation}\label{eq:hierarchy}\text{in $w$-norm}\Rightarrow {w\text{-weakly*}}\Rightarrow\text{in $\ell^1$}\Leftrightarrow\text{in total variation}\Rightarrow\text{pointwise},\end{equation}
for any sequence of points in the space $\ell^1$. If the approximations $\pi_r$ are probability distributions (i.e.\  non-negative with mass one), pointwise convergence implies convergence in total variation if and only if the limit $\pi$ is also a probability distribution, which follows from \eqref{eq:hierarchy} and Scheff\'e's Lemma \cite[5.10]{Williams1991}.  The norms themselves are related as follows: 
\begin{equation}\label{eq:normhier}
\text{if } w \geq 1, \quad \text{then} \quad 
\frac{1}{2}\norm{\rho}_{1}\leq\norm{\rho}_{TV}\leq \norm{\rho}_{1}\leq \norm{\rho}_w\quad\forall \rho\in\ell^1\end{equation}
{with $\norm{\rho}_1/2=\norm{\rho}_{TV}$} if and only if $\rho$ is the difference between two probability distributions, and {$\norm{\rho}_1=\norm{\rho}_{TV}$} if and only if $\rho$ is an unsigned measure (i.e.\ $\rho(x)\geq0$ for all $x$ in $\s$).

\subsection{Approximation error}\label{sec:errors} 

Even though the convergence of a scheme is a reassuring indication that we are on the right track, we are still faced with the question  of how much computational power we need to obtain a good approximation. To answer this question, we must compute, or at least bound, the approximation error of the scheme, measured in terms of one of the norms introduced in Section~\ref{sec:convergence} (if we are approximating a specific average or marginal instead of the entire distribution, we use slightly different error measures, see Section~\ref{sec:iter}). Truncation-based approximations have two sources of error: 
\begin{equation}\label{eq:errordecomptrun}\underbrace{\norm{\pi_r-\pi}{_w}}_{\text{approximation error}}=\underbrace{||\pi_r-\pi_{|r}||{_w}}_{\text{scheme-specific error}}+\underbrace{||\pi-\pi_{|r}||{_w}}_{\text{truncation error}}.\end{equation}
{Here, $\norm{\cdot}_w$ denotes any $w$-norm~\eqref{eq:wnormdef} (including the $\ell^1$-norm obtained by setting $w:=1$), and $\pi_{|r}$ denotes the (zero-padded) restriction of $\pi$ to $\s_r$:}
\begin{equation}
\label{eq:rest}
\pi_{|r}(x):=\left\{\begin{array}{ll}\pi(x)&\text{if }x\in\s_r\\ 0&\text{if }x\not\in\s_r\end{array}\right.\qquad\forall x\in\s.
\end{equation}
The \emph{truncation error} in \eqref{eq:errordecomptrun} accounts for the approximation's (total) failure to describe the stationary distribution outside of the truncation: regardless of the details of the scheme, the approximation is zero everywhere outside of the truncation. The \emph{scheme-specific error}  accounts for the errors introduced by the scheme within the truncation. Because the total variation norm is defined in terms a supremum instead of a sum, we are unable to neatly decompose the approximation error in terms of the truncation and scheme-specific errors as in \eqref{eq:errordecomptrun}. However, it is straightforward to bound the former in terms of the latter:
\begin{equation}\label{eq:errordecomptrun2} \max\{||\pi_r-\pi_{|r}||_{TV},||\pi-\pi_{|r}||_{TV}\}\leq\norm{\pi_r-\pi}_{TV}\leq||\pi_r-\pi_{|r}||_{TV}+||\pi-\pi_{|r}||_{TV}.\end{equation} 

Because $\pi$ is unknown, we are often unable to compute these errors exactly and instead must settle for bounding them. Before we discuss how to do this, we point out a simple but insightful consequence of \eqref{eq:errordecomptrun} and \eqref{eq:errordecomptrun2}: the full approximation error is no smaller than the truncation error which is independent of the approximation $\pi_{r}$. Because there exists at least one truncation-based approximation that achieves this error (namely the restriction $\pi_{|r}$) the truncation error is the smallest possible, or \emph{optimal}, approximation error.%
\subsection{Truncation error}\label{sec:truncerr} In either the $\ell^1$ or total variation cases, the truncation error is simply the \emph{tail mass} $m_r$ (i.e.\  the probability that the stationary distribution awards to the outside of the truncation):
\begin{equation}\label{eq:tailmass}||\pi-\pi_{|r}||_{TV}=||\pi-\pi_{|r}||_{1}=\sum_{x\not\in\s_r}\pi(x)=:m_r,\end{equation}
and we use the terms `truncation error' and `tail mass' interchangeably throughout this review. For this reason, the choice of truncation is critical for the computation of accurate approximations. Moreover, {in our experience (e.g.\ see Section~\ref{togglesec})}, truncations with small tail masses typically also result in small scheme-specific errors  {for the schemes studied in Section~\ref{fiveschemes}}.

Choosing a truncation with a \emph{verifiably} small tail mass is a difficult task. 
We know of two ways to systematically generate truncations accompanied by bounds on their tail masses: 
using Foster-Lyapunov criteria \cite{Dayar2011,Dayar2011a,Spieler2014} or using moment bounds \cite{Kuntz2017,Kuntzthe,Kuntz2018a}. For the sake of simplicity, we focus here on the latter and leave the former to Remark \ref{dsapproach} below. Suppose that we have at our disposal a \emph{moment bound} meaning a norm-like (in the sense of \eqref{eq:sublevel}) function $w$ and constant $c$ such that
\begin{equation}
\label{eq:mombound}\pi(w)=\sum_{x\in\s}w(x)\pi(x)\leq c.
\end{equation}
The name `moment bound' stems from $\pi(w)$ frequently being  a moment, or a linear combination of moments, as $w$ is often chosen to be a polynomial.
This type of bound can be obtained using a Foster-Lyapunov criterion such as Theorem \ref{lyabia} (see also \cite{Glynn2008} and references therein)  or mathematical programming methods that have recently drawn much attention \cite{Schwerer1996,Kuntz2017,Kuntzthe,Sakurai2017,Dowdy2018,Dowdy2017,Ghusinga2017a,Gupta2014,Argeitis2014,Dayar2011,Spieler2014}. Setting $\s_r$ to be the $r^{th}$-sublevel set of $w$ (defined in \eqref{eq:sublevel}), we find that
\begin{equation}\label{eq:tailbound}m_r=\sum_{x\not\in\s_r}\pi(x)=\sum_{x\not\in\s_r}\frac{w(x)}{w(x)}\pi(x)\leq \frac{1}{r}\sum_{x\not\in\s_r}w(x)\pi(x)\leq \frac{c}{r}.\end{equation}
In other words, we obtain a \emph{computable bound} on the truncation error (measured using the total variation or $\ell^1$ distances), which we refer to as a \emph{tail bound}. 
\begin{remark}[Foster-Lyapunov tail bounds]\label{dsapproach}Instead of {exploiting} a moment bound to obtain truncations with computable tail bounds, Dayar, Spieler et al. \cite{Dayar2011,Spieler2014} proposed using the Foster-Lyapunov criterion in Theorem~\ref{lyabia}. They suggested finding a non-negative function $u$ such that $-Qu$ is norm-like (in the sense of \eqref{eq:sublevel}) and defining the truncation as
\begin{equation}\label{eq:suplevel}\s_r:=\left\{x\in\s:Qu(x)>-r\left(\max_{x\in\s}Qu(x)\right)\right\}.\end{equation}
It is not difficult to then show that \eqref{eq:lyabia} is satisfied with $d:=(r+1)/r$, $F:=\s_r$, 
$$v(x):=\frac{u(x)}{r\max_{x\in\s}Qu(x)},\quad
f(x):=\left\{\begin{array}{ll}1&\text{if }x\in\s_r \\
-\dfrac{Qu(x)}{r\max_{x\in\s}Qu(x)} &\text{if }x\not\in\s_r\end{array}\right..$$
In particular, the bound \eqref{eq:lyabound} now reads $m_r\leq1/(r+1)$. 

Because we are only able to influence $Qu$ indirectly via our choice of $u$, the function $Qu$  can be quite complicated and it can be difficult to deduce what its superlevel sets \eqref{eq:suplevel} are. In these cases, it is often possible to replace $\s_r$ with a more manageable outer approximation thereof \cite{Dayar2011}.
\end{remark}
\subsection{Scheme-specific error}\label{sec:schemespec}
Obtaining accurate approximations of $\pi$ within the truncation (i.e.\  ones with small scheme-specific errors) is as important as choosing a good truncation. Computing or bounding the scheme-specific error is a challenging problem for most truncation-based schemes. Notable exceptions are those \cite{Dayar2011,Spieler2014,Kuntzthe,Kuntz2017,Kuntz2018a} that produce collections $l_r=(l_r(x))_{x\in\s_r}$ and $u_r=(u_r(x))_{x\in\s_r}$  of lower and upper bounds on  $(\pi(x))_{x\in\s_r}$:
$$l_r(x)\leq \pi(x)\leq u_r(x)\qquad\forall x\in\s_r.$$
Padding these bounds with zeros (as in \eqref{eq:padding}), we find expressions for the scheme-specific error in terms of the truncation error, $m_r$~in~\eqref{eq:tailmass}, and the masses of $l_r,u_r$:
{
\begin{align}||l_r-\pi_{|r}||_{TV}&=||l_r-\pi_{|r}||_1=\sum_{x\in\s_r}(\pi(x)-l_r(x))=1-m_r-l_r(\s_r),\nonumber\\
||u_r-\pi_{|r}||_{TV}&=||u_r-\pi_{|r}||_1=\sum_{x\in\s_r}(u_r(x)-\pi(x))=u_r(\s_r)-1+m_r.\label{eq:nf78eahf8aehefa}\end{align}}
Using \eqref{eq:tailmass} and the above, we obtain expressions for the full approximation error:
\begin{align}&\norm{l_r-\pi}_{TV}=\norm{l_r-\pi}_{1}=1-l_r(\s_r),\label{eq:trunerrexp2}\\
\label{eq:trunerrexp}&\norm{u_r-\pi}_{TV}=\max\{u_r(\s_r)-1+m_r,m_r\},\enskip\norm{u_r-\pi}_{1}=u_r(\s_r)-1+2m_r,\end{align}
see \cite[Corollary~20(i)]{Kuntz2017}. Computing the error of $l_r$ entails adding up its entries. In the case of $u_r$, matters are not so simple as \eqref{eq:trunerrexp} involves the truncation error, a quantity typically unknown. However, we can easily calculate a lower bound on $u_r$'s error and, assuming that a tail bound of the type in \eqref{eq:tailbound} is available, an upper bound too:
\begin{align}u_r(\s_r)-1&\leq \norm{u_r-\pi}_{TV}\leq\max\left\{u_r(\s_r)-1+\frac{c}{r},\frac{c}{r}\right\}\label{eq:uppererbound1},\\
u_r(\s_r)-1&\leq\norm{u_r-\pi}_{1}\leq u_r(\s_r)-1+\frac{2c}{r}.\label{eq:uppererbound2}\end{align}

In summary, if error guarantees are important, a straightforward way to obtain them is to employ schemes that yield bounds{, for instance those discussed in Sections~\ref{sec:iter} and \ref{lpappiter}}. However, these schemes do not necessarily achieve smaller errors than those that do not produce bounds (see Section~\ref{togglesec}). 
\subsection{Optimal approximating distributions and the censored chain} \label{sec:opt}

Some truncation-based schemes (e.g.\ those in Sections~\ref{sec:bd}--\ref{sec:taa}) yield approximations $\pi_r$ that are probability distributions and we say that they produce \emph{approximating distributions}. {As shown in~\cite{Zhao1996} (see also Appendix~\ref{app:trunminerr}), in these cases we have  that the total variation error is given by
\begin{align}\norm{\pi_r-\pi}_{TV}&=m_r+\sum_{x\in\s_r^-}(\pi(x)-\pi_r(x))\geq m_r,\label{eq:trunminerr}\end{align}
where $m_r$ denotes the tail mass~\eqref{eq:tailmass} and $\s_r^-:=\{x\in\s_r:\pi_r(x)<\pi(x)\}$ is the collection of states {within $\s_r$} whose probability $\pi_r$ underestimates.} For this reason, the approximating distributions that achieve the smallest possible total variation error $m_r$ are those that bound $\pi$ from above  (i.e.\  such that $\s_r^-=\emptyset$), {in which case, \eqref{eq:tailmass} and \eqref{eq:nf78eahf8aehefa} imply that both the total variation scheme-specific and truncation errors equal $m_r$}. In general, there are infinitely many such distributions. However, assuming that the state space is irreducible, the (zero-padded) conditional distribution, 
\begin{equation}\label{eq:picond}\pi(x|\s_r):=\left\{\begin{array}{ll}\dfrac{\pi(x)}{\pi(\s_r)}&\forall x\in\s_r\\0&\forall x\not\in\s_r\end{array}\right.,\end{equation}
is the only approximating distribution $\pi_r$ that minimises the maximum relative error:
\begin{equation}\label{eq:picond2}\max_{x\in\s_r}\frac{\mmag{\pi_r(x)-\pi(x)}}{\pi_r(x)}.\end{equation}
In this case, the maximum relative error is also the tail mass $m_r$~\eqref{eq:tailmass}, see Appendix~\ref{app:opt}, where we include a proof of these facts {(we have been unable to locate such a proof elsewhere)}. 
Note that the conditional distribution also minimises the total variation error \eqref{eq:trunminerr} as the definition \eqref{eq:picond} implies that $\pi(\cdot|\s_r)$ bounds $\pi$ from above. For these reasons, we say that the conditional distribution $\pi(\cdot|\s_r)$ is \emph{optimal among approximating distributions}. {Unfortunately, evaluating $\pi(\cdot|\s_r)$ directly requires obtaining the stationary distribution or an unnormalised version thereof. An alternative approach to its computation relies on the \emph{censored chain}.}

\juansec{The censored chain}

Before proceeding to the actual schemes, we take a moment here to consider an old question that proves insightful for our approximation problem: given a $\varphi$-irreducible chain $X$ with unique stationary distribution $\pi$, how do we construct a second chain that behaves roughly as $X$ does (in the long-run and otherwise) except that it never leaves the truncation $\s_r$?  
For the reasons given above, we would do well in using an approximating chain $X^{\varepsilon_r}$  whose stationary distribution is the (unpadded) conditional distribution $(\pi(x|\s_r))_{x\in\s_r}$ of $X$.
Such a chain 
can be constructed using a rate matrix $Q^{\varepsilon_r}$ known as the \emph{stochastic complement} of $Q$ \cite{Meyer1989}, which is defined in terms of {the \emph{out-rate} $q_o(x)$ and \emph{out-boundary} $\cal{B}_o(\s_r)$:
\begin{equation}\label{eq:qo}
q_o(x):=\sum_{y\not\in\s_r}q(x,y)\quad\forall x\in\s_r,\quad \cal{B}_o(\s_r):=\{x\in\cal{S}_r:q_o(x)>0\},\end{equation}
(i.e.\ the rate at which $X$ jumps out of the truncation from $x$, and the set of states inside $\s_r$ from which the chain may jump out of $\s_r$, respectively), and the \emph{conditional re-entry matrix} $\varepsilon_r:=(\varepsilon_r(x,y))_{x,y\in\s_r}$:
\begin{equation}\label{eq:exre0}
\varepsilon_r(x,y):=
\begin{cases}
\Pb_\lambda \left(\{\text{$X$ first re-enters $\s_r$ via $y\}$}|\text{$\{X$ first left $\s_r$ via $x$}\} \right) & \text{if }x\in\cal{B}_o(\s_r),\\
0 & \text{if }x\not\in\cal{B}_o(\s_r),
\end{cases}
\end{equation}
for all $x,y\in\s_r$. Here $\lambda$ denotes any initial distribution for which the event $\{X$ first left $\s_r$ via $x\}$ 
that $x$ was the last state $X$ visited before leaving the truncation for the first time 
has non-zero probability, thus ensuring that the conditional probability is well-defined.
If $q_o(x)=0$,  no such distribution exists, as the chain cannot jump out of the truncation from $x$, and we set $\varepsilon_r(x,y)=0$. Otherwise, any distribution $\lambda$ with support on at least one state from which the chain can reach $x$ (e.g.\ $\lambda=1_x$) fits the bill and the strong Markov property implies that the conditional probability is independent of the particular $\lambda$ we use. 
The stochastic complement of $Q$ is then defined as
\begin{equation}
\label{eq:truncandaug}
Q^{\varepsilon_r}=(q^{\varepsilon_r}(x,y))_{x,y\in\s_r} \text{  with  }
q^{\varepsilon_r}(x,y):=q(x,y)+q_o(x) \, \varepsilon_r(x,y) \quad \forall x,y\in\s_r.\end{equation}}

The associated $X^{\varepsilon_r}$ is known as the \emph{censored} or \emph{restricted} chain \cite{Williams1966,Freedman1983a,Zhao1996} and can be viewed as the optimal approximating chain with state space $\s_r$. Given our assumption that $X$ is $\varphi$-irreducible with a unique stationary distribution $\pi$,  the censored chain is $\varphi$-irreducible with its unique stationary distribution being the (unpadded) conditional distribution $(\pi(x|\s_r))_{x\in\s_r}$.
This can be seen as follows. The censored chain behaves identically to $X$ while both remain inside of the truncation. However, the instant $\tau$ that $X$ jumps from a state $x$ inside $\s_r$ to a state outside, the censored chain instead jumps to a state sampled from $\varepsilon_r(x,\cdot)$. The Markov property and \eqref{eq:exre0} imply that $X^{\varepsilon_r}_\tau$ has the same distribution as the original chain $X$ does at the moment it re-enters the truncation. Because this process repeats itself in perpetuity, the ensemble of sample paths of $X^{\varepsilon_r}$ is statistically identical to that obtained by erasing the segments of the paths of $X$ lying outside of $\s_r$ and gluing together the ends of the remaining segments (see Appendix~\ref{app:censored} for more details).  That the censored chain is $\varphi$-irreducible with $(\pi(x|\s_r))_{x\in\s_r}$ as its unique stationary distribution then follows from \eqref{eq:timeaverages}. 
{Since $\s_r$ is finite, were $\varepsilon_r$ to be known, then we could compute  the conditional distribution by solving stationary equations for $X^{\varepsilon_r}$, i.e.\ \eqref{eq:stat} with $Q^{\varepsilon_r}$ and $\s_r$ replacing $Q$ and $\s$. Unfortunately, expressions for $\varepsilon_r$ are rarely available in practice---exceptions include the birth-death processes in Section~\ref{sec:bd} and generalisations thereof known as \emph{downward skip-free processes}~\cite[p.270--272]{Ramaswami1996}.}

\section{A review of truncation-based schemes}\label{fiveschemes}

In this section, we review several truncation-based schemes employed in the literature to approximate the stationary distributions of SRNs \eqref{eq:network}.
Before introducing the schemes, we consider in Section~\ref{sec:bd} the approximation problem for birth-death processes (BDPs, Figure~\ref{fig:cartoon}(a)). 
In this case, it is straightforward to compute the conditional distribution \eqref{eq:picond} and, consequently, to obtain approximations of the stationary distribution $\pi$ with appealing properties. These approximations: 
$(i)$ bound $\pi$; 
$(ii)$ converge to $\pi$ as the truncation approaches the entire state space; 
$(iii)$ are accompanied by practical error bounds; 
and $(iv)$ are cheap to compute. 
For chains whose conditional distributions cannot be computed, more sophisticated approximation methods are necessary.  We study five such methods, each of which retains some, but not all, of the aforementioned properties as summarised in Table~\ref{tab:1}. The schemes are pictorially described in Figure~\ref{fig:cartoon}(b--f).

The first of these schemes (Section~\ref{sec:qbd}), is tailored to the multi-dimensional generalisation of birth-death processes: so-called \emph{level-dependent quasi-birth-death processes} (\emph{LDQBDPs}) whose state space decomposes into a union $\cup_{l=0}^\infty\cal{L}_l$ of disjoint sets known as \emph{levels}. Each level is accessible in a single jump  from only those adjacent to it. 
In practice, the scheme consists of setting the truncation $\s_r$ to be the first $L_r$ levels and inverting a $|\L_l|\times|\L_l|$ matrix per level $\L_l$ included therein {(we  use $\mmag{A}$ to  denote the cardinally of a set $A\subseteq \s$).}

\begin{figure}[h!]
	\begin{center}
	\includegraphics[width=0.9\textwidth]{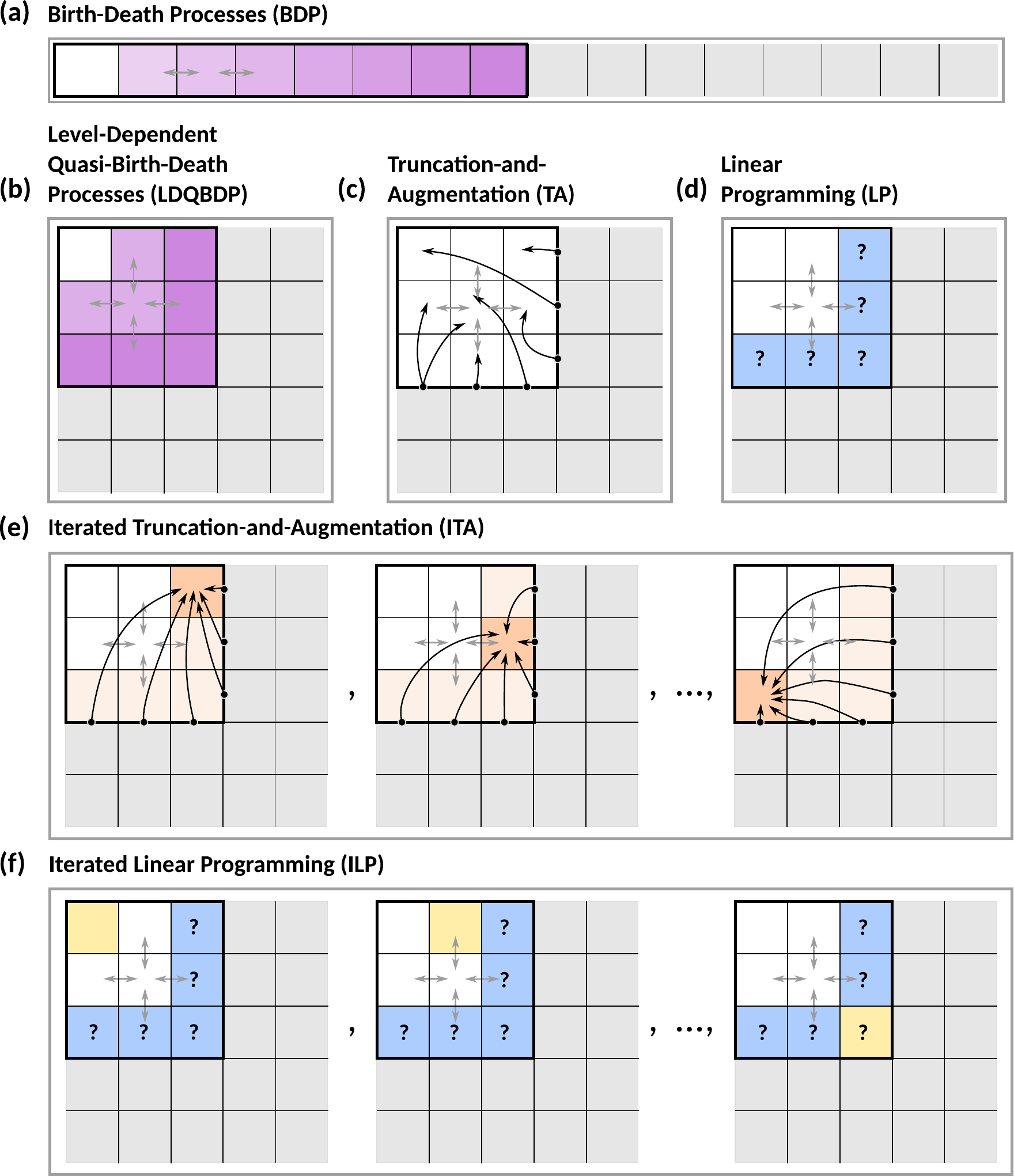}
	\vspace{-5pt}
	\end{center}
\caption{\textbf{Truncation-based schemes at a glance.} \textbf{(a)} The BDP scheme applies to one-dimensional processes that transition only between neighbouring states (grey arrows). In this special case the optimal approximating distribution (see Section \ref{sec:bd}) for the truncation (black box) can be computed exactly. \textbf{(b)} The LDQBDP scheme considers chains that transition between levels (purple shades) and exploits this structure to produce an approximating distribution on the first few levels (black box). 
\textbf{(c)} The TA scheme approximates the stationary distribution with that of an auxiliary chain that mimics the behaviour of the original chain except that whenever it would jump out of the truncation, the auxiliary chain is instead redirected to specified states inside the truncation (black arrows). \textbf{(d)} The LP scheme leaves the states in the in-boundary (blue) as free variables and finds a solution of the stationary equations that do not feature states outside of the truncation.
\textbf{(e)} The ITA scheme optimises over all re-entry states (dark orange) belonging to the in-boundary (light orange) to produce upper and lower bounds on each state in the truncation. \textbf{(f)} The ILP scheme produces upper and lower bounds on the probability of each state  individually (yellow) by optimising over the free variables (blue).}\label{fig:cartoon}
\vspace{-20pt}
\end{figure}

The \emph{truncation-and-augmentation (TA)} scheme of Section~\ref{sec:taa} modifies the chain so that it never exits a given truncation $\s_r$ and uses the finite-dimensional stationary distribution of the modified chain to approximate that of the original chain. Computationally, the scheme entails solving a system of $\mmag{\s_r}$ linear equations in $\mmag{\s_r}$ unknowns.

The \emph{iterated TA} scheme (\emph{ITA}, Section~\ref{sec:iter})   repeatedly applies the TA scheme to obtain upper and lower bounds on the distribution. In practice, this scheme consists of solving $\mmag{\cal{B}_i(\s_r)}$ systems of $\mmag{\s_r}$ linear equations in $\mmag{\s_r}$ unknowns, where $\cal{B}_i(\s_r)$ denotes the truncation's in-boundary (set of states inside the truncation accessible in a single jump from outside).

The \emph{linear programming} scheme (\emph{LP}) in Section~\ref{lpapp} instead constructs tractable approximations of the set of stationary solutions of the CME and optimises over these. The scheme has strong convergence guarantees and is applicable in the non-unique case. Running the scheme consists of solving a linear program with $\mmag{\s_r}$ decision variables and a comparable number of constraints.

The iterated variant of the LP scheme, the \emph{ILP} scheme (Section~\ref{lpappiter}), produces bounds on the distributions and doubles up as a uniqueness test. It consists of solving multiple linear programs with $\mmag{\s_r}$ decision variables and a comparable number of constraints. Specifically, to approximate the entire distribution $\mmag{\s_r}$ linear programs are required, for a marginal distribution the number of programs equals the number of marginal states in the truncation, and for a single average it equals one.

\begin{minipage}{1.3\textwidth} 
\captionof{table}{Summary of the truncation-based schemes covered in this review.} \label{tab:1} 
{\small\sffamily
\begin{tabularx}{\linewidth}{|L | L L L L L L|}\hline
Scheme& BDP (Section~\ref{sec:bd}) & LDQBDP (Section~\ref{sec:qbd})&TA (Section~\ref{sec:taa})& ITA (Section~\ref{sec:iter})& LP (Section~\ref{lpapp})& ILP (Section~\ref{lpappiter})\\\hline
\rowcolor{lightgray}
Approximation type&Bounds&Approximating distribution&Approximating distribution&Bounds&Approximation&Bounds\\\hline
Convergence guarantee&\cmark, in total variation&\cmark, in total variation&$^{(1)}$&$^{(2)}$&\cmark, {$w$-weakly*$^{(3)}$}&\cmark$^{(4)}$, {$w$-weakly*$^{(3)}$}\\\hline\rowcolor{lightgray}
Computable error bound&\cmark$^{(5)}$&\xmark&\cmark$^{(6)}$&\cmark&\xmark&\cmark\\\hline
Uniqueness required?&\cmark&\cmark& \cmark&\cmark&\xmark&\xmark\\\hline\rowcolor{lightgray}
Other requirements&BDP&LDQBDP&None&Tail bound&Moment bound&Moment bound\\\hline
Computational cost$^{(7)}$&Trivial&Low&Low to medium&High&Medium&Medium to high$^{(8)}$\\\hline\rowcolor{lightgray}
Type of computation&Recursion&Linear algebra&Linear algebra&Linear algebra&Linear programming&Linear programming\\\hline
\end{tabularx}}
\smallskip

\small
$^{(1)}$Only known to converge in total variation for irreducible exponentially ergodic chains under certain conditions on the  re-entry matrix. Counterexamples for which the scheme does not converge are known (see Section~\ref{sec:taa}). $^{(2)}$Only the upper bounds are known to converge (pointwise) under the same conditions as the TA scheme. No counterexamples are known (see Section~\ref{sec:iter}). {$^{(3)}$Where $w$ is the function featuring in the moment bound~\eqref{eq:mombound}.}  $^{(4)}$Guaranteed convergence if the stationary distribution is unique. $^{(5)}$Requires a tail bound (see Section~\ref{sec:bd}). $^{(6)}$Requires a Lyapunov function (see Section~\ref{sec:taa}). $^{(7)}$Based on our practical experience using non-optimised MATLAB-based implementations of each of the methods, see Section~\ref{togglesec} for details. $^{(8)}$The cost of the ILP scheme depends on what is approximated: the entire distribution (high), a marginal (medium to high), or just an average (medium).
\bigskip
\end{minipage}

\subsection{Approximations for birth-death processes}\label{sec:bd}

To illustrate the basic properties of truncation-based schemes, we consider a birth-death process of the form 
\begin{equation}
\label{eq:one_step}
\varnothing \underset{a_-}{\overset{a_+}{\rightleftarrows}}S.
\end{equation}
This simple SRN~\eqref{eq:network} has rate matrix \eqref{eq:qmatrixsrn} and state space $\s=\n$. Its state increases by one with \emph{birth rate} $a_+$ and decreases by one with \emph{death rate} $a_-$. The stationary equations \eqref{eq:stat}  read
\begin{align}
&a_-(1)\pi(1)-a_+(0)\pi(0)=0\label{eq:jf987wahaf3wi1}\\
&a_-(x+1)\pi(x+1)-(a_+(x)+a_-(x))\pi(x)+a_+(x-1)\pi(x-1)=0\enskip\forall x\in\zp.\label{eq:jf987wahaf3wi2}
\end{align}
Assuming as we do throughout that $a_-(x)>0$ for all $x$ in $\zp$, a sequence $\pi$ in $\ell^1$ satisfies these equations if and only if 
\begin{equation}
\label{eq:1deg}
\pi(x) =  \pi(0)\left [ \frac{a_+(0)}{a_-(1)}\frac{a_+(1)}{a_-(2)}\dots\frac{a_+(x-1)}{a_-(x)} \right ] =:   \pi(0)\gamma(x),  \quad \forall x \in\zp,
\end{equation}
see \cite[Chapter~7.1]{Gardiner2009}. Such a sequence is a probability distribution if and only if it satisfies the normalising condition
\begin{equation}
 \label{eq:normalising}
\pi(0) = \frac{1}{\sum_{x=0}^\infty \gamma(x)}=:\frac{1}{\gamma(\s)},\end{equation}
where $\gamma(0):=1$. In this case, Theorem~\ref{Qstateq} and Corollary~\ref{iredcor} show that $\pi$ is the unique stationary distribution of the chain, as long as the rate matrix is regular.

In most cases, no closed-form expression is known for the normalising constant $\gamma(\s)$ and, consequently, it is not possible to compute $\pi$ exactly. Instead, let $\s_r$ denote the truncation of the state space $\n$ consisting of the first $r$ states:
$$\s_r:=\{0,1,\dots,r-1\}\quad\forall r\in\zp.$$ 
Dividing both sides of \eqref{eq:1deg} by $\pi(\s_r)$, we find that the (zero-padded) conditional distribution $\pi(\cdot|\s_r)$ (defined in~\eqref{eq:picond}) satisfies
\begin{equation}\label{eq:onestepcond}\pi(x|\s_r)=\gamma(x)\pi(0|\s_r)\quad \forall x\in\s_r.\end{equation}
Combining the above with the normalising condition $\pi(\s_r|\s_r)=1$ yields 
\begin{equation}\label{eq:normalisingcond}\pi(0|\s_r)=\frac{1}{\sum_{x=0}^{r-1} \gamma(x)}=:\frac{1}{\gamma(\s_r)}.\end{equation} 
Note that $\pi(\cdot|\s_r)$ is easy to compute because $\gamma(\s_r)$ is a finite sum.

The definition of the conditional distribution in \eqref{eq:picond} implies that it bounds $\pi$ from above in $\s_r$,
\begin{equation}\label{eq:fenaw8fney8af}\pi(x|\s_r)\geq\pi(x)\quad\forall x\in\s_r,\end{equation}
and we denote it by $u_r$ in what follows. For the reasons given in Section~\ref{sec:opt}, the conditional distribution is optimal among approximating distributions and its maximum relative and total variation errors both equal the truncation error $m_r$:
$$\norm{{u_r}-\pi}_{TV}=\max_{x\in\s_r}\frac{\mmag{u_r(x)-\pi(x)}}{u_r(x)}=\sum_{x\in{\s_r}}\pi(x)=m_r.$$
If an upper bound $c$ on the mean of $\pi$ is known, Markov's inequality yields a practical bound on  the tail mass ($m_r\leq c/r$) and, consequently, one approximation error too:
\begin{equation}\label{eq:faw97e7wea}\norm{u_r-\pi}_{TV}=\max_{x\in\s_r}\frac{\mmag{u_r(x)-\pi(x)}}{u_r(x)}\leq\frac{c}{r}.\end{equation}
Armed with the tail bound, we also easily obtain lower bounds on $\pi$:
$$l_r(x):=\alpha_r u_r(x)=\alpha_r\frac{\pi(x)}{\pi(\s_r)}=\frac{\alpha_r}{\pi(\s_r)}\pi(x)\leq \pi(x)\quad\forall x\in\s_r,\enskip r>c,$$
where $\alpha_r:=(1-\frac{c}{r})\geq 1-m_r=\pi(\s_r)$. Because
{\begin{align*}
\frac{\mmag{l_r(x)-\pi(x)}}{l_r(x)}&=\frac{\mmag{(\alpha_r/\pi(\s_r))\pi(x)-\pi(x)}}{(\alpha_r/\pi(\s_r))\pi(x)}=\mmag{1-\frac{\pi(\s_r)}{\alpha_r}}=\frac{\pi(\s_r)-\alpha_r}{\alpha_r}\leq\frac{c}{r\alpha_r}
\end{align*}
for all $x$ in $\s_r$ and $r>c$,} 
the maximum relative error of the (zero-padded) lower bounds $l_r=(l_r(x))_{x\in\s}$ is bounded by $c/(r\alpha_r)$, while \eqref{eq:trunerrexp2} tells us that the total variation error is the tail bound:
\begin{equation}\label{eq:faw97e7wea2}\norm{l_r-\pi}_{TV}= {1-l_r(\s_r)=1-\alpha_ru_r(\s_r)=}\frac{c}{r}.\end{equation}
Taking the limit $r\to\infty$ in \eqref{eq:faw97e7wea}--\eqref{eq:faw97e7wea2} shows that both the upper bounds $u_r$ and lower bounds $l_r$ converge to $\pi$ in total variation as the truncation approaches the entire state space.

{The reason why the birth-death case is straightforward is that we are able to compute the conditional distribution $\pi(\cdot|\s_r)$.  
Indeed, notice that the analysis starting at \eqref{eq:fenaw8fney8af} and ending underneath \eqref{eq:faw97e7wea2} holds identically for any chain (birth-death or otherwise) with stationary distribution $\pi$, truncation $\s_r$, conditional distribution $u_r(\cdot)=\pi(\cdot|\s_r)$, and tail bound $m_r\leq c/r$. 
In general, obtaining the conditional distribution is non-trivial:
while it is possible to compute this distribution for certain other chains, e.g.\ those whose stationary distribution is known up to a normalising constant (Section~\ref{sec:intro}) or those with known conditional re-entry matrix (Section~\ref{sec:opt}), these are exceptional cases. For most chains the approximation problem proves challenging.}

\subsection{Approximations for level-dependent quasi-birth-death processes}\label{sec:qbd}

Specialised schemes for \emph{level-dependent quasi-birth-death processes} (LDQBDPs) have attracted significant attention (see \cite{Baumann2010,Phung-Duc2010,Bright1995,Dayar2011a,Hanschke1999,Li2004} and references therein). {\emph{Quasi-birth-death processes} (QBDPs) generalise birth-death processes by allowing \emph{block} tridiagonal rate matrices (instead of tridiagonal), and were first considered in~\cite{Evans1967,Wallace1969}.}
In particular, the state space $\s$ of these processes decomposes into a disjoint union $\cup_{l=0}^\infty\L_l$ of finite sets $\L_l$ known as \emph{levels} (states within a level are sometimes referred to as \emph{phases}~\cite{Bright1995}) such that the rate matrix 
\begin{equation}
\label{eq:ldqbd}Q=\begin{pmatrix}
Q^0&Q^0_+&0&0&\dots\\
Q^1_-&Q^1&Q^1_+&0&\dots\\
0&Q^2_-&Q^2&Q^2_+&\dots\\
0&0&Q^3_-&Q^3&\dots\\
\vdots&\vdots&\vdots&\vdots&\ddots
\end{pmatrix},
\end{equation}
where the block $Q^l_-=(q(x,y))_{x\in\L_{l},y\in\L_{l-1}}$ (resp. $Q^l_+=(q(x,y))_{x\in\L_{l},y\in\L_{l+1}}$) describes the transitions from the states in level $\cal{L}_l$ to the states in the level below (resp. above) and $Q^l=(q(x,y))_{x,y\in\L_l}$ describes the transitions between states inside level $\cal{L}_l$. The $0$'s denote matrices of zeros of appropriate sizes.  
{The early literature~\cite{Evans1967,Wallace1969,Neuts1981,Latouche1993,Hajek1982} focused on \emph{level-independent} QBDPs for which blocks $Q^l_-,Q^l,$ and $Q^l_+$ are independent of the level number $l$. LDQBDPs are \emph{level-dependent} QBDPs for which the blocks depend on $l$.}

\juansec{Classes of stochastic reaction networks modelled by LDQBDPs}
As pointed out in \cite{Dayar2011a}, SRNs whose stoichiometric vectors $\nu_1,\dots, \nu_m$ are composed of entirely ones, zeros, and minus ones,
$$\nu_{ij}\in\{-1,0,1\},\quad\forall i=1,\dots,n,\enskip j=1,\dots,m,$$
are LDQBDPs.
Their $l^{th}$ level consist of all count vectors with at least one entry equal to $l$ and no entry greater than $l$:
\begin{equation}\label{eq:masslevelsspieler}\L_l:=\{x\in\s:\max\{x_1,x_2,\dots,x_n\}=l\}\quad\forall l\in\n.\end{equation}
Examples include networks with reactions such as $2S_1\to S_1$, $2S_1+S_2\to 3S_1$, and $S_1+S_2\to \varnothing$ whereas networks with reactions such as $2S_1\to S_2$ or $\emptyset\to 2S_1$ fall outside of this class.

Here, we identify a second class of LDQBDPs: SRNs with well-defined notions of \emph{total mass} and reactions that change this mass by at most one. More concretely, SRNs for which there exists a vector $u=(u_1,\dots,u_n)$ of positive integers such that 
$$\iprod{u}{\nu_j}:=u_1\nu_{1j}+\dots+u_n\nu_{nj}\in\{-1,0,1\}$$
for all stoichiometric vectors $\nu_{j}$~\eqref{eq:qmatrixsrn}, i.e.\ all $j=1,\dots,m$.
We refer to the quantity $u_i$ as the \emph{mass} $S_i$'s molecules and thus to $\iprod{u}{X_t}$ as the \emph{total mass} in the network at time $t$. A reaction \emph{consumes} mass if $\langle u, \nu_j\rangle$ is negative, \emph{produces} mass if it is positive, and \emph{conserves} mass if it is zero. For instance, choosing $u=(1,2)$ for the network
\begin{equation}\label{eq:qbdex}\varnothing\rightleftarrows S_1,\qquad 2S_1\rightleftarrows S_2,\end{equation}
we have that a molecule of $S_2$ has twice as much mass as  a molecule of $S_1$ does and the first reaction produces mass, the second consumes mass, and the third and fourth conserve mass. For these types of networks, the chain is an LDQBDP whose $l^{th}$ level is the set of states with mass $l$:
\begin{equation}\label{eq:masslevels}\L_l:=\{x\in\s:\iprod{u}{x}=l\}\quad\forall l\in\n.\end{equation}
The aforementioned classes overlap, but neither is a subclass of the other: the network
$$\emptyset\xrightarrow{} S_1+S_2,\quad S_1\rightleftarrows S_2,\quad S_1\xrightarrow{}\varnothing,\quad S_2\xrightarrow{}\varnothing,$$
has levels of type \eqref{eq:masslevelsspieler} but not of type \eqref{eq:masslevels}, while the network in \eqref{eq:qbdex} has levels of type \eqref{eq:masslevels} but not of type \eqref{eq:masslevelsspieler}.

More generally, the LDQBDP property can be deduced from the network's stoichiometry. 
Let $f$ be a $\n$-valued norm-like function. If 
\begin{equation}\label{eq:fha78w3a7whrfn3awjf}{f(x+\nu_j)-f(x)}\in\{-1,0,1\}\qquad\forall x\in\s:a_j(x)>0, \enskip \forall j=1,\dots,m,\end{equation}
then the chain is an LDQBDP with levels
$$\L_l:=\{x\in\s:{f(x)}=l\}\quad\forall l\in\n.$$
For instance, we had ${f(x)}=\max\{x_1,\dots,x_n\}$ for the first class of networks above and ${f(x)}=\iprod{u}{x}$ for the second. This condition is not only sufficient but necessary too as setting ${f(x)}:=l$ for all $x$ in $\cal{L}_l$ and $l$ in $\n$ yields an $\n$-valued norm-like function on $\s$ satisfying \eqref{eq:fha78w3a7whrfn3awjf}.

\juansec{Approximating the stationary distribution}
Throughout the remainder of this section, suppose that {the rate  matrix  is regular} and that the chain is irreducible and has a stationary distribution $\pi$. Let $\pi(\cdot|\s_r)$ denote the conditional distribution in \eqref{eq:picond} with respect to the truncation
\begin{equation}\label{eq:qbdtrunc}\s_r:=\bigcup_{l=0}^{L_r-1}\L_l\end{equation}
obtained by discarding all but the first $L_r$ levels.

Similarly to the birth-death case, the conditional distribution can be characterised as follows \cite{Bright1995}. The restriction $\pi_{|l}(\cdot|\s_r):=(\pi(x|\s_r))_{x\in\L_l}$ of the conditional distribution to the $l^{th}$ level $\L_l$ can be expressed in terms of the restriction to the $0^{th}$ level:
\begin{equation}\label{eq:ndua93a3}
\pi_{|l}(x|\s_r)=\sum_{x'\in \L_0}\pi_{|0}(x'|\s_r)\Gamma^l(x',x)\quad\forall x\in\L_l,\enskip l=1,2,\dots,L_r-1,
\end{equation}
where $\Gamma^0$ denotes the identity matrix $(1_x(y))_{x,y\in\L_0}$ on $\L_0$,
$$\Gamma^l:=R^1R^2\dots R^l\quad\forall l\in\zp,$$
and the matrices $R^l=(R^l(x,y))_{x\in\L_{l-1},y\in\L_{l}}$ with dimension $|\L_{l-1}|\times|\L_{l}|$ are the minimal non-negative solutions to the equations
\begin{equation}\label{eq:rmeqs}Q_+^{l-1}+R^lQ^l+R^{l}R^{l+1}Q^{l+1}_-=0\quad\forall l\in\zp.\end{equation}
The restriction $\pi_{|0}(\cdot|\s_r)$ in \eqref{eq:ndua93a3} is the unique solution of the equations
\begin{align}
\sum_{x'\in\L_0} \pi_{|0}(x'|\s_r) \, \left( Q^0(x',x)+R^1Q^1_-(x',x) \right)&=0 \quad\forall x\in\L_0, \label{eq:ndua93a1}\\
\sum_{l=0}^{L_r-1} \sum_{x\in\L_0}\sum_{y\in\L_l}
\pi_{|0} (x|\s_r) \, \Gamma^l(x,y)&=1  \label{eq:ndua93a2}.
\end{align}
{Note that \eqref{eq:ndua93a3} generalises   \eqref{eq:onestepcond} to multiple dimensions; \eqref{eq:rmeqs}--\eqref{eq:ndua93a1} generalise the equations obtained by plugging~\eqref{eq:onestepcond} into \eqref{eq:jf987wahaf3wi1}--\eqref{eq:jf987wahaf3wi2}; and \eqref{eq:ndua93a2} generalises \eqref{eq:normalisingcond}.}
\juansec{The birth-death case} 
In the case of the birth-death process in Section~\ref{sec:bd}, the levels are individual states ($\L_l=\{l\}$) and the entries of the $1 \times 1$ blocks are:
\[ Q^l_-(l,l-1)=a_-(l) \enskip \forall l\in\zp, \enskip 
Q^l(l,l)=-a_-(l)-a_+(l),\enskip Q^l_+(l,l+1)=a_+(l) ,\enskip\forall l\in\n.
\]
Thus, the matrices $R^1,R^2,\dots$ essentially reduce to numbers. {By \eqref{eq:normalisingcond},  $\pi(0|\s_r)=\gamma(\s_r)^{-1}>0$ and \eqref{eq:ndua93a1} reduces to 
$$a_-(1) \, R^1(0)-a_+(0)=0.$$
Combining the above with \eqref{eq:rmeqs}, we find that 
$$R^l(l)=\frac{a_+(l-1)}{a_-(l)}\quad\forall l\in\zp.$$
Consequently, \eqref{eq:ndua93a3} reduces to \eqref{eq:onestepcond}; \eqref{eq:ndua93a2} reduces to \eqref{eq:normalisingcond}; \eqref{eq:rmeqs}--\eqref{eq:ndua93a1} are equivalent to \eqref{eq:jf987wahaf3wi1}--\eqref{eq:jf987wahaf3wi2}}; and we compute the  conditional distribution $\pi(\cdot|\s_r)$ as described in Section~\ref{sec:bd}.
\juansec{The general case and the LDQBDP scheme} 
In contrast with birth-death processes, the size of the levels of multidimensional LDQBDPs typically grows with $l$ (i.e.\ $|\L_0|<|\L_1|<\dots$), rendering the system~\eqref{eq:rmeqs} 
underdetermined, since we have $|\L_{l-1}|\times|\L_{l}|$ equations in $|\L_{l}|\times|\L_{l+1}|$ unknowns. For this reason, we are no longer able to compute $R^{l+1}$ from $R^l$. Moreover, we are unable to solve for $R^1$ since Eqs.~\eqref{eq:ndua93a1} are also underdetermined. 

Given $R^{l+1}$, Eqs.~\eqref{eq:rmeqs} do have a unique solution~\cite{Bright1995} for $R^l$, 
\begin{equation}\label{eq:fenw7a8fhewa87fna21eq}
R^l=-Q^{l-1}_+(Q^l+R^{l+1}Q^{l+1}_-)^{-1}.
\end{equation}
Thus, were $R^{L_r}$ to be known, we could compute $R^{L_r-1},\dots, R^1$ `downwards' using \eqref{eq:fenw7a8fhewa87fna21eq} (or another equivalent equation \cite{Bright1995}). 
However, in practice, $R^{L_r}$ is unknown and we must instead settle for approximations thereof: \cite{Baumann2010,Phung-Duc2010} propose using a matrix of zeros as the approximation of $R^{L_r}$, while  \cite{Bright1995,Dayar2011a} consider more refined approximations. Approximations $R^{1}_r,\dots,R^{L_r-1}_r$ of $R^{1},\dots,R^{L_r-1}$ are then obtained using \eqref{eq:fenw7a8fhewa87fna21eq} and approximations $(\pi_r(x))_{x\in\s_r}$ of the conditional distribution $(\pi(x|\s_r))_{x\in\s_r}$ are obtained {by solving \eqref{eq:ndua93a1}--\eqref{eq:ndua93a2} and applying~\eqref{eq:ndua93a3}} with $R^{1}_r$ replacing $R^{1}$ and $\Gamma^l_r:=R^{1}_r\dots R^l_r$ replacing $\Gamma^l$. The stationary distribution $\pi$ is then approximated using the zero-padded version of $\pi_r$~\eqref{eq:padding}.

\juansec{Convergence of the scheme and approximation error} 

If the sequence of truncations $(\s_r)_{r\in\zp}$ approaches the entire state space (i.e.\  $L_r\to\infty$ as $r\to\infty$) and, for each $l$ in $\zp$, the sequence $(R^l_r)_{r\in\zp}$ is increasing and has pointwise limit $R^l$  (as is the case in \cite{Baumann2010,Phung-Duc2010,Bright1995,Dayar2011a}), then the sequence of approximations $(\pi_r)_{r\in\zp}$ converges to $\pi$ in total variation, see Appendix~\ref{app:qbdconv} for a proof {(we have been unable to locate such a proof elsewhere)}. 
Except for special cases \cite{Bright1995}, it remains to be shown how to compute or bound the error of these approximations. The articles \cite{Bright1995,Phung-Duc2010,Baumann2010,Dayar2011a} employ several measures to estimate the error. However, these measures are \emph{local} in the sense that they do not account for the chain's behaviour outside of the truncation and, for this reason, can be unreliable indicators of the error, see the Section~\ref{sec:conclusion} for more on this.

\subsection{Truncation-and-augmentation}\label{sec:taa}
The \emph{truncation-and-augmentation} (TA) scheme was originally considered by  E.~Seneta\footnote{Seneta~\cite[p.242]{Seneta2006} states that the idea of `stochasticizing truncations of an infinite stochastic matrix' underpinning the TA scheme was `suggested by Sarymsakov~\cite{Sarymsakov1945} and used for other purposes'. To go from the cofactors-of-truncated-matrices formulation in \cite{Seneta1967,Seneta1968,Tweedie1971,Tweedie1973} to the formulation given here use \eqref{eq:truncatedinverse} and its discrete-time analogue \cite[p.229]{Seneta2006}.} \cite{Seneta1967,Seneta1968} for discrete-time chains in the late 60s (see also \cite{Seneta1980,Wolf1980,Heyman1991,Gibson1987a,Gibson1987b,Tweedie1973,Tweedie1998,Li2000,Liu2010,Herve2014,Masuyama2015}). Here, we discuss its continuous-time counterpart first touched upon by R.~L.~Tweedie in the early 70s \cite{Tweedie1971,Tweedie1973} (see also \cite{Hart2012,Liu2015,Kazeev2015,Cao2008,Cao2016,Cao2016b,Masuyama2016,Masuyama2017,Masuyama2017a,Gupta2017,Liu2018,Liu2018a}). {In the context of SRNs~\eqref{eq:network}, special cases of the TA scheme have more recently been referred to as the \emph{finite buffer dCME method}~\cite{Cao2008}, the \emph{stationary finite state projection (FSP) algorithm}~\cite{Gupta2017}, and the \emph{reflecting FSP approach}~\cite{Kazeev2015}.}

The TA scheme applies to $\varphi$-irreducible  chains $X$ with unique stationary distribution $\pi$. It entails approximating $\pi$ with a stationary distribution of a second chain that takes values in a given truncation $\s_r$. In particular, we choose an $\mmag{\s_r}\times\mmag{\s_r}$ \emph{re-entry matrix} $e_r=(e_r(x,y))_{x,y\in\s_r}$ satisfying
{$$e_r(x,y)\geq0\quad\forall x\in\cal{B}_o(\s_r),\enskip y\in\s_r,\qquad\sum_{y\in\s_r}e_r(x,y)=1\quad\forall x\in\cal{B}_o(\s_r),$$}
and approximate $\pi$ using a stationary distribution of the chain $X^{e_r}$ with rate matrix 
\begin{equation}\label{eq:truncandaug2}q^{e_r}(x,y):=q(x,y)+q_o(x)e_r(x,y)\quad \forall x,y\in\s_r,\end{equation}
{where the out-rate $q_o$ and out-boundary $\cal{B}_o(\s_r)$ are as in~\eqref{eq:qo}.}

Analogously to the censored chain of Section~\ref{sec:opt}, $X^{e_r}$ behaves identically to $X$ while both remain inside of the truncation. However, whenever $X^{e_r}$ tries to leave the truncation, it is instead redirected to a state sampled from $e_r(x,\cdot)$, where $x$ denotes the position of $X^{e_r}$ right before this jump {($x$ must belong to $\cal{B}_o(\s_r)$ for otherwise $X$ would be unable to jump out of the truncation)}. Because $\s_r$ is finite, Theorems~\ref{Qstateq}, \ref{doeblinc}, and \ref{lyacom} (with $v=d_2=d_1=1$) imply that $X^{e_r}$ has at least one stationary distribution and that its stationary distributions are the solutions of 
\begin{align}\pi^{e_r}_rQ^{e_r}&=0,\label{eq:stateqmod1}\\
\pi^{e_r}_r(\s_r)&=1,\label{eq:stateqmod2}
\end{align}
i.e.\  the solutions of $\mmag{\s_r}+1$ linear equations in $\mmag{\s_r}$ unknowns.
\juansec{The birth-death case}
The one-step structure of the birth-death process introduced in Section~\ref{sec:bd} implies that the chain $X$ may only return to the truncation $\s_r:=\{0,1,\dots,r-1\}$ by visiting the boundary state $r-1$. For this reason, the conditional re-entry matrix $\varepsilon_r=(\varepsilon_r(x,y))_{x,y\in\s_r}$ in \eqref{eq:exre0} is given by
$$\varepsilon_r(x,y):=1_{r-1}(y)\quad\forall x,y\in\s_r.$$
With this choice of re-entry matrix ($e_r:=\varepsilon_r$), our approximating chain $X^{e_r}$ becomes the censored chain of Section~\ref{sec:opt}, and
its unique stationary distribution is the conditional distribution given by 
\eqref{eq:onestepcond}--\eqref{eq:normalisingcond}.

\juansec{The general case and the TA scheme}
For general $\varphi$-irreducible chains, it is not possible to compute the conditional re-entry matrix \eqref{eq:exre0}  and our approximating chain $X^{e_r}$ differs from the censored chain. {However, we can still compute one of its stationary distributions $\pi^{e_r}_r$ (by solving \eqref{eq:stateqmod1}--\eqref{eq:stateqmod2}); pad it with zeros~\eqref{eq:padding}; and use it as an approximation for $\pi$ (i.e.\ the TA scheme). 
The re-entry matrix is often chosen so that $X^{e_r}$ is $\varphi$-irreducible and \eqref{eq:stateqmod1}--\eqref{eq:stateqmod2} have a unique solution. 
In the non-unique case, it is unclear which solution should be chosen. However, in certain situations all solutions may approach $\pi$ for large enough truncations, and we are unsure whether this non-uniqueness truly limits the successful use of the TA scheme. 
Sufficient conditions for $X^{e_r}$ to be $\varphi$-irreducible include:}
\begin{enumerate}[label=(\alph*)]
\item if $X$ is irreducible and the re-entry location is chosen independently of the pre-exit location $x$: $e_r(x,y)=\alpha_r(y)$ for each {$x\in\cal{B}_o(\s_r)$ and $y\in\s_r$}, where $(\alpha_r(y))_{y\in\s_r}$ is a probability distribution; or
\item if $X$ is $\varphi$-irreducible and re-entry may occur via any state in the truncation, for all pre-exit locations: $e_r(x,y)>0$ for each {$x\in\cal{B}_o(\s_r)$ and $y\in\s_r$}.
\end{enumerate}
Note that if the re-entry matrix does not satisfy conditions $(a)$--$(b)$ above, $X^{e_r}$ may not be $\varphi$-irreducible even if $X$ is $\varphi$-irreducible and the re-entry location is independent of the pre-exit location (see Appendix~\ref{lossofuni} for an example). In practice, re-entry is often set to occur through a fixed state $z_r$ (i.e.\  $e_r(x,y)=1_{z_r}(y)$ $\forall x,y\in\s_r$) in which case we write $\pi_r^{z_r}$ instead of $\pi_r^{e_r}$.

\juansec{Choosing the re-entry matrix} 
Ideally, we would like to use the conditional re-entry matrix $\varepsilon_r=(\varepsilon_r(x,y))_{x,y\in\s_r}$ in \eqref{eq:exre0} as, in this case, the TA scheme yields the conditional distribution $\pi^{\varepsilon_r}_r$, optimal among approximating distributions (c.f.\ Section~\ref{sec:opt}). However, as mentioned before, this matrix is generally  unknown and we must instead settle for approximations $e_r$ thereof (note that an argument of the type given at the end of Appendix~\ref{app:qbdconv} shows that the approximation $\pi^{e_r}_r$ {is} close to $\pi^{\varepsilon_r}_r$ if the re-entry matrix $e_r$ close to $\varepsilon_r$). A starting point in choosing such an $e_r$ is only allowing re-entry through the states belonging to the \emph{in-boundary} $\cal{B}_i(\s_r)$, i.e.\  the collection of states through which $X$ itself can re-enter the truncation:
\begin{equation}\label{eq:insidebound}\cal{B}_i(\s_r):=\{y\in\cal{S}_r:\exists x\not\in\s_r,\enskip q(x,y)>0\}.\end{equation}
Better approximations of $\varepsilon_r$ can be obtained by running simulations or expressing $\varepsilon_r$ as an infinite sum and truncating this sum, see Appendix~\ref{app:condre} for more on the latter.

\juansec{Convergence of the scheme}
It is straightforward to find irreducible chains and re-entry matrices $e_r$ for which the TA approximations $\pi^{e_r}_r$ do not converge pointwise (e.g.\ the continuous-time version of \cite[(2.5)]{Wolf1980}). However, in the case of a fixed re-entry state $z$ independent of $r$ and an irreducible {exponentially} ergodic chain {with a regular rate matrix}, \cite[Theorem 3.3]{Hart2012} shows\footnote{This theorem's premise includes aperiodicity of the chain as a further requirement. However, we can omit it as all continuous-time Markov chains are aperiodic (e.g.\ aperiodicity follows from \cite[Theorem 3.2.1]{Norris1997}). Additionally, when reading the proof of this theorem it helps to remember that, if $Q$ is regular and there exists $v\geq 1,d_1,d_2>0$ and finite set $F$ satisfying
$$Qv(x)\leq d_11_F(x)-d_2v(x)\quad\forall x\in\s,$$
then there also exists (generally unknown) $\tilde{v}\geq1,\tilde{d}_1,\tilde{d}_2>0$ that satisfy the above with $F=\{z\}$, see \cite[Theorems 6.3 and 7.2]{Down1995}.}
that $\pi^z_r$ converges to the stationary distribution $\pi$ in total variation as $\s_r$ approaches the entire state space $\s$. 
These approximations are also known to converge for monotone chains \cite{Hart2012,Masuyama2017a}, some generalisations thereof \cite{Hart2012,Masuyama2017a}, and certain other special cases \cite{Masuyama2017,Masuyama2016,Grassmann1993}.

\juansec{The issue of error control}
Presently, the biggest drawback of the TA scheme is the difficulty in assessing the quality of its approximations. 
In \cite{Gupta2017}, the authors consider a single re-entry state $z$ {independent of $r$} and chains satisfying the Foster-Lyapunov criterion in Theorem~\ref{lyacom} for some known $v$. They propose using the so-called \emph{convergence factor}
\begin{equation}\label{eq:convergencefactor}F_r:=\left(v(z)+\max_{x\in\cal{B}_o(\s_r)}v(x)\right)O_r\end{equation}
to quantify the error, where $O_r$ denotes the \emph{outflow rate}
\begin{equation}\label{eq:outflow}O_r:=\sum_{x\in\s_r}\pi^{z}_r(x)q_o(x),\end{equation}
with  $q_o$  and $\cal{B}_o(\s_r)$ given by \eqref{eq:qo}. The rationale behind this suggestion is that, in {the regular} and exponentially ergodic case, the total variation error (and the $v$-norm error) is bounded above by the convergence factor times a constant $M$~\cite[Theorem~III.1(C)]{Gupta2017}:
\begin{equation}\label{eq:guptabound}\norm{\pi-\pi^{z}_r}_{TV}\leq\norm{\pi-\pi^{z}_r}_{v}\leq MF_r\quad\forall r\in\zp,\end{equation}
Unfortunately, the constant $M$ is generally unknown and, while the convergence factor is informative regarding the \emph{rate} of convergence,  {the values a Lyapunov function takes in a finite set can be modified as pleased} (see Appendix~\ref{app:cfprob}). Hence, $F_r$ is an unreliable measure of the error for a particular truncation $\s_r$. 

Recent efforts \cite{Liu2015,Masuyama2017,Masuyama2017a,Liu2018,Liu2018a} have been directed at obtaining computable error bounds. One of the simplest of these bounds applies to {single re-entry states $z_r$ (possibly dependent on $r$) and} irreducible chains satisfying the Foster-Lyapunov criterion in Theorem \ref{lyabia} for some known $v,f,d,F$:
\begin{align}
\norm{\pi-\pi_r^{z_r}}_{TV}&\leq \frac{1}{2}\left(1+\frac{1}{\inf_{x\in\s}f(x)}\right)\left(v(z_r)+\frac{d}{\beta \overline{\phi}^\beta_r}\right)O_r\leq \left(v(z_r)+\frac{d}{\beta \overline{\phi}^\beta_r}\right)O_r,
\label{eq:l1boundsfsp}
\end{align}
for all $r$, such that $F$ is contained in the truncation $\s_r$ and $\overline{\phi}^\beta_r$ is positive.
Here, $\beta$ is any positive constant and
\begin{equation}\label{eq:phir}\overline{\phi}^\beta_r:=\max_{x\in\s_r}\min_{y\in F}\phi^\beta_r(x,y),\qquad \phi^\beta_r:= \left(I- \frac{1}{\beta} \, Q_r \right)^{-1},\end{equation}
where $Q_r$ denotes the truncated rate matrix $(q(x,y))_{x,y\in\s_r}$, and $I$ denotes the identity matrix $(1_x(y))_{x,y\in\s_r}$. Note that $\overline{\phi}^\beta_r$ is known \cite{Liu2018} to be positive for all sufficiently large $r$. The bound~\eqref{eq:l1boundsfsp} follows from \cite[Corollary 2.3]{Liu2018} and the fact that $f\geq 1$ (as mandated by Theorem \ref{lyabia}).

The quality of the  bound~\eqref{eq:l1boundsfsp}  (and of other bounds presented in \cite{Liu2015,Masuyama2017,Masuyama2017a,Liu2018,Liu2018a}) depends on the $v,f,d,F$ available. Finding such functions and constants is often a formidable task in practice and, as we show in Example \ref{schloglfsp} below, the error bounds can be rather conservative.  Furthermore, the computation of the bounds is more expensive than that of the approximation because it requires a matrix inversion in~\eqref{eq:phir} (see \cite[Remark 2.6]{Masuyama2017} for advice on this matter). Note that $\beta$ is a free parameter to be chosen.
However, it is unclear how the $1/(\beta \overline{\phi}^\beta_r)$ term in~\eqref{eq:l1boundsfsp} varies with $\beta$ and, consequently, what $\beta$s yield tighter error bounds (see \cite[Section~4.2.3]{Masuyama2017} for further discussion). {
As pointed out by one of our anonymous referees, 
once $v,f,d,F$ satisfying~\eqref{eq:lyabia} have been found and $\beta$ has been chosen, one can use linear programming to systematically modify the Lyapunov function $v$ inside the finite set $F$ so that the bound in~\eqref{eq:l1boundsfsp} is tightened (see Appendix~\ref{app:lptighten} for details).}

\begin{example}\label{schloglfsp}
Consider the classic autocatalytic network proposed by Schl\"ogl~\cite{schlogl1972,Vellela2009} as a model for a chemical phase transition:
\begin{equation}
\label{eq:smdl}
2S\underset{a_2}{\overset{a_1}{\rightleftarrows}}3S,
\qquad \varnothing \underset{a_4}{\overset{a_3}{\rightleftarrows}}S. 
\end{equation}
The state space is $\mathbb{N}$ and we assume that the propensities follow mass-action kinetics:
\begin{align*} 
 a_1(x):=k_1 x (x-1),
\quad a_2(x):= k_2 x (x-1) (x-2),
\quad a_3(x):= k_3, 
\quad  a_4(x):=k_4x, 
\end{align*}
where $k_1, k_2, k_3, k_4>0$ are the rate constants. 

As shown in \cite[Appendix~B]{Kuntz2017}, the chain associated with \eqref{eq:smdl} is an irreducible exponentially ergodic birth-death process with a unique stationary distribution $\pi$. In this case, an explicit formula \cite[(69)]{Kuntz2017} for the normalising constant $\gamma(\s)$ in \eqref{eq:normalising} can be obtained:
\begin{equation}\label{eq:hypergeo}\gamma(\s)={_2H_2}\left(-\frac{c_1+1}{2},\frac{c_1-1}{2};-\frac{c_2+1}{2},\frac{c_2-1}{2};\frac{k_1}{k_2}\right),\end{equation}
where $c_1:=\sqrt{1-4 {k_3}/{k_1}}$ and $c_2:= \sqrt{1-4 {k_4}/{k_2}}$, and ${_2H_2}$ denotes the generalised hypergeometric function.

Using \eqref{eq:1deg}--\eqref{eq:normalising}~and~\eqref{eq:hypergeo}, it is straightforward to compute the total variation errors of the TA approximations, so as to benchmark the performance of the {refined version~\eqref{eq:refbounds} of the} error bounds~\eqref{eq:l1boundsfsp}.  Figure~\ref{fig:sfsp} shows the stationary distribution, total variation approximation errors $\norm{\pi-\pi^{z_r}_r}_{TV}$,  {and refined error bounds~\eqref{eq:refbounds}  obtained using three different Lyapunov functions: $v(x)=x,x^2,x^3$.} 
To compute the approximations, we used truncations consisting of the first $r$ states,
$\s_r:=\{0,1,\dots,r-1\}$ and re-entry states $z_r=0,r-1$. 

\begin{figure}[h!]
	\begin{center}
	\includegraphics[width=1\textwidth]{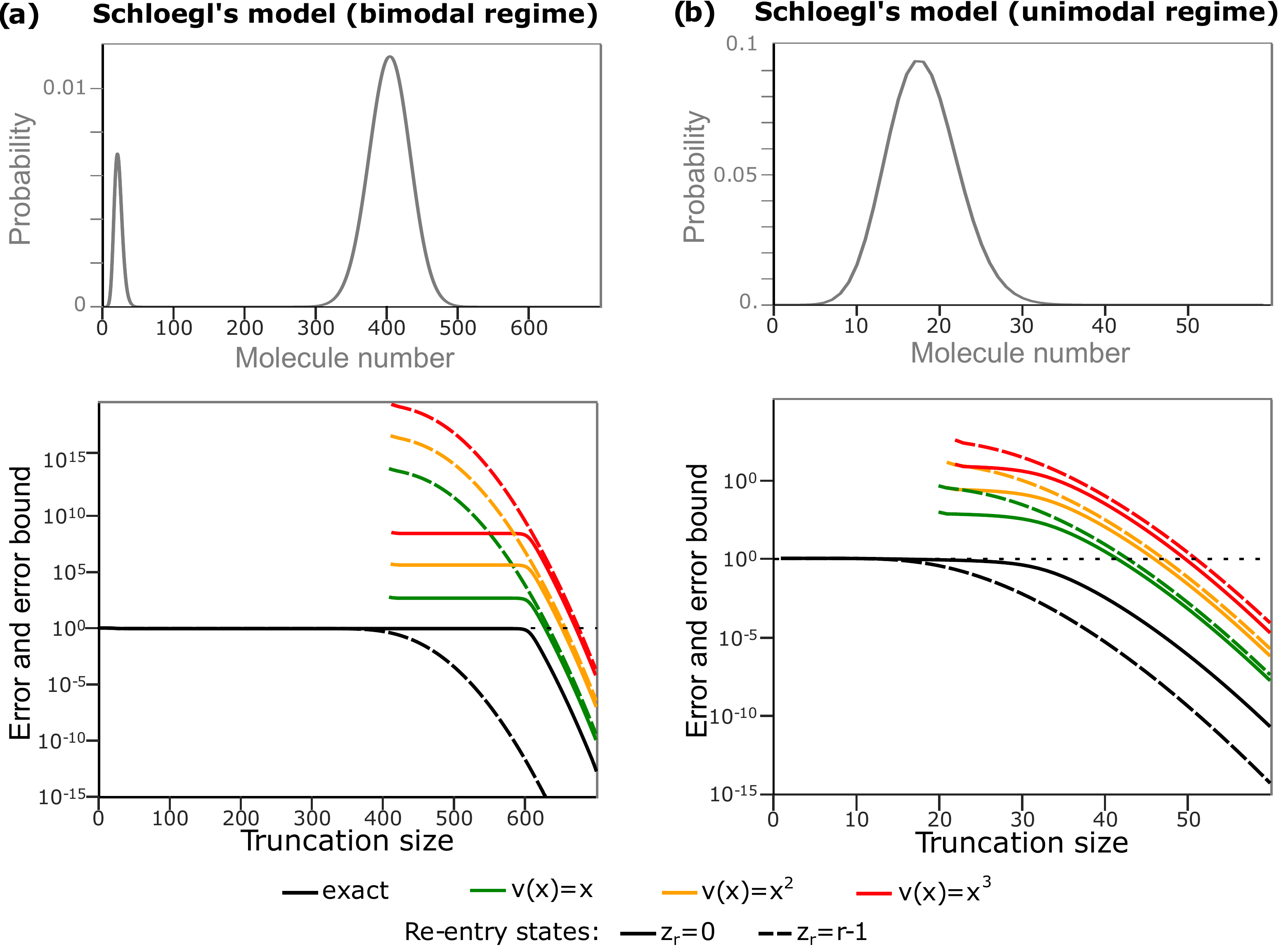}
	\end{center}
	\vspace{-5pt}
\caption{\textbf{Applying the TA scheme to Schl\"ogl's model~\eqref{eq:smdl}.} \textbf{(a)} Rate constants $k_1=0.025$, $k_2=4.17\times10^{-5}$, $k_3=60$, and $k_4=3.127$ lead to a bimodal stationary distribution (top plot) computed using \eqref{eq:1deg}--\eqref{eq:normalising} and \eqref{eq:hypergeo}. {Total variation approximation error $\norm{\pi-\pi^{z_r}_r}_{TV}$ of the TA scheme (black, bottom plot) and error bounds \eqref{eq:refbounds} (coloured,  bottom plot) as a function of $r$ with re-entry states $z_r=0$ (solid) and $z_r=r-1$ (dashed). In the Foster-Lyapunov criterion for the error bounds, we used $v(x)=x$ with $F=\{0,1,\dots,437\}$ (green), $v(x)=x^2$ with $F=\{0,1,\dots,439\}$ (yellow), and $v(x)=x^3$ with $F=\{0,1,\dots,441\}$ (red). We chose the free parameter $\beta$ in~\eqref{eq:phir} so as to obtain the best possible bounds (the optimal $\beta$ varied little with $r$: it was approximately $0.7$). \textbf{(b)} As in (a) but with rate constants $k_1=6$, $k_2=1/3$, $k_3=50$, and $k_4=3$ that lead to a unimodal stationary distribution. In this case, we used $v(x)=x$ with $F=\{0,1,\dots,19\}$ (green), $v(x)=x^2$ with $F=\{0,1,\dots,20\}$ (yellow), and $v(x)=x^3$ with $F=\{0,1,\dots,22\}$ (red) in the Foster-Lyapunov criterion, and we chose a $\beta$ that yields best bounds (as in the bimodal case, the optimal $\beta$ varied little with $r$: it was approximately $28$). }}\label{fig:sfsp}\vspace{-20pt}
\end{figure}

In  Figure~\ref{fig:sfsp}(a), we employ rate constants that lead to a bimodal stationary distribution with a small peak centred around $20$ molecules and a second larger peak centred around $440$ molecules. In Figure~\ref{fig:sfsp}(b) we use rate constants that lead to a unimodal stationary distribution with a peak centred around $20$ molecules. In both cases, we achieved a smaller error with re-entry state $z_r=r-1$ than with $z_r=0$ (up to $2\times10^{12}$  times smaller in (a) and up to $4\times10^3$ times in (b)). This stark difference is partly explained because for $z_r=r-1$ the TA scheme returns the conditional distribution, which is optimal in the sense of Section~\ref{sec:opt}. In contrast, the choice $z_r=0$ is thought \cite[Section~5]{Gibson1987b} to lead to the worst possible error because $0$ is the state furthest away from the boundary of the truncation in terms of travel time for the chain.
Similarly, the bounds~\eqref{eq:refbounds} proved to be far more conservative with the re-entry state $z_r=r-1$ than with $z_r=0$. {In Figure~\ref{fig:sfsp}(a), the bounds were greater than the error for $z_r=r-1$ by a factor of ${\sim}10^{15}$ (for $v(x){=}x$), ${\sim}10^{18}$ ( for $v(x){=}x^2$), and ${\sim}10^{21}$ (for $v(x){=}x^3$), whereas the bounds were greater than the error for $z_r=0$ by a factor of ${\sim}550$ (for $v(x){=}x$), ${\sim}10^{5}$ (for $v(x){=}x^2$), and ${\sim}10^{10}$ (for $v(x){=}x^3$), regardless of the truncation size $r$. In Figure~\ref{fig:sfsp}(b), this range narrowed: the bounds were ${\sim}10^6$ ($v(x){=}x$), ${\sim}10^8$ ($v(x){=}x^2$), and ${\sim}10^{10}$ ($v(x){=}x^3$) times greater than the error for $z_r=r-1$ and ${\sim}10^3$ ($v(x){=}x$), ${\sim}10^4$ ($v(x){=}x^2$), and ${\sim}10^{6}$ ($v(x){=}x^3$) times for $z_r=0$.} For large $r$, the bounds appeared to become almost independent of the re-entry state $z_r$ {and proportional to the error of the approximation obtained with $z_r=0$}. This could explain why the bounds are far more conservative for good re-entry choices, such as $z_r=r-1$, than for poor ones, such as $z_r=0$.

{The quality of the bounds deteriorated with the degree $n$ of the Lyapunov function $v(x)=x^n$ by $1$--$3$ orders of magnitude per degree. This could  be a consequence of higher degree polynomials $v$ and $Qv$ achieving higher values in $F$ (which was roughly the same set for all $n$, see the caption of Figure~\ref{fig:sfsp}) and inflating $v(z_r)$ and $c_r$ in~\eqref{eq:refbounds}. The deterioration with increasing $n$ seems independent of $r$: the shape of the error bound curves 
is similar, 
indicating that  
the dependence is dominated by the outflow rate $O_r$~\eqref{eq:outflow}.

The refined bounds in~\eqref{eq:refbounds} proved to be $1.5$--$3$ times tighter than those in~\eqref{eq:l1boundsfsp} obtained  with the naive choice $d:=\max_{x\in F}Qv(x)+1$. This is a significant practical boon, but not one that noticeably altered the semi-log plots in Figure~\ref{fig:sfsp}. On the other hand, choosing $\beta$ proved very influential, but non-trivial and expensive: both too small and too large $\beta$ values made $\beta\overline{\phi}_r^\beta$ arbitrarily small and each evaluation of $\beta\overline{\phi}_r^\beta$ required inverting a $r\times r$-dimensional matrix. Moreover, for certain parameter values (e.g.\ those in Figure~\ref{fig:sfsp}(b) with $r\geq 20$) the function $\beta\mapsto\beta\overline{\phi}_r^\beta$ was non-concave and had multiple local minima. Hence, a simple gradient ascent approach would not necessarily return a global maximum and we had to resort to evaluating $\beta\overline{\phi}_r^\beta$ for many $\beta$. This issue was ameliorated by the fact that $\overline{\phi}_r^\beta$ is known~\cite{Liu2018} to converge as $r\to\infty$ and this convergence occurred rapidly for our parameter sets. Hence, $\beta$ values that were optimal for some $r$ proved to be good candidates for other values of $r$.} 
\end{example}

\subsection{Iterated truncation-and-augmentation}\label{sec:iter}

The \emph{iterated truncation-and-augmentation (ITA)} builds on the work of Courtois and Semal \cite{Courtois1984,Courtois1985,Courtois1986,Semal1995} for the discrete-time case and that of Dayar, Spieler, \textit{et al} \cite{Dayar2011,Spieler2014} for the continuous-time one (see also \cite{Muntz1989,Lui1994,Mahevas2001} for related work by others) which showed that the stationary distribution $\pi$ can be bound by repeatedly applying the TA scheme. 
The key observation here is that, at least in the irreducible ergodic case, the conditional distribution \eqref{eq:picond} is a convex combination of the TA approximations $\pi^z_r$ with re-entry states $z$ belonging to the in-boundary $\cal{B}_i(\s_r)$ (c.f.\ \eqref{eq:insidebound}):
\begin{equation}\label{eq:condconv}\pi(\cdot|\s_r)=\sum_{z\in\cal{B}_i(\s_r)}\theta_z\pi^z_r(\cdot)\end{equation}
for some non-negative weights $(\theta_z)_{z\in\cal{B}_i(\s_r)}$ satisfying $\sum_{z\in\cal{B}_i(\s_r)}\theta_z=1$, see Appendix~\ref{app:condconvproof} for a proof. Due to \eqref{eq:condconv}, we obtain~\cite[Theorem 4]{Dayar2011} upper and lower bounds on the conditional distribution  $\pi(\cdot|\s_r)$ in \eqref{eq:picond} by exhaustively searching over the TA approximations $\pi^z_r$ with re-entry states $z$ belonging to the in-boundary:
\begin{equation}\label{eq:csbounds}\min_{z\in \cal{B}_i(\s_r)}\pi^z_r(x)\leq \pi(x|\s_r)\leq \max_{z\in \cal{B}_i(\s_r)}\pi^z_r(x)\quad\forall x\in\s_r{.}\end{equation}
%

{Because, by its definition in~\eqref{eq:picond}, the conditional distribution~$\pi(\cdot|\s_r)$ bounds  the stationary distribution~$\pi$, the upper bounds in~\eqref{eq:csbounds} also bound $\pi$.} To convert the {lower} bounds on the conditional distribution into {lower} bounds on the stationary distribution $\pi$, the authors of \cite{Dayar2011,Spieler2014} compute a tail bound using the Foster-Lyapunov criterion in Theorem~\ref{lyabia} as described in Remark \ref{dsapproach}.  In order to facilitate the comparison of the schemes' performances in Section~\ref{togglesec}, here we instead use tail bounds obtained with the moment bound approach of Section~\ref{sec:truncerr}. In particular, suppose that we have at our disposal a norm-like function $w$ and constant $c$ such that $\pi$ satisfies the moment bound~\eqref{eq:mombound} and let $\s_r$ be the $r^{th}$ sublevel~\eqref{eq:sublevel} set of $w$. In this case, the definition \eqref{eq:picond} of the conditional distribution, the tail bound \eqref{eq:tailbound}, and the conditional bounds \eqref{eq:csbounds} imply that %
\begin{equation}\label{eq:itaabounds}l_r(x):=\left(1-\frac{c}{r}\right)
\min_{z\in\cal{B}_i(\s_r)}\pi^z_r(x) \leq \pi(x)\leq \max_{z\in\cal{B}_i(\s_r)}\pi^z_r(x)=:u_r(x)\quad\forall x\in\s_r.\end{equation}
After padding these bounds with zeros \eqref{eq:padding}, the approximation error of $l_r$ can be computed using \eqref{eq:trunerrexp2} while that of $u_r$ can be bounded using \eqref{eq:uppererbound1}--\eqref{eq:uppererbound2}.

A useful observation is that the total variation and $\ell^1$ errors of the lower bounds are bounded below by the tail bound:
\begin{align*}
\norm{\pi-l_r}_{TV}=\norm{\pi-l_r}_{{1}}&=1-l_r(\s_r)=1-\left(1-\frac{c}{r}\right)\sum_{x\in\s_r}\min_{z\in\cal{B}_i(\s_r)}\pi^z_r(x)\\
&\geq 1-\left(1-\frac{c}{r}\right)\min_{z\in\cal{B}_i(\s_r)}\pi^z_r(\s_r)= \frac{c}{r}\quad \forall r\geq c.\end{align*}
As we will see in Section~\ref{togglesec}, the accuracy {of the }upper bounds is not limited by the tail bound and, consequently, the upper bounds outperform the lower ones for sufficiently large truncations.

\juansec{Bounding stationary averages} 
In applications, we are often only interested {in} one, or several, stationary averages of the form $\pi(f)$ instead of the full distribution, where $f$ is a given real-valued function on $\s$. In this case, it follows from \eqref{eq:condconv} that
\begin{equation}\label{eq:boundsitaf}l_r^f\leq \pi_{|r}(f)\leq u_r^f\quad\forall r\geq c,\end{equation}
where $\pi_{|r}$ denotes the restriction of $\pi$ to the truncation \eqref{eq:rest} and
\begin{align}l_r^f&:=\min
\left\{\min_{z\in\cal{B}_i(\s_r)}\pi^z_r(f), \, \left(1-\frac{c}{r}\right) \min_{z\in\cal{B}_i(\s_r)}\pi^z_r(f)
\right\},\nonumber\\
 u_r^f&:=\max
\left\{\max_{z\in\cal{B}_i(\s_r)}\pi^z_r(f), \, \left(1-\frac{c}{r}\right)\max_{z\in\cal{B}_i(\s_r)}\pi^z_r(f)
 \right\}\label{eq:lfuf},\end{align}
see Appendix~\ref{app:boundsitafproof} for details. Using \eqref{eq:boundsitaf} and an argument of the sort in the proof of \cite[Theorem~15]{Kuntz2017}, we obtain the following bounds on the stationary average:

\begin{enumerate}[label=(\roman*)]

\item If $f$ is non-negative  outside the truncation (i.e.\  $f(x)\geq0$ for $x\not\in\s_r$), then 
\begin{equation}\label{eq:fboundsp}l_r^f\leq \pi(f).\end{equation}
\item If $f$ is non-positive outside the truncation (i.e.\  $f(x)\leq0$ for $x\not\in\s_r$), then
\begin{equation}\label{eq:fboundsn}\pi(f)\leq u_r^f.\end{equation}
\item If $f$ is $\pi$-integrable and {the rate of growth of $f$ is at most proportional to that of $w$ (i.e.\  $\sup_{x\not\in\s_1}|f(x)|/w(x)<\infty$)},
then
\begin{equation}\label{eq:fbounds}
l_r^f-c \, \sup_{x\not\in\s_r}\frac{\mmag{f(x)}}{w(x)} 
\leq \pi(f)\leq 
u_r^f+c \, \sup_{x\not\in\s_r}\frac{\mmag{f(x)}}{w(x)}.
\end{equation}
\end{enumerate}
In summary, we use the bounds in \eqref{eq:fboundsp}--\eqref{eq:fbounds} as approximations of the stationary average $\pi(f)$. {By computing both a lower bound $L$ and an upper  bound $U$ on $\pi(f)$, we constrain} the approximation error:  {$L$} and {$U$} are no further than {$U-L$} away from $\pi(f)$ (similarly, the midpoint {$(L+U)/2$} of the bounds is no further than {$(U-L)/2$} away from $\pi(f)$). {We should point out here that, as long as $\pi(w)\leq c$, the bounds \eqref{eq:fboundsp}--\eqref{eq:fbounds} hold for \emph{any} $l^f_r, u^f_r$ satisfying~\eqref{eq:boundsitaf} and not just \eqref{eq:lfuf} computed using the ITA scheme (indeed, these bounds were originally introduced in \cite{Kuntz2017} for the scheme discussed in Section~\ref{lpappiter}).}
\juansec{Bounding marginal distributions}
In the case of high-dimensional state spaces, we are often interested in one or more marginal distributions rather than the full stationary distribution $\pi$. A marginal distribution is defined with respect to a partition $\{A_i\}_{i\in\cal{I}}$ of the state space:
\[
\bigcup_{i\in\cal{I}} A_i=\s, 
\quad  A_i \cap A_j=\emptyset\quad \forall i \neq j \in \cal{I}.
\]
The corresponding  \emph{marginal} $\hat{\pi}$ is the probability distribution on the indexing set $\cal{I}$ defined by
\begin{equation}\label{eq:marginal2}\hat{\pi}(i):=\pi(A_i),\quad \forall i\in\cal{I}.\end{equation}
For instance, in the case of an SRN~\eqref{eq:network} with state space $\nn$, $\hat{\pi}$ may be the distribution describing the molecule counts of the $k^{th}$ species so that
\begin{equation}\label{eq:kthm}A_i:=\n^{k-1}\times\{i\}\times \n^{n-k}\quad \forall i\in\n,\qquad \cal{I}:=\n.
\end{equation}

Let $\hat{l}_r^i$ (resp. $\hat{u}_r^i$) denote $l_r^f$ (resp. $u_r^f$) in \eqref{eq:boundsitaf} with $f$ being the indicator function $1_{A_i}$ of the set $A_i$. It follows from \eqref{eq:fboundsp} that $\hat{l}_r^i$ is a lower bound on $\hat{\pi}(i)$. On the other hand, because we may be marginalising  over states not included in the truncation $\s_r$ (e.g.\ in the case of \eqref{eq:kthm}), $\hat{u}_r^i$ is not necessarily an upper bound on $\hat{\pi}(i)$ . Collecting these quantities together and padding them with zeros, we obtain approximations of the marginals analogous to those  of the entire distribution in \eqref{eq:itaabounds}:
\begin{equation}\label{eq:lumeasdefm}\hat{l}_r(i):=\left\{\begin{array}{ll}\hat{l}_r^{i}&\text{if }i \in\cal{I}_r\\0&\text{if }i \not\in\cal{I}_r\end{array}\right.,\qquad \hat{u}_r(i):=\left\{\begin{array}{ll}\hat{u}_r^{i}&\text{if }i \in\cal{I}_r\\0&\text{if }i \not\in\cal{I}_r\end{array}\right.,\qquad \forall i\in\cal{I},\end{equation}
where $\cal{I}_r=\{i\in\cal{I}:A_i\cap\s_r\neq\emptyset\}$ is the (finite) subset of $i$s in $\cal{I}$ such that the intersection of $A_i$ with the truncation is non-empty. 
Similar manipulations to those behind \eqref{eq:trunerrexp2}--\eqref{eq:uppererbound2} give us a computable expression for the approximation error of $\hat{l}_r$ and bounds on that of $\hat{u}_r$:
\begin{align}||\hat{l}_r-\hat{\pi}||_{TV}&=||\hat{l}_r-\hat{\pi}||_{1}=1-\hat{l}_r(\cal{I}_r),\label{eq:hd72dh121j121}\\
\hat{u}_r(\cal{I}_r)-1&\leq\norm{\hat{u}_r-\hat{\pi}}_{TV}\leq\max\left\{\hat{u}_r(\cal{I}_r)-1+\frac{c}{r},\frac{c}{r}\right\},\\
\hat{u}_r(\cal{I}_r)-1&\leq\norm{\hat{u}_r-\hat{\pi}}_{1}\leq \hat{u}_r(\cal{I}_r)-1+\frac{2c}{r},\label{eq:hd72dh121j123}\end{align}
see \cite[Section~IVB1]{Kuntz2017} for details. Thus, while $\hat{u}_r$ may not bound $\hat{\pi}$ from above, its total variation and $\ell^1$ errors are straightforward to bound in practice. {As above, \eqref{eq:hd72dh121j121}--\eqref{eq:hd72dh121j123} hold for any bounds $\hat{l}^r(i)\leq \pi(A_i\cap\cal{S}_r)\leq \hat{u}^r(i)$ and not just those obtained with the ITA scheme.}

\juansec{The issue of convergence}
Little is known about this scheme's convergence. As shown in \cite[p.930]{Courtois1986}, $\pi^{z}_r(x)\leq\pi^x_r(x)$ for all $z$ and $x$ in $\s_r$, and it follows from \eqref{eq:itaabounds}  that $u_r(x)\leq \pi^x_r(x)$ for all $x$ in $\s_r$. Whenever the TA scheme converges (see end of Section~\ref{sec:taa}), $\pi^x_r(x)$ tends to $\pi(x)$ as $r$ approaches infinity implying that the upper bounds $u_r$ converge pointwise to $\pi$:
$$\lim_{r\to\infty}u_r(x)=\pi(x)\quad\forall x\in\s.$$ 
Because no analogous inequality is available for the lower bounds $l_r$ and because the in-boundary $\cal{B}_i(\s_r)$ in \eqref{eq:insidebound} over which we optimise varies with $r$, we have been unable to establish any type of convergence for $l_r$. However, were we able to show that $l_r$ converges pointwise, we could show that it converges in total variation using the {trick in Appendix~\ref{app:itaconvtrick}.}

\subsection{The linear programming approach}\label{lpapp}

To obtain approximations of the stationary distributions with strong convergence guarantees and computable errors, we  introduced in \cite{Kuntz2018a,Kuntz2017} two  truncation-based schemes that employ linear programming. They apply to chains with one or more stationary distributions under  the following assumption:
\begin{assumption}[{Moment bound}]\label{ass:1}We have at our disposal a norm-like function $w$ and constant $c$ such that the moment bound \eqref{eq:mombound} holds for all stationary distributions $\pi$.\end{assumption}

{If the rate matrix is regular,} Assumption~\ref{ass:1} and Theorem~\ref{Qstateq} imply that the set of stationary distributions is the set stationary solutions $\cal{P}$ of the CME  that satisfy the moment bound: 
\begin{equation}\label{eq:brp}\cal{P}:=\left\{\pi\in\ell^1: \begin{array}{l}
\pi Q(x)=0\enskip\forall x\in\s,\\
\pi(\s) = 1,\\
\pi(w)\leq c,\\
\pi(x)\geq0\enskip\forall x\in\s
\end{array}\right\}.\end{equation}
The linear programming (LP) scheme consists of viewing $\cal{P}$ as  a convex polytope in $\ell^1$, building finite-dimensional \emph{outer approximations} thereof, and optimising over these approximations. In particular, we set the truncation $\s_r$ to be the $r^{th}$ sublevel set \eqref{eq:sublevel} of $w$ and define the outer approximation
\begin{equation}\label{eq:br}\cal{P}_r:= \left\{\pi_r\in\ell^1: \begin{array}{l}
   \pi_r Q(x)=0 \enskip \forall x\in\cal{N}_r,\\
\pi_r(\s_r^c)=0,\\
1-c/r\leq \pi_r(\s)\leq 1,\\
\pi_r(w)\leq c,\\
\pi_r(x)\geq0\enskip\forall x\in\s
\end{array}\right\},\end{equation}
where
\begin{equation}\label{eq:nr}\cal{N}_r:= \left \{x\in\s_r: q(z,x)=0, \enskip \forall z \in \s^c_r \right \}\end{equation}
denotes the set of states inside the truncation that cannot be reached in a single jump from outside. For instance, in the case of SRNs~\eqref{eq:network} with rate matrices \eqref{eq:qmatrixsrn}, we have that $x$ belongs to $\cal{N}_r$ if and only if $x-\nu_j$ belongs to $\s_r$ for every $j$ such that $x-\nu_j$ belongs to $\s$ and $a_j(x-\nu_j)>0$. 
 
We say that $\cal{P}_r$ is an \emph{outer approximation} of $\cal{P}$ because the restriction $\pi_{|r}$~\eqref{eq:rest} to the truncation $\s_r$  of any stationary solution $\pi$ in $\cal{P}$ belongs to $\cal{P}_r$ \cite[Lemma~14]{Kuntz2017}. This outer approximation is $\mmag{\s_r}$-dimensional in the sense that any $\pi_r$ in $\cal{P}_r$ has support contained in the truncation ($\pi_r(x)=0$ for all $x\not\in\s_r$). Interestingly, the conditional distribution $\pi(\cdot|\s_r)$ defined in \eqref{eq:picond} also belongs to $\cal{P}_r$, see Appendix~\ref{app:condouter} details. 
\juansec{The birth-death case} 
In the case of the birth-death processes introduced in Section~\ref{sec:bd}, it is straightforward to obtain simple analytical descriptions of the outer approximations $\cal{P}_r$. In particular, suppose that we have at our disposal a bound $c$ on the mean of the stationary distribution $\pi$ so that \eqref{eq:mombound} holds with $w(x):=x$  and  consider the truncation $\s_r=\{0,1,\dots,r-1\}$ composed of the first $r$ states. Using an argument analogous to that in the proof of \cite[Theorem~11]{Kuntz2017}, we find that $\pi_r$ belongs to the outer approximation $\cal{P}_r$ if and only if there exists a constant 
\begin{equation}\label{eq:exp1}\left(1-\frac{c}{r}\right)\frac{1}{\gamma(\s_r)}\leq\alpha\leq \frac{1}{\gamma(\s_r)}\end{equation}
such that
\begin{equation}\label{eq:exp2}\pi_r(x)=\alpha\gamma(x)\quad\forall x=0,1,\dots,r-1,\qquad\pi_r(x)=0\quad\forall x=r,r+1,\dots,\end{equation}
where $\gamma(x)$ is as in \eqref{eq:1deg}. Due to \eqref{eq:1deg}--\eqref{eq:normalising}, the $\ell^1$ error of $\pi_r$ is
\begin{align*}\norm{{\pi_r}-\pi}_{1}&=\sum_{x=0}^\infty\mmag{\pi_r(x)-\pi(x)}=\sum_{x=0}^{r-1}\mmag{\pi_r(x)-\pi(x)}+\sum_{x=r}^\infty\pi(x)\\
&=\mmag{\frac{1}{\gamma(\s)}-\alpha}\gamma(\s_r)+m_r,\end{align*}
where $m_r$ denotes truncation error (c.f.\ \eqref{eq:tailmass}). Because the total variation and $\ell^1$-norms are equivalent \eqref{eq:normhier}, it follows from the above that the outer approximations $\cal{P}_r$ converge to $\cal{P}$ in the sense that any $\pi_r$ in $\cal{P}_r$ converges in total variation to the unique point $\pi$ in $\cal{P}$ as $r$ tends to infinity.
\juansec{The general case and the LP scheme}
In general, it is not possible to find analytical descriptions of the type \eqref{eq:exp1}--\eqref{eq:exp2} for the outer approximations $\cal{P}_r$. Instead, we may compute points belonging to these outer approximations by solving linear programs (LPs). LPs are particularly tractable convex optimisation problems~\cite{Boyd2004,Ben-Tal2001,Renegar2001} for which mature solvers are available. In our context, given any real-valued function $f$ on $\s$, the LP solver returns an \emph{optimal point} $\pi_r^*$ in $\cal{P}_r$ satisfying
\begin{equation}\label{eq:lp}\pi_r^*(f)=\sup\{\pi_r(f):\pi_r\in\cal{P}_r\}.\end{equation}
The optimisation problem on the right-hand side is a linear program because it entails optimising the linear functional $\pi_r\mapsto \pi_r(f)$ over a set defined by affine equalities and inequalities (known as \emph{constraints}). The supremum is referred to as the program's \emph{optimal value}. If $f:=0$, the linear program is said to be a \emph{feasibility problem} and its optimal points (i.e.\  all points in $\cal{P}_r$) are referred to as \emph{feasible points}.

In the case of a unique stationary solution {(i.e.\  $\cal{P}=\{\pi\}$)}, any feasible point $\pi_r$ of $\cal{P}_r$ can be used as an approximation of $\pi$. {In our practical experience,} the optimal points of the program
\begin{equation}\label{eq:lpmass}\sup\{\pi_r(\s_r):\pi\in\cal{P}_r\}.\end{equation}
are good approximations of $\pi$. For instance, in the birth-death case, there is only one such optimal point: the conditional distribution $\pi(\cdot|\s_r)$, optimal in the sense of Section~\ref{sec:opt}.

In the non-unique case, there exist several ergodic distribution each of which has support in a positive recurrent closed communicating class (c.f.\ Section~\ref{sec:doeblin}). If $x$ is a state in any such class $\cal{C}_i$, the optimal points of the program
\begin{equation}\label{eq:optiprogs}\sup\{\pi_r(x):\pi_r\in\cal{P}_r\},\end{equation} approximate the corresponding ergodic distribution $\pi^i$. Thus, by examining the  states to which any such optimal point $\pi_r^*$ assigns non-zero probability, we often identify the intersection of the class $\cal{C}_i$ and the truncation $\s_r$. 
Therefore, by setting $x$ to be one of the states to which $\pi_r^*$ assigns zero probability and recomputing an optimal point $\pi_r^*$, we are often able to discover other positive recurrent closed communicating classes and approximate their ergodic distributions, see \cite[Section~IVB2]{Kuntz2017} for further details and \cite[Section~VC]{Kuntz2017} for an example.

\juansec{{Obtaining moment bounds in practice---verifying Assumption~\ref{ass:1}}}
There are two main ways to find norm-like functions $w$ and constants $c$ satisfying Assumption~\ref{ass:1}. The first is to use a Foster-Lyapunov criterion of the type described in Section~\ref{sec:foster} (see also \cite{Glynn2008}). The second applies to SRNs~\eqref{eq:network} with polynomial or rational propensities and entails making use of mathematical programming approaches 
\cite{Schwerer1996,Kuntz2017,Kuntzthe,Sakurai2017,Dowdy2018,Dowdy2017,Ghusinga2017a,Gupta2014,Argeitis2014}. In this latter approach, we pick a $w$ and the mathematical programming methods yield a $c$. For guidance on how to pick $w$, see \cite[Section~IVB3]{Kuntz2017}.

\juansec{Convergence of the scheme}

In the case of a unique stationary solution $\pi$ {(i.e.\  $\cal{P}=\{\pi\}$)}, any sequence of feasible points $\pi_1\in\cal{P}_1,\pi_2\in\cal{P}_2,\dots$ converges {$w$-weakly*} to $\pi$  (c.f.\ Section~\ref{sec:convergence}), {where $w$ is the function featuring in the definition of $\cal{P}$~\eqref{eq:brp},} as long as the sets $\cal{N}_r$ in \eqref{eq:nr} approach the entire state space as $r$ approaches infinity:
\begin{equation}\label{eq:alleqs}\bigcup_{r=1}^{\infty}\cal{N}_r=\s,\end{equation}
see \cite[Corollary~3.6]{Kuntz2018a} for a proof.  In the non-unique case, any sequence of optimal points $\pi_{1}^*,\pi_2^*,\dots$ of the programs \eqref{eq:optiprogs} (with $r=1,2,\dots$) converges {$w$-weakly*} to the ergodic distribution $\pi^i$ associated with a positive recurrent closed communicating class $\cal{C}_i$ as long as $x$ belongs to $\cal{C}_i$, $Q$ is regular, and \eqref{eq:alleqs} is satisfied.

It is not difficult to see that \eqref{eq:alleqs} is equivalent to the columns of $Q$ having  finitely many non-zero entries:
$$\{z\in\s:q(z,x)\neq 0\}\text{ is finite for all } x\in\s.$$
For this reason, \eqref{eq:alleqs} asks that the chain is able to reach any given state, in a single jump, from at most finitely many others. Any SRN~\eqref{eq:network} with rate matrix \eqref{eq:qmatrixsrn}   satisfies this condition as the chain may only reach a state $x$ in a single jump from $x-\nu_1,\dots,x-\nu_m$. If \eqref{eq:alleqs}  is violated, then the equations indexed by states not in $\cup_{r=1}^{\infty}\cal{N}_r$ will not be included in any of outer approximations $\cal{P}_r$. In this case, it is possible {to} tweak the definition of the outer approximation $\cal{P}_r$ so that the convergence is recovered as long as a sufficiently good moment bound is available (in particular, one such that that $q_o$ in \eqref{eq:truncandaug} asymptotically grows slower than $w$ in the sense of \eqref{eq:weakstar2}), see \cite[Appendix~C]{Kuntz2018a} for details.

No computable error expressions or bounds are known for the approximations produced by this scheme. To obtain these, {we instead} iterate the scheme as described in the next section.
\subsection{Iterated linear programming}\label{lpappiter}
Just as for the TA scheme, iterating the LP scheme yields approximations accompanied by computable error expressions or bounds. We refer to this iterated variant (also introduced in \cite{Kuntz2017,Kuntz2018a}) as the \emph{iterated linear programming (ILP)} scheme. The ILP scheme yields bounds on $\pi$ by repeatedly solving the linear program \eqref{eq:lp} for various functions $f$.
In particular, the outer approximation property of $\cal{P}_r$ implies that the restriction $\pi_{|r}$ in \eqref{eq:rest} of any stationary solution $\pi$ in $\cal{P}$ can be bounded as follows:
\begin{equation}\label{eq:lowupthe4}l_r^f:=\inf\{\pi_r(f):\pi\in{\cal{P}_r}\}\leq\pi_{|r}(f)\leq\sup\{\pi_r(f):\pi\in{\cal{P}_r}\}=: u_r^f,\quad\forall \pi\in\cal{P}.\end{equation}
where $f$ is any given real-valued function on $\s$. We then obtain bounds on the entire average $\pi(f)$ using \eqref{eq:fboundsp}--\eqref{eq:fbounds}, {where $w$ is the function featuring in~\eqref{eq:brp}.} 

By computing these bounds for the indicator function $f:=1_x$ of each state $x$ in the truncation, we obtain state-wise lower $(l_r(x))_{x\in\s_r}$ and upper $(u_r(x))_{x\in\s_r}$ bounds on the  restriction $\pi_{|r}$ of any $\pi$ in $\cal{P}$. In the case of a unique $\pi$, we pad these bounds with zeros \eqref{eq:padding}, use them as approximations of $\pi$, and evaluate their errors using \eqref{eq:trunerrexp2}--\eqref{eq:uppererbound2}. Just as with the ITA scheme of Section~\ref{sec:iter}, the quality of the lower bounds is limited by the tail bound \cite[Proprosition~22]{Kuntz2017} while that of the upper bounds is not, see Section \ref{togglesec} for an example.

Similarly, to approximate a marginal  distribution $\hat{\pi}$ defined in \eqref{eq:marginal2}, we  compute $\hat{l}_r^i:=l_r^f$ and $\hat{u}_r^i:=u_r^f$ in \eqref{eq:lowupthe4}  for each indicator function $f=1_{A_i}$ of the sets $A_i$ with $i$ belonging to $\cal{I}_r$ (notation introduced in \eqref{eq:marginal2}--\eqref{eq:lumeasdefm}). By padding $(\hat{l}_r^i)_{i\in\cal{I}_r}$ and $(\hat{u}_r^i)_{i\in\cal{I}_r}$ with zeros \eqref{eq:lumeasdefm}, we obtain approximations of the marginal $\hat{\pi}$ whose errors can be evaluated using \eqref{eq:hd72dh121j121}--\eqref{eq:hd72dh121j123}.

Just as with the LP scheme of the previous section, no irreducibility or uniqueness assumptions are required for the ILP scheme: the bounds hold for the set of stationary solutions satisfying the moment bound \eqref{eq:mombound}. Indeed, as shown in \cite[Corollary~28]{Kuntz2017} a single non-zero state-wise lower bound {$l_r(x)$} serves as a numerical certificate proving that at most one stationary solution exists.

\juansec{The convergence of the bounds} 

If \eqref{eq:alleqs} holds and $f$ satisfies \eqref{eq:weakstar2}, then the sequences $l_1^f,l_2^f,\dots$ and $u_1^f,u_2^f,\dots$ in \eqref{eq:lowupthe4} converge to the respective infimum and supremum over the set of stationary solutions,
\begin{align*}\lim_{r\to\infty}{l^f_r}=l_f:=\inf\{\pi(f):\pi\in\cal{P}\},\qquad \lim_{r\to\infty}{u^f_r}=u_f:=\sup\{\pi(f):\pi\in\cal{P}\}.\end{align*}
If the stationary solution $\pi$ is unique, 
then the lower bounds $l_r=(l_r(x))_{x\in\s}$ on the full solution $\pi$ converge {$w$-weakly*} to  $\pi$  and the lower bounds $\hat{l}_r=(\hat{l}_r(i))_{i\in\cal{I}}$ on the marginal $\hat{\pi}$ converge in total variation to $\hat{\pi}$. Even though in our practical experience the upper bounds $u_r=(u_r(x))_{x\in\s}$ tend to converge at a faster rate than the lower ones (see Section~\ref{togglesec}), they are only known to converge pointwise to $\pi$ (and similarly for $\hat{u}_r=(\hat{u}_r(i))_{i\in\cal{I}}$ and $\hat{\pi}$). See \cite{Kuntz2017,Kuntz2018a} for detailed arguments behind these statements. 

\section{Numerical comparison of the schemes on a biological example}\label{togglesec}
In this section, we study the performance of the truncation-based schemes discussed in Sections~\ref{sec:qbd}--\ref{lpappiter} on a two-dimensional example. In particular, we consider {a toggle switch model without cooperativity~\cite{gardner2000construction,lipshtat2006genetic,strasser2012stability,Thomas2014,Thomas2020}},
$$
\varnothing \underset{a_2}{\overset{a_1}{\rightleftarrows}} P_1, \qquad \varnothing \underset{a_4}{\overset{a_3}{\rightleftarrows}} P_2,
$$
describing a network of two mutually repressing genes. For simplicity, we consider the symmetric case with repression modelled via {effective promoter-activity} functions and degradation modelled via linear decay: 
$$a_1(x) = \frac{20}{1+x_2}, \qquad a_2(x) = x_1, \qquad a_3(x) = \frac{20}{1+x_1}, \qquad a_4(x) = x_2,$$
where $x=(x_1,x_2)$ and $x_1$ (resp. $x_2$) denotes the copy number of protein $P_1$ (resp. $P_2$). The evolution of the copy numbers in time is then described by a continuous-time chain with state space $\s=\n^2$ and rate matrix \eqref{eq:qmatrixsrn}. The same argument as that in~\cite[Appendix~B]{Kuntz2017} shows that the rate matrix $Q$ is regular, that the chain is exponentially ergodic  with unique stationary distribution $\pi$, and that all of $\pi$'s moments are finite. The distribution (Figure~\ref{fig:toggleerror}(a)) is unimodal with almost all of its mass concentrated in the simplex $\{x\in\n^2:x_1+x_2<40\}$ containing $820$ states.

\juansec{Details of the different schemes}

To test the approximation schemes, we use the sublevel sets  \eqref{eq:sublevel} of the norm-like function
$$w(x):=(x_1+x_2)^6\quad\forall x\in\n^2$$
as truncations $\s_r$. Because the reactions consume or produce only one protein at a time, the chain is an LDQBDP with levels 
$$\L_l =\{x\in\n^2:x_1+x_2=l\}\quad\forall l\in\n.$$
To apply the LDQBDP scheme (Section~\ref{sec:qbd}), we use the level cut-offs  $L_r=\lfloor r^{1/6}\rfloor$ (so that the LDQBDP truncations in \eqref{eq:qbdtrunc} coincide with the sublevel sets of $w$) and we approximate $R^{L_r}$ with a matrix of zeros as done in \cite{Baumann2010,Phung-Duc2010,Hanschke1999}. For the TA scheme (Section~\ref{sec:taa}), we use a single re-entry state in the middle of the truncation's in-boundary.
To compute the bounds of the ITA  and ILP  schemes (Sections~\ref{sec:iter} and \ref{lpappiter}, respectively), and to define the outer approximations for the LP (Section~\ref{lpapp}) and ILP schemes, we compute the moment bound
$$\pi(w)\leq c := 1.8\times10^7$$
with the semidefinite programming approach  of \cite{Kuntz2017} (using all moment equations that only involve rational moments $\pi(x_1^{\alpha_1}x_2^{\alpha_2}/(1+x_1+x_2+x_1x_2))$ with exponents $\alpha_1,\alpha_2\in\n$ such that $\alpha_1+\alpha_2\leq 11$). In the LP scheme, we use optimal points of the program \eqref{eq:lpmass} as approximations of $\pi$.
To avoid numerical instability in our computations, we scale our approximations by the diagonal of the rate matrix (i.e.\  we carry out all computations using $\widetilde{\pi}_r=(q(x)\pi(x))_{x\in\s_r}$ instead of $\pi_r$). 

To benchmark the performance of all the schemes, we obtain a high precision reference approximation using the ITA scheme with a truncation $\s_{238^6}$ composed of $28441$ states 
which yields a guaranteed
total variation error~\eqref{eq:trunerrexp2} smaller than $10^{-7}$. 

\juansec{Results}

In the case of the truncation $\s_{24^6}$ composed of $300$ states, the LDQBDP, TA, and LP schemes produce approximations (blue, orange, yellow, Figure~\ref{fig:toggleerror}(b)) of the $P_1$-marginal distribution
$$\hat{\pi}(i):=\{(i,x):x\in\n\}\quad\forall i\in\n$$ 
with smaller errors than the ITA and ILP schemes do (teal, magenta, Figure~\ref{fig:toggleerror}(c)). 

\begin{figure}[h!]
	\begin{center}
	\includegraphics[width=1\textwidth]{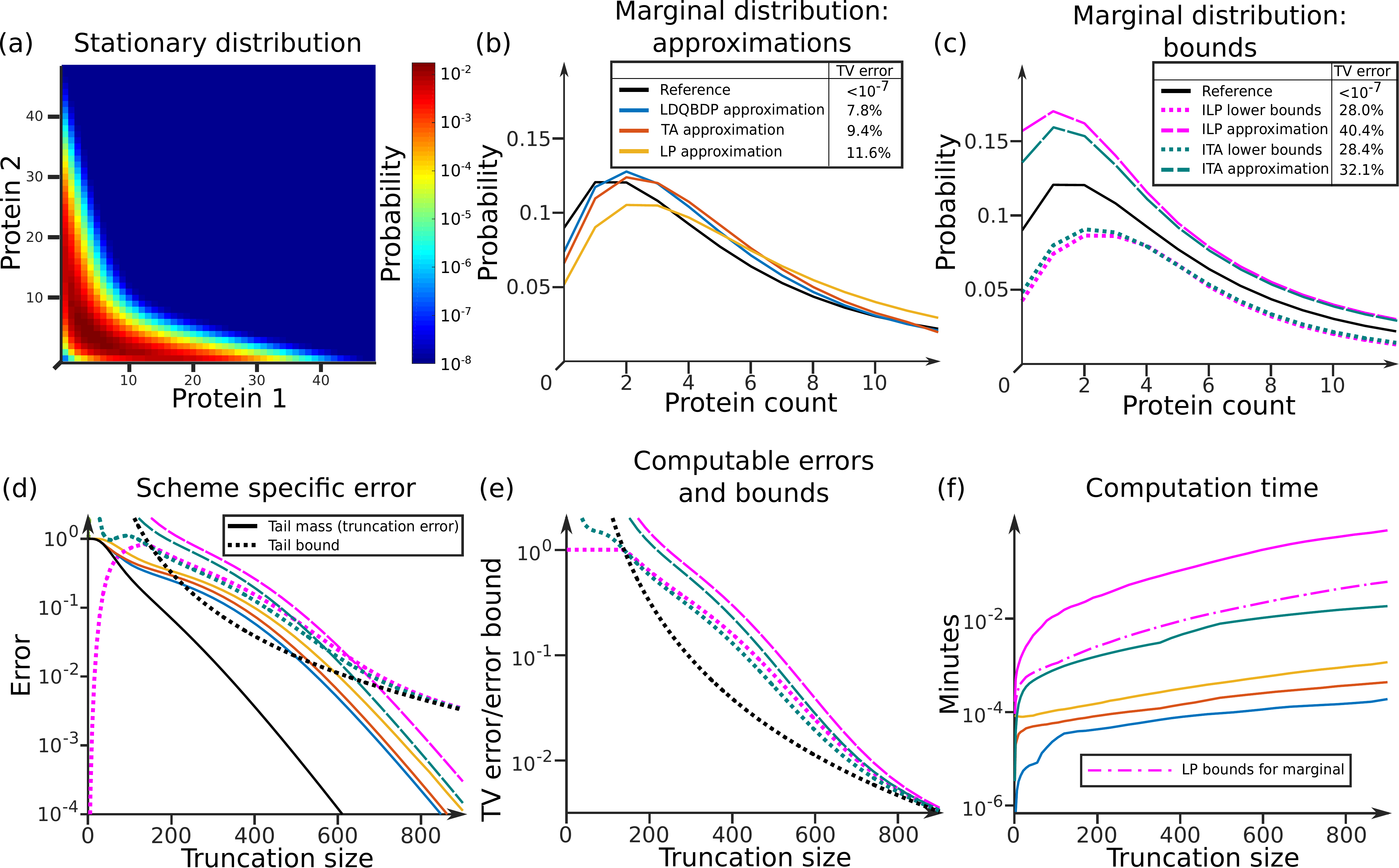}
	\vspace{-10pt}
	\end{center}
\caption{\textbf{Approximating the stationary distribution of the toggle switch model.} \textbf{(a)} High precision reference approximation of the stationary distribution obtained with the ITA scheme (see main text for details). \textbf{(b)} The $P_1$ marginal of the reference approximation (black) and lower quality approximations obtained using the LDQBDP (blue), TA (orange), and LP (yellow) schemes and the truncation $\s_{24^6}$ with $300$ states.  \textbf{(c)} As in $(b)$, with the ITA (teal) and ILP (magenta) lower bounds $\hat{l}_{24^6}$ (dotted) and approximations $\hat{u}_{24^6}$ (dashed) instead of the LDQBDP, TA, and LP approximations. \textbf{(d)} Tail mass/truncation error \eqref{eq:tailmass} (black solid), tail bound \eqref{eq:tailbound} (black dotted), and $\ell^1$ scheme-specific errors~\eqref{eq:errordecomptrun} of all five schemes (colors and line patterns as in (b,c)) as a function truncation size $\mmag{\s_r}$. \textbf{(e)} Tail bound, total variation errors \eqref{eq:trunerrexp2} of the ITA and ILP lower bounds and total variation error bounds \eqref{eq:uppererbound1} of ITA and ILP upper bounds as a function of truncation size (colors and line patterns as in (d)). \textbf{(f)} Computation times of all five schemes as a function of truncation size. In contrast with the other schemes, when using the ILP scheme, the cost of approximating the marginal distribution is smaller than that of approximating the full distribution and we plot the computation time for the marginal separately (magenta dashed/dotted) from the time for the full distribution (magenta solid).}\vspace{-25pt}\label{fig:toggleerror}
\end{figure}

To investigate how these errors depend on the truncation size $\mmag{\s_r}$, we note that the truncation error is the same for all schemes (solid black line, Figure~\ref{fig:toggleerror}(d)). We then compute the $\ell^1$ scheme-specific errors \eqref{eq:errordecomptrun} of all five schemes for truncations with sizes ranging from $1$--$903$ states (Figure~\ref{fig:toggleerror}(d)). While the LDQBDP, TA, and LP approximations consistently achieve smaller errors (solid) than the ITA and ILP upper bounds (dashed lines) do, the error of the least accurate upper bounds (ILP, dashed magenta) remains within one order of magnitude of that of the most accurate approximation (LDQBDP, blue solid line). Moreover, for sufficiently large truncations, their errors become proportional to the truncation error (black solid) or, equivalently, the tail mass \eqref{eq:tailmass}. In contrast, the errors of the ITA and ILP lower bounds (dotted lines) converge to the tail bound. Consequently, the ITA and ILP lower bounds  substantially underperform all other approximations for large enough truncations. An important observation is that the smallest scheme-specific error (LDQBDP) remains roughly two orders of magnitude greater than the truncation error (Figure~\ref{fig:toggleerror}(d)) and, thus, scheme-specific errors dominate the accuracy of all schemes. Because the truncation error is the smallest achievable error, there remains room for improvement.

Next, we focus on the computable errors~\eqref{eq:trunerrexp2} and error bounds~\eqref{eq:uppererbound1}--\eqref{eq:uppererbound2} of the ITA and ILP schemes (Figure~\ref{fig:toggleerror}(e)). In contrast, to the actual errors, the computable error bounds of the ITA and ILP upper bounds converge to the tail bound. Thus, even though the upper bounds outperform the lower bounds, the error guarantees we obtain in practice are similar in all cases. 

Finally, we compare the computation times of the schemes (Figure~\ref{fig:toggleerror}(f)).
For our example, the LDQBDP scheme outperforms all other schemes. Unsurprisingly, the bound-producing, iterative schemes ITA and ILP are orders of magnitude slower than the LDQBDP, TA and LP schemes as they involve a multiple of the latter's computations. The ILP scheme (solid magenta line) is significantly slower than the ITA scheme as involves solving a linear program per state in the truncation instead of a set of linear equations per state in the in-boundary of the truncation. However, when used to compute marginals (dash-dotted magenta line) 
far less linear programs are required and the computation times become comparable to those of the ITA scheme.

\section{Summary and open questions}\label{sec:conclusion}

Truncation-based schemes are a class of numerical methods used to approximate the stationary distributions of continuous-time Markov chains with large or infinite state spaces. They involve approximating the distribution within a finite truncation of the state space using the corresponding truncated rate matrix. We reviewed several of these schemes (Figure~\ref{fig:cartoon}) paying particular attention to their convergence and to the errors they introduce (Table~\ref{tab:1}). There are several other such schemes we omitted as we found them to be less applicable to the SRNs~\eqref{eq:network} that motivated this work. 
These other schemes include approximations for level-independent quasi-birth-death processes~{\cite{Evans1967,Wallace1969,Neuts1981,Latouche1993,Hajek1982}}, chains with Toeplitz-like rate matrices  \cite{Klimenok2006,Grassmann1990}, and certain types of queueing models, see~\cite{Neuts1981,Ramaswami1996,Takine2016,Kim2012,Phung-Duc2013,Grassmann1993,Breuer2002} and references therein.


The truncation error of any of these schemes equals the tail mass: the fraction of the stationary distribution lying outside of the truncation (Section~\ref{sec:truncerr}). It bounds from below the total variation error and may be viewed as the \emph{optimal error} given that  these schemes have the theoretical potential to achieve it (e.g.\ all five schemes discussed do achieve it in the case of birth-death processes). Moreover, our numerical experiments in Section~\ref{togglesec} suggest that the error of most of these methods is 
proportional to the truncation error for all sufficiently large truncations. For this reason, reliable estimates of, and bounds on, the tail mass are important when evaluating the error. The latter may be obtained by computing a moment bound using semidefinite programming \cite{Schwerer1996,Kuntz2017,Kuntzthe,Sakurai2017,Dowdy2018,Dowdy2017,Ghusinga2017a} and {an inequality of the type in~\eqref{eq:tailbound}}, directly using semidefinite programming (\cite[Section~4.3.1]{Kuntz2016} or \cite[Section~VII]{Dowdy2018}), or using Lyapunov functions~\cite{Glynn2008} and computational tools that search for them~\cite{Parrilo2000,Gupta2014,Argeitis2014,Dayar2011,Spieler2014}.


Similarly, selecting a truncation with a small tail mass is crucial for the computation of accurate approximations. Unfortunately, this generally proves challenging. One way to guide this choice is to generate long sample paths and fix the truncation to be the set of states that the paths spend most time in (the theoretical justification behind this approach being Corollary \ref{iredcor} and \eqref{eq:timeaverages}). However, the initial condition of the path can significantly bias the states visited (c.f.\ \cite{Gelman1992,Brooks1998} and \cite[Chap. 4]{Asmussen2007}). Because of this, several other so-called dynamic state space exploration techniques (see \cite{DeSouzaeSilva1992} and references therein) have been proposed to guide the selection of truncations. The use of these methods often turns out to be a balancing act as they can carry higher computational costs than the truncation-based schemes themselves.


Evaluating the errors introduced by these schemes is a challenging issue of practical relevance, as it eradicates the need to validate approximations using, for example, simulations. Sometimes, \emph{local error measures} such as the outflow rates and convergence factors in~\cite{Gupta2017} (see Appendix~\ref{app:cfprob} on the latter) and the residual norms in~\cite{Baumann2010,Dayar2011a,Dayar2012} are used to this end.
These present error estimates whose computation requires only the rows $q(x,\cdot)$ of the rate matrix indexed by states $x$ belonging to the truncation. 
By construction, local error measures do not account for the chain's behaviour outside of the truncation and, for this reason, are unreliable on their own. Rigorous error control requires some sort of \emph{global} information, for instance that contained in functions and constants satisfying a Foster-Lyapunov criterion, tail bounds, or moment bounds.

 
For these reasons, obtaining approximations with \emph{verified} small errors is significantly more difficult than obtaining ones with small (but unverified) errors. Schemes such as ITA (Section~\ref{sec:iter}) and ILP (Section~\ref{lpappiter}) that yield lower and upper bounds on the distribution have errors that are straightforward to evaluate or, at least, to bound. These schemes require a tail bound and their computable errors and error bounds are limited by this tail bound. For example, in the case of the toggle switch (Figure~\ref{fig:toggleerror}(e)), these computable errors and error bounds all collapse to the tail bound for sufficiently large truncations. For the toggle switch, the ITA and ILP scheme-specific errors dwarfed the truncation error (Figure~\ref{fig:toggleerror}(d)). Hence, in this case, setting $m_r$ in \eqref{eq:trunerrexp} to zero yields excellent error estimates for the upper bounds. 


The price we pay for error control is not just theoretical but computational too. Schemes that produce state-wise error bounds (ITA and ILP) are iterated variants of others that do not (TA and LP) and, thus, have computation times orders of magnitude greater (Figure~\ref{fig:toggleerror}(f)). It is worth noting that, when approximating a 
marginal distribution, the ILP scheme requires solving an LP per marginal state while the ITA scheme requires solving a set of linear equations per state in the in-boundary. Because, in the case of an SRN with $n$ species, the dimension of the in-boundary is typically $n-1$, we expect ILP to scale worse with $n$ than ITA  when approximating the entire distribution but better when approximating low-dimensional marginals thereof.


Furthermore, state-wise bounds are worst-case approximations. Hence, ITA and ILP tend to incur larger overall errors (e.g.\ see Section~\ref{togglesec}) than the non-bound producing schemes (LDQBDP, TA, and LP) do. This difference is most pronounced for the lower bounds as their errors are limited by the tail bound while those of the upper bounds are only limited by the truncation error. Indeed, in the toggle switch example (Figure~\ref{fig:toggleerror}(d)), the LDQBDP, TA, and LP approximations and the ITA and ILP upper bounds all appeared to be proportional to the truncation {error} for large truncations and the range of errors never spanned more than one order of magnitude.

An interesting alternative for error control to the iterated schemes  are the recent Lyapunov-function-based TA computable error bounds on the total variation error \cite{Liu2015,Masuyama2017,Masuyama2017a,Liu2018,Liu2018a}. While these also carry a  computational penalty as their calculation involves inverting the matrix in~\eqref{eq:phir}, they are not limited by any conservative tail bound: a potentially decisive practical boon. Indeed, our preliminary numerical experiments (Figure~\ref{fig:sfsp}) show that, even though these bounds are conservative for low to medium truncation sizes, they eventually become proportional to the actual error for large truncation sizes. Investigating this matter further, and that of how to choose the free parameters featuring in these bounds, would be very beneficial for their use.


The computational complexity of the reviewed schemes is a low-degree polynomial of the truncation size, which grows combinatorially with the number of species in the network.
Hence developing efficient implementations of these schemes (e.g.\ \cite{Baumann2013a,Dayar2012,Kazeev2015,Gupta2017,Cao2008,Cao2016,Cao2016b}) is crucial for their future use. Truncation-based schemes have an oft-unmentioned Achilles heel that we have omitted in this review: their numerical stability. It is normally the case that, for any large enough truncations, the entries of the truncated rate matrix $Q_r$ and of the approximation $\pi_r$ vary by many orders of magnitude. This results in large condition numbers and large round-off errors in the double precision floating point arithmetic typically employed when implementing the schemes. One way to mitigate this issue is to scale the approximations so that the range of orders of magnitude in their entries (and/or in those of the truncated rate matrix) is reduced \cite{Gupta2017,Kuntz2017}. Alternatively, using higher precision arithmetic also ameliorates this issue.
%
%
Stationary distributions with several isolated peaks pose a further practical challenge as any contiguous truncation including all peaks will normally also include many unimportant low-probability states and lead to high computational costs. For many of these cases, especially adapted methods along the lines of those in \cite{Spieler2014,Spieler2013,Milias-Argeitis2011} are necessary.


Another practical issue is how to deal with non-uniquness of the stationary distributions. The LDQBDP, TA, and ITA schemes implicitly assume uniqueness. Fortunately, the ILP scheme provides an automatic test for it: if a single state-wise lower bound is non-zero, then there exists at most one stationary distribution \cite[Corollary~28]{Kuntz2017}. In the non-unique case, the challenge becomes deducing what are the chain's closed communicating classes, something the LP scheme can help with as discussed in Section~\ref{lpapp} (see also \cite[Section~IVB2]{Kuntz2017}). Once these classes are identified, we approximate the ergodic distributions by replacing the state space with the appropriate closed communicating class and applying any of these schemes as we normally do for the unique case.

Aside from the above perhaps more practical matters, there remain many open theoretical questions. We close this review by listing those we find most intriguing. For any given chain, how should the re-entry matrices be chosen to ensure that the TA scheme converges? Similarly, under what circumstances does the ITA scheme converge? How should the free parameter, $\beta$ in \eqref{eq:l1boundsfsp}, be picked to produce tight computable error bounds for the TA scheme? Can similar Lyapunov-function-based error bounds be obtained for the LDQBDP scheme? Most importantly, how can we predict the approximation error a priori without running these schemes? A positive resolution to the latter question would be if, for all sufficiently large truncations, the errors of the ITA and ILP lower bounds were indeed equal to the tail bound and those of the other truncation-based approximations proportional to the tail mass as observed in the example of Section~\ref{togglesec}. Were this to be the case in general, it would open the door to a new kind of \emph{a priori error control} that would remove the need for the costly trial and error often involved in the computation of these approximations.

{\juansec{Acknowledgements} We thank an anonymous referee for their suggestion of using linear programming to tighten the error bounds of the TA scheme and for their many helpful comments that substantially improved this manuscript.}

\bibliographystyle{siamplain} 
\bibliography{statsurv}

\appendix
\addtocontents{toc}{\protect\setcounter{tocdepth}{0}}

\section{Relegated proofs and material for Section~\ref{sec:theory}}\label{app:a}

{\subsection{Chains with rate matrices that are not regular}\label{app:nonreg}
For the interested reader, we  discuss here how the theory of Sections~\ref{sec:cme}--\ref{sec:doeblin} extends beyond the case of a regular rate matrix $Q$. We still assume that $Q$ is totally stable and conservative (as defined in~\eqref{eq:qmatrix}), see \cite[Chapter~III]{Rogers2000} for an introduction to the general case.

By definition, a rate matrix is not regular if only if for at least one starting position, the chain has non-zero probability of exploding:  infinitely many jumps accumulate in a finite amount of time and the chain diverges to infinity (i.e.\ leaves every finite subset of the state space~\cite[Theorem~26.10]{Kuntz2020}). At this instant, the chain `runs out of instructions' and one must choose how to continue its path. 
In applied contexts, perhaps the most popular choice entails `killing' the chain by leaving it at some `cemetery' state $\Delta \not\in \mathcal{S}$ for all time past the explosion (i.e.\ setting $X_t=\Delta, \, \forall t\geq T_\infty$). 
The chain is then viewed as a process living in the \emph{extended state space} $\s_\Delta:=\s \cup \{\Delta\}$ and a probability measure $\Pb_{\Delta}$ satisfying 
\[
\Pb_{\Delta}(\{X_t=\Delta\enskip\forall t\geq0\})=1
\]
is constructed.
In this case, the restriction $(p_t(x))_{x\in\s}$ to $\s$ of the time-varying law $(p_t(x))_{x\in\s_\Delta}$ still satisfies the CME~\eqref{eq:CME}. 
However, $(p_t(x))_{x\in\s}$ need no longer be the CME's only solution---instead, it is the one with smallest mass, see \cite[Chapters~2~and~4]{Anderson1991} for the $\lambda=1_x$ case and \cite[Section~33]{Kuntz2020} for the general $\lambda$ case. The characterisation of stationary distributions in Theorem~\ref{Qstateq} then becomes~\cite[Theorem~43.6]{Kuntz2020} which states: `A probability distribution $\pi$ on $\s$ satisfies~\eqref{eq:statdef} if and only if it satisfies the stationary equations~\eqref{eq:stat} and the chain cannot explode when its starting position is sampled from $\pi$ (i.e.\ $\Pb_\pi(\{T_\infty=\infty\})=1$)'. See \cite[Example~1]{Miller1963} for a counter-example showing that this extra $\Pb_\pi(\{T_\infty=\infty\})=1$ requirement cannot be omitted. 

In~\cite{Meyn1993a}, it is shown that the stability theory of Sections~\ref{sec:stability}--\ref{sec:doeblin} only requires a few adjustments. In particular,  we defined 
a closed communicating class as a subset $\cal{C}$ of the state space $\s$ such that, for each $x,y \in \cal{C}$, there exists a sequence of states $x_1,\dots,x_l \in \s$ satisfying~\eqref{eq:irred} (i.e.\ a sequence of states through which a chain with rate matrix $Q$ can travel from $x$ to $y$), in which case we write $x \to y$. For killed chains, the transitive relation $\to$ must be extended to account for the fact that states in $\s$ can travel via an explosion to the cemetery state $\Delta$ from which they cannot return:
\begin{align}
\label{eq:extending}
x\to \Delta\Leftrightarrow \Pbx{\{T_\infty<\infty\}}>0,\quad \Delta\not\to x,\quad\forall x\in \s.
\end{align}
%
Hence the closed communicating classes must be redefined accordingly. The results of Sections~\ref{sec:stability}--\ref{sec:doeblin} then hold with $\s_\Delta$ replacing $\s$. Note that the singleton $\{\Delta\}$ now counts as a positive recurrent closed communicating class; $1_\Delta$ counts as an ergodic distribution; and any convex combination of $1_\Delta$ with the other ergodic distributions counts as a stationary distribution. 
In situations where $\Delta$ and the expanded state space $\s_\Delta$ are viewed as mathematical constructions of little modelling value, one could remove them from the stability theory and use the results in \cite[Sections~43--45]{Kuntz2020} showing that Theorem~\ref{doeblinc} holds as is (i.e.\ without extending~\eqref{eq:extending} the relation $\to$ and replacing $\s$ with $\s_\Delta$ throughout Sections~\ref{sec:cme}--\ref{sec:doeblin}).

Killing the chain is not the only way to continue its sample paths past an explosion such that the Markov property is preserved. For instance, at the moment of explosion, one can sample a state from any given distribution $\rho$ on $\s$; re-initialise the chain at this state; continue the sample path by running the Kendall-Gillespie algorithm up until the next explosion; sample another state from $\rho$; re-initialise, and so on. (More complicated ways of `coming back from infinity' while preserving the Markov property are also possible~\cite{Levy1951,Kendall1956,Reuter1957,Freedman1983b,Anderson1991,Rogers2000}).  Of course, the probability that a chain which comes back from infinity lies in any given state $x\in\s$ at any given  time is at least that of a chain which is killed: both chains are identical up until the first explosion and the latter may not return to $x$ past this point. For this reason, chains that are killed at the first explosion are called \emph{minimal}. 

The time-varying law of non-minimal chains need not satisfy the CME~\eqref{eq:CME} because its right-hand side does not account for the possibility that the chain enters a state directly via an explosion~ \cite[Chapters~2~and~4]{Anderson1991}. Consequently, stationary distributions might not satisfy the stationary equations~\eqref{eq:stat} obtained by setting the left-hand side of the CME~\eqref{eq:CME} to zero, see e.g.~\cite[Example~2]{Pollett1990}. On the other hand, the stability theory of Sections~\ref{sec:stability}--\ref{sec:doeblin} holds almost unchanged, except that it requires significantly more involved arguments~\cite{Meyn1993a} and that the relation $\to$ must be extended to account for the ability of the chain to travel between states by exploding and coming back.}
\subsection{Boundedness in probability and boundedness in probability on average}\label{app:bip}

A chain with a regular rate matrix is said to be \emph{bounded in probability on average} if for each $0<\varepsilon< 1$ and deterministic initial condition $x$, there exists a finite set $F\subseteq\s$ such that  the chain spends on average at least $(1-\varepsilon)$ of any sufficiently long period of time in $F$:
\begin{align}\liminf_{T\to\infty}\Ebx{\frac{1}{T}\int_0^{{T}} 1_F(X_t)dt}&=\liminf_{T\to\infty}\frac{1}{T}\int_0^T\Pbx{\{X_t\in F\}}ds\nonumber\\
&\geq 1-\varepsilon.\label{eq:bipa}\end{align}
However, all continuous-time chains are aperiodic which ensures that the limit 
$$L:=\lim_{T\to\infty}\Pbx{\{X_T\in F\}}$$
exists, see \cite[Theorem~5.1.3]{Anderson1991}. Thus, for any $\bar{\varepsilon}>0$, we can find a $\bar{T}$ such that 
$$\mmag{L-\Pbx{\{X_t\in F\}}}\leq\bar{\varepsilon}\quad\forall t\geq\bar{T}.$$
For this reason,
\begin{align*}&\mmag{\frac{1}{T}\int_0^T\Pbx{\{X_t\in F\}}dt-\Pbx{\{X_T\in F\}}}\\
&\qquad\qquad\leq\frac{1}{T}\int_0^{\bar{T}}\mmag{\Pbx{\{X_t\in F\}}-\Pbx{\{X_T\in F\}}}dt\\
&\qquad\qquad\quad+\frac{1}{T}\int_{\bar{T}}^T\mmag{\Pbx{\{X_t\in F\}}-\Pbx{\{X_T\in F\}}}dt\leq 2\frac{\bar{T}}{T}+2\bar{\varepsilon}\quad\forall T\geq \bar{T}.\end{align*}
Because we are able to make the right-hand side arbitrarily small by picking small enough $\bar{\varepsilon}$s and large enough $T$s, it is straightforward to verify that \eqref{eq:bip} holds for some given $x,F,\varepsilon$ if and only if \eqref{eq:bipa} holds for the same $x,F,\varepsilon$. In other words, a continuous-time chain is bounded in probability if and only if it is bounded in probability on average.

\subsection{The proof of Theorem~\ref{doeblinc}}\label{app:doeblin}

Part $(i)$ can be found in any textbook on Markov chains (e.g.\ \cite{Anderson1991,Asmussen2003,Norris1997}). For a proof of $(ii)$, see {\cite[Theorem~43.17]{Kuntz2020}}. Theorem~8.2 in \cite{Meyn1993a} shows that the chain is bounded on probability on average if and only if
$$\Pbx{\bigcup_{i\in\cal{I}}H_i}=1\quad\forall x\in\s.$$
Note that to use the results in \cite{Meyn1993a,Meyn1993b}, we must topologise $\s$ using the discrete metric which makes every real-valued function on $\s$ continuous and ensures that the chain is a ``T-process'' (as defined in  \cite{Meyn1993a}). Part $(iii)$ then follows from the above because 
\begin{equation}\label{eq:fen78wafhna87w3ona}\Pb_\lambda=\sum_{x\in\s}\lambda(x)\Pb_x,\end{equation}
e.g.\ see \cite[Section~26]{Kuntz2020}.  Because a sequence of probability distributions $(\mu_m)_{m\in\zp}$ on $\s$ converges to $\pi$ in total variation if and only if
$$\lim_{m\to\infty}\mu_m(f)=\pi(f)$$
for all bounded functions $f$ (e.g.\ see Appendix~B in \cite{Kuntz2018a}), \cite[Theorem~8.1]{Meyn1993a} shows that, if the chain is bounded in probability, then \eqref{eq:spaceaverages} holds with $\pi=\sum_{i\in\cal{I}}\Pbx{H_i}\pi^i$ and \eqref{eq:timeaverages} holds $\Pb_x$-almost surely with $\pi=\sum_{i\in\cal{I}}1_{H_i}\pi^i$, for all $x\in\s$. That \eqref{eq:spaceaverages}--\eqref{eq:timeaverages} hold for general $\lambda$ then follows from the bounded convergence theorem. 
{Given our assumption that $Q$ is regular, the converse also follows from $(iii)$ because~\eqref{eq:spaceaverages},  Tonelli's theorem, and monotone convergence imply that
%
%
\begin{align*}\sum_{i\in\cal{I}}\Pbl{H_i}&=\sum_{i\in\cal{I}}\Pbl{H_i}\left(\sum_{x\in\s}\pi^i(x)\right)=\sum_{x\in\s}\left(\sum_{i\in\cal{I}}\Pbl{H_i}\pi^i(x)\right)=\lim_{t\to\infty}\sum_{x\in\s}p_t(x)\\
&=1.\end{align*}
Alternatively, see \cite[Corollary~45.6]{Kuntz2020} for proofs of $(iii)$--$(iv)$ avoiding the technical set-up of \cite{Meyn1993a}.}
\subsection{Proofs of the Foster-Lyapunov criteria}\label{app:lyaproofs}

\begin{proof}[Proof of Theorem~\ref{lyareg}]The forward direction was first shown in \cite[Theorem~16]{Chen1986} (see also \cite[Corollary~2.16]{Anderson1991} or \cite[Theorem~2.1]{Meyn1993b}). The reverse direction was proven in F.~M.~Spieksma's recent paper \cite[Theorem~2.1]{Spieksma2015}.
\end{proof}

\begin{proof}[Proof of Theorem~\ref{lyabia}]The forward direction and \eqref{eq:lyabound} follow from \cite[Theorems 4.6--4.7]{Meyn1993b}. Inequality \eqref{eq:lyabound2} follows directly from \eqref{eq:lyabound}. The reverse direction in the irreducible case follows from\footnote{Note that inequality \cite[(5)]{Tweedie1981} contains a typo: the right-hand side should read ``$-\lambda y_i-1$'' instead of ``$-y_i-1$''.} \cite[Theorem 3$(i)$]{Tweedie1981}. {For the slight extension to $\cal{T}$-empty-and-$\cal{I}$-finite case, see \cite[Corollary~45.6~and~Theorem~49.1]{Kuntz2020}.}
\end{proof}

\begin{proof}[Proof of Theorem~\ref{lyacom}]Regularity and boundedness in probability follow immediately from Theorems \ref{lyareg} and \ref{lyabia}. The exponential convergence in the irreducible case follows from \cite[Theorem 7.1]{Meyn1993b} after noting that, if $w\geq1$, then the $w$-norm dominates the total variation norm (see Section~\ref{sec:convergence} for details). {For the general case, see~\cite[Theorem~50.4]{Kuntz2020}.}
\end{proof}

\section{Relegated proofs for Section~\ref{sec:truncs}}

{\subsection{Proof of \eqref{eq:trunminerr}}\label{app:trunminerr}
Let $\s^-:=\{x\in\s:\pi_r(x)<\pi(x)\}$ be the set of states whose probability $\pi_r$ underestimates and $\s^+:=\{x\in\s:\pi_r(x)\geq \pi(x)\}$ be its complement. Because $\pi_r$ has support contained in $\s_r$,
\begin{align*}
 \pi_r(\s^-)=\pi_r(\s^-_r),\quad \pi(\s^-)=\pi(\s^-_r)+\pi(\{x\not\in\s_r:0<\pi(x)\})=\pi(\s^-_r)+m_r.
\end{align*}
For these reasons,
\begin{align*}2\norm{\pi_r-\pi}_{TV}&=\norm{\pi_r-\pi}_{1}=\sum_{x\in\s^-}(\pi(x)-\pi_r(x))+\sum_{x\in\s^+}(\pi_r(x)-\pi(x))\\
&=\pi(\s^{-})-\pi_r(\s^-)+\pi_r(\s^+)-\pi(\s^+)=2(\pi(\s^{-})-\pi_r(\s^-))\\
&=2(m_r+\pi(\s^-_r)-\pi_r(\s^-_r))=2\left(m_r+\sum_{x\in\s_r^-}(\pi(x)-\pi_r(x))\right).\end{align*}
}

\subsection{The conditional distribution is the optimal approximating distribution}\label{app:opt} %
The irreducibility assumption implies that $\pi(x)>0$ for all $x$ in $\s_r$. If $\pi_r(x)=0$ for some $x$ in $\s_r$, then 
$$\frac{\mmag{\pi_r(x)-\pi(x)}}{\pi_r(x)}=\infty$$
and the maximum relative error is infinite (in particular, larger than $m_r$). Instead, suppose that $\pi_r(x)>0$ for all $x$ in $\s_r$ and let 
$$\varepsilon(x):=\pi_r(x)-\pi(x|\s_r)\qquad\forall x\in\s_r.$$
Note that
\begin{equation}\label{eq:fenuaife}\frac{\mmag{\pi_r(x)-\pi(x)}}{\pi_r(x)}=\frac{\mmag{m_r+\alpha(x)}}{1+\alpha(x)}\qquad\forall x\in\s_r.\end{equation}
where $\alpha(x)=\pi(\s_r)\varepsilon(x)/\pi(x)$ and the denominator is positive.  If $\varepsilon(x)>0$ for some state $x$ in $\s_r$, we have that 
$$\frac{\mmag{m_r+\alpha(x)}}{1+\alpha(x)}>m_r$$
because $m_r<1$ (by irreducibility) implies that $z\mapsto (m_r+z)/(1+z)$ is a strictly increasing function on $[0,\infty)$. Thus, in this case, the maximum relative error is also greater than $m_r$. Suppose instead that $\varepsilon(x)<0$ for some state $x$in $\s_r$. Because $\pi_r$ and $\pi(\cdot|\s_r)$ are probability distributions on $\s_r$, we have that $\sum_{z\in\s_r}\varepsilon(z)=0$. It follows that there must exist another state $x'$ in $\s_r$ such that $\varepsilon(x')>0$ and using the same reasoning as before we have the maximum relative error is greater than $m_r$. In other words, the maximum {relative} error is greater than $m_r$ unless $\varepsilon(x)=0$ for all $x$ in $\s_r$ in which case \eqref{eq:fenuaife} implies that the maximum relative error is $m_r$. The result follows as $\varepsilon(x)=0$ for all $x$ in $\s_r$ if and only if $\pi_r$ is the conditional  distribution $\pi(\cdot|\s_r)$.

\subsection{The censored chain and its rate matrix}\label{app:censored}

A complete and rigorous argument showing that the stochastic process $\bar{X}=(\bar{X}_t)_{t\geq0}$ obtained by erasing the segments of the paths of $X$ lying outside of {the finite truncation} $\s_r$ and  gluing together the ends of the remaining segments is statistically indistinguishable to any minimal continuous-time chain $X^{\varepsilon_r}$ with rate matrix $Q^{\varepsilon_r}$ (c.f.\ \eqref{eq:truncandaug}) involves repeated applications of the strong Markov property and requires a level of technical machinery beyond the scope of this review {(e.g.\ see \cite[Sections~28--30, 37]{Kuntz2020})}. Instead, we give a sketch of the argument and leave the full details to the motivated reader. In what follows, we assume that the $\varphi$-irreducible chain $X$ has a unique stationary distribution $\pi$ and that $\pi(\s_r)>0$ implying that the chain will keep revisiting the truncation for all time (e.g.\ this follows from \cite[Theorems~39.2, 43.15]{Kuntz2020}).

Because the process $\bar{X}$ is a continuous-time Markov chain~\cite[Section~1.6]{Freedman1983a}, it suffices~{\cite[Theorem~37.1]{Kuntz2020}} to show that its rate matrix $\bar{Q}=(\bar{q}(x,y))_{x,y\in\s_r}$ coincides with $Q^{\varepsilon_r}=(q^{\varepsilon_r}(x,y))_{x,y\in\s_r}$ in \eqref{eq:truncandaug}. To do so, let $(\bar{q}(x))_{x\in\s_r}$ denote minus the diagonal of $\bar{Q}$ and $\bar{P}=(\bar{p}(x,y))_{x,y\in\s_r}$ denote the one-step matrix of $\bar{X}$'s embedded discrete-time chain (defined by replacing $Q$ with $\bar{Q}$ in \eqref{eq:qmatrix} and \eqref{eq:jumpmat}, respectively). Similarly for $(q^{\varepsilon_r}(x))_{x\in\s_r}$, $P^{\varepsilon_r}=(p^{\varepsilon_r}(x,y))_{x,y\in\s_r}$, and $Q^{\varepsilon_r}$. Given that a rate matrix is fully determined by its diagonal and the associated one-step matrix, it suffices to show that
\begin{equation}\label{eq:jfa87fn37afmiea}
\bar{p}(x,y)=p^{\varepsilon_r}(x,y)\quad\forall x,y\in\s_r,\qquad \bar{q}(x)=q^{\varepsilon_r}(x)\quad\forall x\in\s_r.
\end{equation}
A state $x$ in $\s_r$ is absorbing for $\bar{X}$ (i.e.\  $\bar{q}(x)=0$ and $\bar{p}(x,x)=1$) if and only if it is absorbing for $X$ or jumps from $x$ always take $X$ outside of the truncation and $X$ always returns via $x$:
$$q(x)=0\enskip\text{or}\enskip p(x,\s_r^c)=1\enskip\text{and}\enskip \varepsilon{_r}(x,x)=1,$$
where $p(x,\s_r^c):=\sum_{y\not\in\s_r}p(x,y)$~in \eqref{eq:jumpmat} denotes the probability that $X$ next jumps out of the truncation if it is at $x$ while $\varepsilon_r(x,x)$  defined in \eqref{eq:exre0} denotes the probability that $X$ returns to the truncation by jumping to $x$ if the last state it visited before leaving was $x$. In either case, \eqref{eq:truncandaug} implies that $q^{\varepsilon_r}(x)=0$ and $p^{\varepsilon_r}(x,x)=1$, and \eqref{eq:jfa87fn37afmiea} follows for any absorbing $x$ in $\s_r$ and (absorbing or not) $y$ in $\s_r$.

Suppose that $x$ in $\s_r$ is not an absorbing state for $\bar{X}$ (and, hence, neither for $X$). Standard theory \cite[Sections.~A.II.1--2]{Asmussen2003} tells us that $p(x,y)$ is the probability that $\bar{X}$ first jumps to $y$ if it starts at $x$ and $\bar{q}(x)$ is one over the mean amount of time elapsed until this jump occurs:
\begin{equation}\label{eq:fbya87n873awf}
\bar{p}(x,y)=\Pb_x(\{\bar{X}_{\bar{T}_1}=y\})\quad\forall y\in\s_r,\qquad \bar{q}(x)= \frac{1}{\Ebx{\bar{T}_1}},
\end{equation}
where $\bar{T}_1$ denotes the first jump time of $\bar{X}$. Suppose that $X$ (and, consequently, $\bar{X}$) start at $x$. The paths of $\bar{X}$ that lie in a given state $y$ in $\s_r$ after the first jump correspond to the paths of $X$ {that} visit  $y$ at some point $\tau$ and remain either in $x$ or the outside of the truncation up until $\tau$. For this reason, the event $\{\bar{X}_{\bar{T}_1}=y\}$ that $\bar{X}$ first jumps to $y$ {decomposes} into the following disjoint union
\begin{align*}A_0:=\{&X\text{ jumps directly to }y\}\\
\cup A_1:=\{&X\text{ leaves the truncation, returns by jumping to }x\text{, and jumps to }y\}\\
\cup A_2:=\{&X\text{ leaves the truncation, returns by jumping to }x\text{, leaves the truncation,}\\
&\text{returns by jumping to }x\text{, and jumps to }y\}\\
\vdots&\\
\cup B_1:=\{&X\text{ leaves the truncation, returns by jumping to }y\}\\
\cup B_2:=\{&X\text{ leaves the truncation, returns by jumping to }x\text{, leaves the truncation,}\\
&\text{returns by jumping to }y\}\\
\cup B_3:=\{&X\text{ leaves the truncation, returns by jumping to }x\text{, leaves the truncation,}\\
&\text{returns by jumping to }x\text{, leaves the truncation, returns by jumping to }y\}\\
\vdots&
\end{align*}
The strong Markov property then implies that
$$\Pbx{A_n}=(p(x,\s_r^c)\varepsilon_r(x,x))^np(x,y),\enskip \Pbx{B_{n+1}}=p(x,\s_r^c)(p(x,\s_r^c)\varepsilon_r(x,x))^{n}\varepsilon_r(x,y),$$
for all $n$ in $\n$, where $p(x,y)$ in \eqref{eq:jumpmat} denotes the probability that $X$ next jumps to $y$ if it currently lies at $x$, while $\varepsilon_r(x,y)$  defined in \eqref{eq:exre0} denotes the probability that $X$ returns to the truncation by jumping to $y$ if the last state it visited before leaving was $x$. Thus, the leftmost term in \eqref{eq:fbya87n873awf} reads
\begin{align*}\bar{p}(x,y)&=\sum_{n=0}^\infty(\Pbx{A_n}+\Pbx{B_{n+1}})\\
&=(p(x,y)+p(x,\s_r^c)\varepsilon{_r}(x,y))\sum_{n=0}^\infty(p(x,\s_r^c)\varepsilon_r(x,x))^n\\
&=\frac{p(x,y)+p(x,\s_r^c)\varepsilon{_r}(x,y)}{1-p(x,\s_r^c)\varepsilon_r(x,x)}\end{align*}
Multiplying the numerator and denominator by $q(x)$ and comparing with~\eqref{eq:truncandaug}, we obtain the leftmost equation in \eqref{eq:fbya87n873awf}:
$$\bar{p}(x,y)=\frac{q(x,y)+q_o(x)\varepsilon{_r}(x,y)}{q(x)-q_o(x)\varepsilon_r(x,x)}=\frac{q^{\varepsilon_r}(x,y)}{q^{\varepsilon_r}(x)}=p^{\varepsilon_r}(x,y).$$

Similarly, assuming that $X$ starts at the state $x$, the first jump time $\bar{T}_1$ of $\bar{X}$ is equal to the total amount of time $X$ spends in $x$ up until the moment it enters any other state belonging to the truncation (i.e.\  until $X$ visits $\s_r^x:=\s_r\backslash\{x\}$ for the first time). Each visit of $X$ to state $x$ last{s} an exponentially distributed amount of time with mean $1/q(x)$ that is independent of the number of visits (e.g.\ see \cite[Theorem~A.II.1.2]{Asmussen2003}). For these reasons, the mean first jump time $\Ebx{\bar{T}_1}$ of $\bar{X}$ equals
$$\sum_{n=1}^\infty\frac{n}{q(x)}\Pbx{\{\text{$X$ visits state $x$ exactly $n$ times before hitting $\s_r^x$ for the first time}\}}.$$
The event that $X$ lies in state $x$ exactly $n$ times before visiting $\s_r^x$ for the first time decomposes into the disjoint union $A_n\cup B_n$, where
\begin{align*}A_n:=\{&X\text{ leaves the truncation and returns by jumping to $x$  consecutively $n-1$}\\
&\text{times, and then transitions from $x$ to $\s_r^x$ on the following jump}\},\\
B_n:=\{&X\text{ leaves the truncation and returns by jumping to $x$  consecutively $n-1$}\\
&\text{times, leaves the truncation again, and then hits $\s_r^x$ upon its return}\}.\\
\end{align*}
The strong Markov property then implies that
\begin{align*}\Pbx{A_n}&=(p(x,\s_r^c)\varepsilon_r(x,x))^{n-1}p(x,\s_r^x),\\
\Pbx{B_{n}}&=(p(x,\s_r^c)\varepsilon_r(x,x))^{n-1}p(x,\s_r^c)\varepsilon_r(x,\s_r^x),\quad\forall n\in\zp,\end{align*}
where $\varepsilon_r(x,\s_r^x):=\sum_{y\in\s_r^x}\varepsilon_r(x,y)$ denotes the probability the chain returns to the truncation by jumping into any state except $x$ if $x\in\cal{B}_o(\s_r)$ was the last state it visited before leaving. Because $p(x,\s_r^x)=p(x,\s_r)$ (as $x$ is not absorbing) and the never-ending returns to $\s_r$ ensure that $\varepsilon_r(x,\s_r^x)=1-\varepsilon_r(x,x)$, the above and the arithmetico-geometric progression yield the rightmost equation in \eqref{eq:fbya87n873awf}:
\begin{align*}\Ebx{\bar{T}_1}&=\frac{1}{q(x)}\sum_{n=1}^\infty n(\Pbx{A_n}+\Pbx{B_n})\\
&=\frac{p(x,\s_r)+p(x,\s_r^c)\varepsilon_r(x,\s_r^x)}{q(x)}\sum_{n=1}^\infty n(p(x,\s_r^c)\varepsilon_r(x,x))^{n-1}\\
&=\frac{1-p(x,\s_r^c)+p(x,\s_r^c)(1-\varepsilon_r(x,x))}{q(x)(1-p(x,\s_r^c)\varepsilon_r(x,x))^2}=\frac{1}{q(x)-q_o(x)\varepsilon_r(x,x)}=\frac{1}{q^{\varepsilon_r}(x)}.\end{align*}

\section{Relegated proofs and examples for Section~\ref{fiveschemes}}\label{app:c}\subsection{Convergence of the LDQBDP scheme}\label{app:qbdconv}

Suppose that the {rate matrix is regular}, that the chain is irreducible and has a stationary distribution $\pi$, that the approximations $R^1_r,R^2_r,\dots$ of $R^1,R^2,\dots$ satisfy
\begin{equation}\label{eq:rmatconv}R^l_1(x,y)\leq R^l_2(x,y)\leq \dots\leq R^l(x,y),\quad\lim_{r\to\infty}R^{l}_r(x,y)=R^l(x,y),\quad\forall x,y\in\L_l,\end{equation}
for all $l$ in $\zp$, and that the truncations $\s_r$ approach the entire state space as $r$ tends to infinity (i.e.\  $L_r\to\infty$ as $r\to\infty$). The approximations $\pi_r$ of the stationary distribution obtained {by solving \eqref{eq:ndua93a1}--\eqref{eq:ndua93a2} and applying~\eqref{eq:ndua93a3}} (with $R^{1}_r$ replacing $R^{1}$, $\Gamma^l_r:=R^{1}_r\dots R^l_r$ replacing $\Gamma^l$, and $\pi_r(\cdot)$ replacing $\pi(\cdot|\s_r)$) converge to $\pi$ in total variation as $r$ tends to infinity.

\begin{proof}To simplify the notation, we denote the approximation associated with $\s_r$ by $\pi^r$ instead of $\pi_r$ throughout this proof. Suppose we are able to show that the restriction $\pi^r_{|0}=(\pi^r(x))_{x\in\L_0}$ to the zeroth level of the approximation $\pi^r$ converges pointwise to the restriction $\pi_{|0}$ of $\pi$:
\begin{equation}\label{eq:fn7a9ewn73ahjnfw4}\lim_{r\to\infty} \pi^r_{|0}(x)=\pi_{|0}(x)\quad\forall x\in\s^0.\end{equation}
Due to \eqref{eq:rmatconv},
$$\lim_{r\to\infty}\Gamma^l_r=\lim_{r\to\infty}R^1_rR^2_r\dots R^l_r=R^1R^2\dots R^l=\Gamma^l\quad\forall l\in\zp$$
pointwise, and it would follow from \eqref{eq:fn7a9ewn73ahjnfw4} that
$$\lim_{r\to\infty}\pi^r_{|l}=\lim_{r\to\infty}\pi^r_{|0}\Gamma^l_r=\pi_{|0}\Gamma^l=\pi_{|l}$$
pointwise, for each $l$ in $\zp$. Because the state space is the union of the levels ($\s=\cup_{l=0}^\infty\L_l$), combining the above with \eqref{eq:fn7a9ewn73ahjnfw4} proves that $\pi^r$ converges to $\pi$ pointwise. The convergence in total variation then follows as a sequence of probability distributions converges pointwise to a limit that is a probability distribution if and only if it converges in total variation (c.f.\ end of Section~\ref{sec:convergence}).

All that remains to be shown is \eqref{eq:fn7a9ewn73ahjnfw4}. To do so, note that the monotone convergence theorem and \eqref{eq:rmatconv} imply that 
$$\lim_{r\to\infty}\sum_{l=0}^{L_r-1}\sum_{x'\in\L_l}\Gamma^l_r(x,x')=\sum_{l=0}^{\infty}\sum_{x'\in\L_l}\Gamma^l(x,x')\quad\forall x\in\L_0.$$
For this reason, $A_r$ converges pointwise to $A$  as $r$ tends to infinity, where $A_r$ and $A$ are the $|\L_0|\times(|\L_0|+1)$ dimensional matrices defined by
\begin{align*}A_r&:=\begin{bmatrix}Q^0+R^1_rQ^1_-&{\vert}& \sum_{l=0}^{L_r-1} \sum_{x\in\L_l}\Gamma^l_r(\cdot,x)\end{bmatrix},\\ A&:=\begin{bmatrix}Q^0+R^1Q^1_-&{\vert}& \sum_{l=0}^\infty \sum_{x\in\L_l}\Gamma^l(\cdot,x)\end{bmatrix}.\end{align*}
Multiplying {\eqref{eq:ndua93a1}--\eqref{eq:ndua93a2}} through by $\pi(\s_r)$ {and taking the limit $r\to\infty$, we find that  $\pi_{|0}$ satisfies the equations $\rho_0 A=[0\enskip1]=:b$. But, as shown in~\cite{Bright1995}, $\rho=(\rho_l)_{l=0}^\infty$, with $\rho_l:=\rho_0\Gamma^l$ for all $l>0$, is a probability distribution and satisfies the stationary equations $\rho Q=0$ whenever $\rho_0$ is a non-negative vector satisfying $\rho_0 A=b$. Because we are assuming that $Q$ is irreducible and regular, Theorem~\ref{Qstateq} and Corollary~\ref{iredcor} imply that there exists only such $\rho$ (namely the stationary distribution $\pi$). It follows that $\pi_{|0}$ must be the only non-negative solution to $\rho_0 A=b$ and, consequently, that $A$ has full row rank.}  For these reasons, $\pi_{|0}$ can be expressed in terms of the Moore-Penrose pseudoinverse of $A$: $\pi_{|0}=bA^T(AA^T)^{-1}$ where $A^T$ denotes the transpose of $A$. Because the singular values of a matrix are continuous functions of its entries (as the roots of a polynomial are continuous functions of its coefficients) and $A_r$ converges to $A$, it follows that $A_r$ has full row rank for all sufficiently large $r$. For all such $r$, we have that $\pi_{|0}^r=bA^T_r(A_rA^T_r)^{-1}$ and the convergence of $\pi_{|0}^r$ to $\pi_{|0}$ follows from that of $A_r$ to $A$.
\end{proof}

\subsection{Loss of uniqueness for the TA scheme}\label{lossofuni}

Consider a chain with state space $\{1,2,3\}$ that waits a unit mean exponential amount of time and then jumps down a state until it reaches $1$ where it remains forever:
$$Q=\begin{bmatrix}
0&0&0\\1&-1&0\\0&1&-1
\end{bmatrix}.$$
In particular, $Q$ is $\varphi$-irreducible and $(1_1(x))_{x\in\{1,2,3\}}$ is the unique stationary distribution. If we pick the truncation $\{1,3\}$ and set the re-entry to always occur at $3$, then both $1$ and $3$ are absorbing states for the modified chain and every probability distribution on $\{1,3\}$ is a stationary distribution of this chain. 
Similarly, if the re-entry location depends on the pre-exit location, then the modified chain may not be $\varphi$-irreducible even if original chain is irreducible. For instance, consider again the truncation $\{1,3\}$ of $\{1,2,3\}$, but this time with
$$Q=\begin{bmatrix}
-1&1&0\\1&-2&1\\0&1&-1
\end{bmatrix}.$$
If the re-entry matrix is the identity matrix, then $1$ and $3$ are once again absorbing states for the modified chain.

\subsection{The conditional re-entry matrix and its truncation-based approximations}\label{app:condre}

Assuming that $X$ is irreducible, the conditional re-entry matrix $\varepsilon_r$ in \eqref{eq:exre0} is \cite{Dayar2011,Meyer1989,Zhao1996} given by
{
\begin{equation}
\label{eq:exre}
\varepsilon_r(x,y)=
\begin{cases}
\sum_{z\not\in\s_r} \dfrac{q(x,z)}{q_{o}(x)} \sum_{n=0}^\infty\left(\sum_{z'\not\in\s_r} p^n_{|\s_r^c}(z,z')p(z',y)\right)& \text{if }x\in\cal{B}_o(\s_r) \\
0&\text{if }x\not\in\cal{B}_o(\s_r)
\end{cases}
\end{equation}
for all $x,y$ in $\cal{S}_r$, where $q_o$ and $\cal{B}_o(\s_r)$ denote the out-rate and out-boundary \eqref{eq:qo}},  $(p(x,y))_{x,y\in\s}$ denotes the one-step matrix \eqref{eq:jumpmat} of the embedded discrete-time chain, and $(p^n_{|\s_r^c}(x,y))_{x,y\not\in\s_r}$ denotes the $n^{th}$ matrix power of the restriction  $(p_{|\s_r^c}(x,y))_{x,y\in\s_r^c}$ of $(p(x,y))_{x,y\in\s}$ to the truncation's complement $\s_r^c$:
$$p_{|\s_r^c}(x,y):=\left\{\begin{array}{cl}\frac{q(x,y)}{q(x)}&\text{if }q(x)>0\\0&\text{if }q(x)=0\end{array}\right.\qquad\forall x,y\not\in\s_r.$$
While a rigorous proof of \eqref{eq:exre} requires the use of the strong Markov property and lies beyond the scope of this review, the argument goes as follows:
\begin{enumerate}
\item $p^n_{|\s_r^c}(z,z')p(z',y)$ is the probability that $X$ first returns to the truncation after $n$ jumps and by jumping from $z'$ to $y$, conditioned on the event that $z$ was the first state outside of the truncation it visited. 
\item For this reason, $\sum_{z'\not\in\s_r}p^n_{|\s_r^c}(z,z')p(z',y)$ is the probability that $X$ first returns to the truncation after $n$ jumps and by jumping to $y$, conditioned on the event that $z$ was the first state outside of the truncation it visited. 
\item Thus, $\sum_{n=0}^\infty \sum_{z'\not\in\s_r}p^n_{|\s_r^c}(z,z')p(z',y)$ is the probability that $X$ returns to the truncation by jumping to $y$, conditioned on the event that $z$ was the first state outside of the truncation it visited.
\item $\frac{q(x,z)}{q_{o}(x)}$ is the probability that $z$ is the first state that $X$ visits after leaving the truncation, conditioned on the event that the last state it visited before leaving was $x$.
\item Thus, the right-hand side of \eqref{eq:exre} is the probability that $X$ returns to the truncation the truncation by jumping to $y$, conditioned on the event that the last state it visited before leaving was $x$. 
\end{enumerate}

By truncating the sums in \eqref{eq:exre}, we obtain lower bounds $l^{\varepsilon_r}$ on $\varepsilon_r$. Augmenting these bounds so that the rows sum indexed by states $x$ in $\cal{B}_o(\s_r)$ to one, we obtain re-entry matrices that converge to $\varepsilon_r$ as fewer and fewer terms are truncated from \eqref{eq:exre}. Of course, computing these re-entry matrices in practice comes at a further computational expense.

\subsection{The convergence factor as an indicator of error}\label{app:cfprob}

Ostensibly, the convergence factor $F_r$ in \eqref{eq:convergencefactor} incorporates global information regarding the chain's behaviour as its definition features a Lyapunov function $v$ satisfying the inequality \eqref{eq:lyacom} that holds over the entire state. However, $F_r$'s definition only involves the values that $v$ takes within the truncation and does not  utilise any of the constants appearing in the inequality (e.g.\ compare with the error bound in \eqref{eq:l1boundsfsp}). 
As we show below, by tweaking the constants featuring in the inequality \eqref{eq:lyacom}, we may set $v$ to be any desired constant $k\geq1$ on $\s_r$ without invalidating the premise of the necessary criterion (Theorem~\ref{lyacom}), for any chain with rate matrix of the sort \eqref{eq:qmatrixsrn}. Thus, {the convergence factor is not a useful error bound for a fixed truncation $\s_r$ because 
$$v(z)+\max_{x\in\cal{B}_o(\s_r)}v(x)$$
in \eqref{eq:convergencefactor} may be set to any number no smaller than two.}

Suppose that $v,d_1,d_2$ satisfy the premise of Theorem~\ref{lyacom}.  Let $\s_r$ be any finite truncation of the state space $\s$, $k\geq1$, and 
$$\tilde{v}(x):=\left\{\begin{array}{ll}k&\text{if }x\in\s_r\\v(x)&\text{if }x\not\in\s_r\end{array}\right.\quad\forall x\in\s.$$
By definition, $\tilde{v}$ is norm-like and takes values no smaller than one.
Moreover, for states $x$ inside of the truncation $\s_r$, we have that
\begin{align*}Q\tilde{v}(x)&\leq\left(\max_{x\in\s_r}Q\tilde{v}(x)\right)+k-\tilde{v}(x).\end{align*}
On the other hand, for states $x$ outside of the truncation, we have that
\begin{align*}Q\tilde{v}(x)&=\sum_{x'\in\s_r}q(x,x')\tilde{v}(x')+\sum_{x'\not\in\s_r}q(x,x')\tilde{v}(x')=k\sum_{x'\in\s_r}q(x,x')+\sum_{x'\not\in\s_r}q(x,x')v(x')\\
&\leq k\sum_{x'\in\s_r}q(x,x')+Qv(x) \leq k\sum_{x'\in\s_r}q(x,x')+d_2-d_1v(x)\\
&=k\sum_{x'\in\s_r}q(x,x')+d_2-d_1\tilde{v}(x).\end{align*}
If the rate matrix is of the type in \eqref{eq:qmatrixsrn}, then the set
$$\left\{x\not\in\s_r:\sum_{x'\in\s_r}q(x,x')>0\right\}$$
is finite and setting 
$$\tilde{d}_1:=\min\{d_1,1\},\quad\tilde{d}_2:=\max\left\{\left(\max_{x\in\s_r}Q\tilde{v}(x)\right)+k,d_2+k\left(\max_{x\not\in\s_r}\sum_{x'\in\s_r}q(x,x')\right)\right\},$$
we obtain $(\tilde{v},\tilde{d}_1,\tilde{d_2})$ satisfying the premise of Theorem~\ref{lyacom} with $\tilde{v}(x)=k, \, \forall x \in \s_r$.

{\subsection{Tightening the bound~\eqref{eq:l1boundsfsp} using linear programming}\label{app:lptighten}
Once $v$ and $F$ satisfying~\eqref{eq:lyabia} for some $f\geq1$ and $d$ have been found and $\beta$ has been chosen for some $r$ satisfying $F\subseteq\s_r$ and $\overline{\phi}^\beta_r>0$, the bound~\eqref{eq:l1boundsfsp} can be tightened by solving the linear program:
\begin{align}
\label{eq:fne78wa9hf7yae}
c_r:=\inf\left\{1_{F_\circ}(z_r)w(z_r)+\frac{e}{\beta \overline{\phi}^\beta_r}: \hspace{-5pt}
\begin{array}
{l}(w(x))_{x\in F_\circ}\in\r^{\mmag{F_\circ}}, \, e\in\r,\, w(x)\geq0, 
\, \forall x\in F_\circ, \\
\sum_{y\in \s}q(x,y)[1_{F_\circ}(y)w(y)+1_{F_\circ^c}(y)v(y)]\leq e-1, \, x\in F
\end{array}
\hspace{-5pt}
\right\},
\end{align}
where $F_\circ:=\{x \in F: q(x,y)=0\enskip \forall y\not\in F\}$ denotes the set of states in $F$ from which the chain cannot leave $F$ in a single jump and $F_\circ^c:=\s\backslash F_\circ$ denotes its complement. Because  $((v(x))_{x\in F_\circ},d)$ satisfies the constraints of the linear program, its optimal value  $c_r$ is at most $1_{F_\circ}(z_r)v(z_r)+d/(\beta\overline{\phi}_r^\beta)$ and we have that
$$1_{F_\circ}(z_r)c_r+1_{F_\circ^c}(z_r)[v(z_r)+c_r]\leq v(z_r)+\frac{d}{\beta \overline{\phi}^\beta_r}.$$
Note that $F_\circ$ is the set of states in $F$ that do not feature in the stationary equations indexed by states outside $F$. Hence 
$$\tilde{v}(x):=\left\{\begin{array}{ll}w(x)&\text{if }x\in F_\circ\\ v(x)&\text{if }x\not\in F_\circ\end{array}\right.\enskip\forall x\in\s,\enskip \tilde{f}:=e-Q\tilde{v},\quad \tilde{d}:=e, \quad \tilde{F}:=F,$$
satisfy the criterion in Theorem~\ref{lyabia}  if $(w,e)$ is any feasible point of the linear program~\eqref{eq:fne78wa9hf7yae}. For these reasons,  replacing $v,f,d,F$ in~\eqref{eq:l1boundsfsp} with $\tilde{v},\tilde{f},\tilde{d},\tilde{F}$ for any optimal point $(w,e)$ achieving $c_r$, we obtain the refined bound:
\begin{equation}\label{eq:refbounds}\norm{\pi-\pi_r^{z_r}}_{TV}\leq (1_{F_\circ}(z_r)c_r+1_{F_\circ^c}(z_r)[v(z_r)+c_r])O_r\leq \left(v(z_r)+\frac{d}{\beta \overline{\phi}^\beta_r}\right)O_r.\end{equation}
Of course, obtaining this bound comes at a further computational expense, as its calculation requires solving a linear program with $\mmag{F_\circ}+1$ variables and a comparable number of constraints. However, this extra cost should be at worst comparable to that incurred by the matrix inversion in~\eqref{eq:phir}.
}

\subsection{A proof of \eqref{eq:condconv}}\label{app:condconvproof}

Note that
$$B:=\frac{1}{\max_{x\in\s_r}q(x)}Q^{\varepsilon_r}_r+I\geq \frac{1}{\max_{x\in\s_r}q(x)}Q_r+I\geq0,$$
where $Q_r^{\varepsilon_r}$ denotes the stochastic complement of $Q$ (c.f.\ \eqref{eq:truncandaug}), $Q_r$ the truncated rate matrix $(q(x,y))_{x,y\in\s_r}$,  and $I$ the identity matrix $(1_x(y))_{x,y\in\s_r}$. Irreducibility of $Q$ implies irreducibility of $B$. Moreover, it is straightforward to verify that $B$ has spectral radius of one and that its Perron-Frobenius right-eigenvector is the vector of ones. Because we may rewrite \eqref{eq:stateqmod1} (with $e_r:=\varepsilon_r$) as $\pi_r^{\varepsilon_r}B=\pi_r^{\varepsilon_r}$ and because the unique probability distribution solving these equations is the conditional distribution (c.f.\ Section~\ref{sec:opt}), \cite[Theorem~3]{Courtois1984} shows that the conditional distribution belongs to the convex hull of the normalised rows of $Q_r^{-1}$ (irreducibility of the state space ensures that $Q_r$ is invertible) indexed by states belonging to the in-boundary \eqref{eq:insidebound}.

Replacing $e_r$ with $(1_z(y))_{x,y\in\s_r}$ in \eqref{eq:truncandaug2}, we find that $\pi_r^z$ satisfies \eqref{eq:stateqmod1} if and only if
\begin{equation}\label{eq:truncatedinverse}\sum_{x\in\s_r}\pi_r^z(x)q(x,y)=-\left(\sum_{x\in\s_r}\pi_r^z(x)q_o(x)\right)1_z(y)\quad\forall y\in\s_r.\end{equation}
Because the state space is irreducible, $\sum_{x\in\s_r}\pi_r^z(x)q_o(x)>0$  and we may rewrite the above as $(\alpha \pi_r^z) Q_r=1_z$ for some constant $\alpha\neq0$. In other words, the normalised rows of $Q_r^{-1}$ are the TA approximations $(\pi^z_r)_{z\in\s_r}$ and \eqref{eq:condconv} follows.

\subsection{A proof of \eqref{eq:boundsitaf}}\label{app:boundsitafproof}

Equation \eqref{eq:condconv} implies that
\begin{equation}\label{eq:fneau8n387naw}\pi(\s_r)\left(\min_{z\in\cal{B}_i(\s_r)}\pi^z_r(f)\right)\leq \pi_{|r}(f)\leq \pi(\s_r)\left(\max_{z\in\cal{B}_i(\s_r)}\pi^z_r(f)\right).\end{equation}
If 
$$\min_{z\in\cal{B}_i(\s_r)}\pi^z_r(f)\geq0,$$
then
\begin{align*}\pi(\s_r)\left(\min_{z\in\cal{B}_i(\s_r)}\pi^z_r(f)\right)&\geq \left(1-\frac{c}{r}\right)\left(\min_{z\in\cal{B}_i(\s_r)}\pi^z_r(f)\right),\\
\min_{z\in\cal{B}_i(\s_r)}\pi^z_r(f)&\geq \left(1-\frac{c}{r}\right)\left(\min_{z\in\cal{B}_i(\s_r)}\pi^z_r(f)\right),\end{align*}
for all $r\geq c$. Otherwise,
$$\pi(\s_r)\left(\min_{z\in\cal{B}_i(\s_r)}\pi^z_r(f)\right)\geq \min_{z\in\cal{B}_i(\s_r)}\pi^z_r(f),\enskip \left(1-\frac{c}{r}\right)\left(\min_{z\in\cal{B}_i(\s_r)}\pi^z_r(f)\right)\geq \min_{z\in\cal{B}_i(\s_r)}\pi^z_r(f),$$
for all $r\geq c$. In either case, the leftmost inequality in \eqref{eq:boundsitaf} follows from \eqref{eq:fneau8n387naw}. The argument for the rightmost inequality in \eqref{eq:boundsitaf} is entirely analogous and we skip it.

\subsection{Strengthening pointwise convergence of $l_r$ into total variation convergence}\label{app:itaconvtrick}
{Suppose that
$$l_r(x)\leq \pi(x),\quad\lim_{r\to\infty}l_r(x)=\pi(x),\quad\forall x\in\s.$$
Because the total variation norm of an unsigned measure is its mass,
\begin{align*}\norm{l_r-\pi}_{TV}&=\sum_{x\in\s}(\pi(x)-l_r(x))\leq \sum_{x\in\s_{r'}}(\pi(x)-l_r(x))+\sum_{x\not\in\s_{r'}}\pi(x)\\
&=\sum_{x\in\s_{r'}}(\pi(x)-l_r(x))+m_{r'}\quad\forall r,r'\in\zp,\end{align*}
where $(\s_r)_{r\in\zp}$ denotes any sequence of increasing finite truncations that approach the state space (i.e.\ $\cup_{r=1}^\infty\s_r=1$)). Fix any $\varepsilon>0$ and pick an $r'$ large enough that $m_{r'}\leq \varepsilon/2$. Because $\mathcal{S}_{r'}$ is finite, the  pointwise convergence implies that the leftmost sum is smaller than $\varepsilon/2$ for all large enough $r$. As $\varepsilon>0$ was arbitrary, it follows that $\norm{l_r-\pi}_{TV}\to0$ as $r\to\infty$.}

\subsection{The conditional distribution belongs to $\cal{P}_r$ in \eqref{eq:br}}\label{app:condouter}

By the definition in \eqref{eq:nr} of $\cal{N}_r$, any equation $\pi Q(x)=0$ with $x$ in $\cal{N}_r$ only involves entries of $\pi$ indexed by $x'$s inside the truncation. For this reason,
\begin{align*}\sum_{x'\in\s}\pi(x'|\s_r)q(x',x)&=\sum_{x'\in\s_r}\pi(x'|\s_r)q(x',x)=\sum_{x'\in\s_r}\frac{\pi(x')q(x',x)}{\pi(\s_r)}\\
&=\sum_{x'\in\s}\frac{\pi(x')q(x',x)}{\pi(\s_r)}=0\quad\forall x\in\cal{N}_r,\end{align*}
and we have that the conditional distribution satisfies the first constraint in \eqref{eq:br}. That it satisfies the second, third, and fifth constraints in \eqref{eq:br} follows directly from its definition in \eqref{eq:picond}. To show that it satisfies the fourth constraint, note that
$$
\sum_{x\not\in\s_r}\frac{\pi(x)}{\pi(\s_r^c)}w(x)\geq \sum_{x\not\in\s_r}\frac{\pi(x)}{\pi(\s_r^c)}r=r=\sum_{x\in\s_r}\frac{\pi(x)}{\pi(\s_r)}r>\sum_{x\in\s_r}\frac{\pi(x)}{\pi(\s_r)}w(x).$$
Thus,
\begin{align*}
\pi(w|\s_r)&=\sum_{x\in\s_r}\frac{\pi(x)}{\pi(\s_r)}w(x)=\pi(\s_r)\sum_{x\in\s_r}\frac{\pi(x)}{\pi(\s_r)}w(x)+\pi(\s_r^c)\sum_{x\in\s_r}\frac{\pi(x)}{\pi(\s_r)}w(x)\\
&<\sum_{x\in\s_r}\pi(x)w(x)+\sum_{x\not\in\s_r}\pi(x)w(x)\leq c,
\end{align*}
completing the proof.

\end{document}